\documentclass{amsart}

\usepackage{amsmath}
\usepackage{amsthm}
\usepackage{hyperref}
\usepackage{amsfonts,graphics,amsthm,amsfonts,amscd,latexsym}
\usepackage{epsfig}
\usepackage{flafter}
\usepackage{mathtools}
\usepackage{comment}
\usepackage{stmaryrd}

\usepackage{mathabx,epsfig}

\hypersetup{
    colorlinks=true,    
    linkcolor=blue,          
    citecolor=blue,      
    filecolor=blue,      
    urlcolor=blue           
}
\usepackage{tikz}
\usetikzlibrary{graphs,positioning,arrows,shapes.misc,decorations.pathmorphing}

\tikzset{
    >=stealth,
    every picture/.style={thick},
    graphs/every graph/.style={empty nodes},
}

\tikzstyle{vertex}=[
    draw,
    circle,
    fill=black,
    inner sep=1pt,
    minimum width=5pt,
]
\usepackage[position=top]{subfig}
\usepackage{amssymb}
\usepackage{color}

\calclayout

\usetikzlibrary{decorations.pathmorphing}
\tikzstyle{printersafe}=[decoration={snake,amplitude=0pt}]

\newcommand{\rank}{\operatorname{rank}}

\newcommand{\supp}{\operatorname{supp}}

\newcommand{\vol}{\operatorname{vol}}

\newcommand{\pp}{\mathbb{P}}

\newcommand{\qq}{\mathbb{Q}}
\newcommand{\zz}{\mathbb{Z}}

\newcommand{\rr}{\mathbb{R}}

\newcommand{\kk}{\mathbb{K}}

\def\O#1.{\mathcal {O}_{#1}}			
\def\pr #1.{\mathbb P^{#1}}				
\def\af #1.{\mathbb A^{#1}}			
\def\ses#1.#2.#3.{0\to #1\to #2\to #3 \to 0}	
\def\xrar#1.{\xrightarrow{#1}}			
\def\K#1.{K_{#1}}						
\def\bA#1.{\mathbf{A}_{#1}}			
\def\bM#1.{\mathbf{M}_{#1}}				
\def\bL#1.{\mathbf{L}_{#1}}				
\def\bB#1.{\mathbf{B}_{#1}}				
\def\bK#1.{\mathbf{K}_{#1}}			
\def\subs#1.{_{#1}}					
\def\sups#1.{^{#1}}

\usepackage{tikz}
\usetikzlibrary{matrix,arrows,decorations.pathmorphing}

\newtheorem{introdef}{Definition}

  \newtheorem{theorem}{Theorem}[section]

  \newtheorem{lemma}[theorem]{Lemma}
  \newtheorem{proposition}[theorem]{Proposition}
  \newtheorem{corollary}[theorem]{Corollary}

  \newtheorem{notation}[theorem]{Notation}
  \newtheorem{definition}[theorem]{Definition}
  \newtheorem{example}[theorem]{Example}

  \newtheorem{question}[theorem]{Question}

\newtheorem{remark}[theorem]{Remark}

\theoremstyle{remark}

\numberwithin{equation}{section}

\usepackage[all]{xy}

\begin{document}

\title[Birational complexity and conic fibrations]{Birational complexity and conic fibrations}

\author[J.~Moraga]{Joaqu\'in Moraga}
\address{UCLA Mathematics Department, Box 951555, Los Angeles, CA 90095-1555, USA
}
\email{jmoraga@math.ucla.edu}

\subjclass[2020]{Primary 14E05, 14E30;
Secondary 14D06.}

\thanks{The author was partially supported by NSF research grant DMS-2443425.}

\begin{abstract}
Let $(X,B)$ be a log Calabi--Yau pair of 
dimension $n$, index one, and birational complexity $c$.
We show that $(X,B)$ has a crepant birational model that admits a tower of Mori fiber spaces 
of which at least $n-c$ are conic fibrations.
The main tool to prove this statement is a geometric version
of the Jordan property for finite actions on rationally connected varieties.
Motivated by the proof of the previous statement, we introduce new measures of the complexity of a log Calabi--Yau pair; the {\em alteration complexity} and the {\em conic complexity}.
We characterize when these invariants are zero.
Finally, we give applications 
of the tools of the main theorem to
birational superrigidity,
Fano hypersurfaces,
dual complexes,
Weil indices of Fano varieties,
and klt singularities.
\end{abstract}

\maketitle

\setcounter{tocdepth}{1} 
\tableofcontents

\section{Introduction}

The complexity $c(X,B)$ of a log Calabi--Yau pair $(X,B)$
is defined to be $\dim X + \rank {\rm Cl}(X)_\qq - |B|$,
where
$|B|$ stands for the sum of the coefficients of $B$.
The complexity gives an intrinsic measure of how far a given log Calabi--Yau pair is from being toric~\cite{BMSZ18}.
In~\cite{MS21}, Svaldi and the author 
characterized toric singularities
using a local variant of the complexity.
In~\cite{GM23}, Gongyo and the author 
introduced a variant of the complexity for generalized pairs. 
Using this variant, they proposed a conjecture 
that unifies the complexity,
Mukai's conjecture,
and the Kobayashi--Ochiai Theorem. The aforementioned results lead to strong geometric consequences when the considered complexity measure is strictly smaller than one. However, not much is known when the complexity is larger than or equal to one. In this article, we study the birational geometry of log Calabi--Yau pairs $(X,B)$
from the perspective of birational complexity.
The birational complexity $c_{\rm bir}(X,B)$ is a birational variant of the complexity (see Definition~\ref{def:bir-comp}).
It was introduced by Mauri and the author as a proxy to study dual complexes~\cite{MM24}.
In this work, we consider log Calabi--Yau pairs $(X,B)$ of index one
and birational complexity $c$.
We prove that some crepant birational model of $(X,B)$ admits a tower of Mori fiber spaces of which at least $\dim X - c$ are conic fibrations. 
In particular, this result
sheds light on the birational geometry of $(X,B)$ when the birational complexity is smaller than the dimension.

\subsection{Finite actions on rationally connected varieties} 
The main motivation of this article is to study how the birational complexity of a log Calabi--Yau pair reflects in its (birational) geometry. However, most of the statements that are proved in the article are reduced to problems about finite actions on rationally connected varieties. The Jordan property for finite actions on rationally connected varieties states the following. Let $n$ be a positive integer, if $G$ is a finite group acting faithfully on a $n$-dimensional rationally connected variety, then there exists $A\leqslant G$ a normal abelian subgroup of rank at most $n$ and index at most $c(n)$. This theorem was proved by Prokhorov and Shramov~\cite{PS14} by using the boundedness of Fano varieties proved by Birkar~\cite{Bir21}.
In other words, finite groups acting on rationally connected varieties are almost abelian, with respect to the dimension of the variety. 
In previous works, the author has studied how the existence of such a group $A$ reflects on the geometry of the rationally connected variety.
In~\cite{Mor20a}, it is proved that if $A=\zz/m\zz \leqslant X$ for $X$ a Fano type surface with $m$ large enough, then in an $A$-equivariant birational model $X'$ of $X$, we have $A<\mathbb{G}_m<{\rm Aut}(X')$. 
In~\cite{Mor20c}, it is proved that if $A=(\zz/m\zz)^n \leqslant X$ for $X$ a Fano type variety of dimension $n$ and $m$ large enough (compared with $n$), then $X$ is rational and it is indeed a compactification of $\mathbb{G}_m^n$.
In~\cite{Mor21}, it is proved that if $A=(\zz/m\zz)^k\leqslant X$ for $X$ a Fano type variety of dimension $n$ and $m$ large enough (compared with $n$), then the coregularity of $X$ is at most $n-k$.
The following theorem gives a more geometric explanation of the Jordan property for finite groups acting on rationally connected varieties.

\begin{theorem}[Geometric Jordan property for finite actions on rationally connected varieties]\label{thm:finite-actions}
Let $n$ be a positive integer. There exists a constant $c(n)$, only depending on $n$, satisfying the following. Let $G$ be a finite group acting on a $n$-dimensional rationally connected variety $X$. Then, there exists a normal abelian subgroup $A\leqslant G$ of rank at most $n$ and index at most $c(n)$ satisfying the following. 
There exists an $A$-equivariant birational map $X\dashrightarrow X'$
and an $A$-equivariant fibration $X'\rightarrow \mathbb{P}^k$ that induces an action of $A$ on $\mathbb{P}^k$ and this action factors through $\mathbb{G}_m^k$.
\end{theorem} 

In other words, in some birational model, the group $A$ comes from a torus action on the base of a fibration. It is unclear under which conditions the torus action of $\pp^k$ can be lifted to a torus action of $X'$. 

\subsection{Birational complexity and conic fibrations} 
The following theorem is the second main result of this article.

\begin{theorem}\label{theorem:cbir-conic}{\rm(c.f. Theorem~\ref{thm:cbir-conic})}
Let $(X,B)$ be a log Calabi--Yau pair of dimension $n$ and index one. 
Let $c:=c_{\rm bir}(X,B)$.
Then $(X,B)$ has a crepant birational model 
that admits a tower of Mori fiber spaces of which at least $n-c$ are conic fibrations.
\end{theorem}

Example~\ref{ex:not-conic-bundle} shows that even if $c_{\rm bir}(X,B)< \dim X$, the variety $X$ may not admit a birational conic fibration. 
As a consequence of Theorem~\ref{theorem:cbir-conic}, we obtain the following corollary.

\begin{corollary}\label{cor:birsuper}
Let $X$ be a birationally superrigid Fano variety of dimension $n$.
Let $(X,B)$ be a log Calabi--Yau pair 
of index one and coregularity zero.
Then, we have $c_{\rm bir}(X,B)=n$.
\end{corollary} 

In particular, a Fano variety $X$ of Picard rank one that admits a $1$-complement with at least two prime components is not birationally superrigid (see Corollary~\ref{cor:bsuper}).
As a second application of Theorem~\ref{theorem:cbir-conic}, we get the following statement regarding smooth Fano hypersurfaces.

\begin{corollary}\label{cor:hyperf}
Let $X_d$ be a smooth hypersurface of degree $d$ in $\mathbb{P}^n$.
If $d\leq n$, then $X_d$ has a birational model
that admits a tower of Mori fiber spaces
of which at least $n-d$ are conic fibrations.
\end{corollary}

The previous statement is known to experts and can be proved using projections from points and the Minimal Model Program (MMP).
However, we emphasize it in the introduction to draw a parallel between the birational complexity and the degree of Fano hypersurfaces:
log Calabi--Yau pair $(X,B)$ of dimension $n$, index one, and birational complexity $c$ 
behave similarly to smooth Fano hypersurfaces of degree $c+1$ in $\pp^{n+1}$.

\subsection{Alteration complexity} In the proof of Theorem~\ref{theorem:cbir-conic}, we will need to study the complexity of finite covers of $(X,B)$.
This motivates us to study the complexity of a log Calabi--Yau pair up to finite Galois covers.
The {\em alteration complexity} $c_{\rm alt}(X,B)$ of a log Calabi--Yau pair $(X,B)$ is the infimum of the birational complexities of sub-log Calabi--Yau pairs obtained by finite Galois covers of $(X,B)$ (see Definition~\ref{def:alt-comp}). 
The alteration complexity is non-negative (see Theorem~\ref{theorem:alt-comp-noneg}).
Furthermore, the alteration complexity
is indeed computed by a log Calabi--Yau pair (see Theorem~\ref{theorem:alt-comp-computation}).
The following theorem characterizes when the alteration complexity is zero.

\begin{theorem}\label{theorem:alt-comp-zero}
Let $(X,B)$ be a log Calabi--Yau pair of dimension $n$ and index one. 
Then, we have that $c_{\rm alt}(X,B)=0$ if and only if there exist:
\begin{enumerate}
\item[(i)] a log Calabi--Yau pair $(X',B')$ which is crepant birational equivalent to $(X,B)$, and 
\item[(ii)] a log Calabi--Yau pair $(Y,B_Y)$ 
with a crepant finite Galois morphism $(Y,B_Y)\rightarrow (X',B')$, 
\end{enumerate}
such that $(Y,B_Y)$ is crepant birational equivalent to 
$(\pp^n,H_0+\dots + H_n)$.
\end{theorem} 

In other words, an $n$-dimensional log Calabi--Yau pair $(X,B)$ with index one has alteration complexity zero
if and only if
it is crepant birational equivalent to $(Y,\Delta)/G$ where 
$G\leqslant {\rm Bir}(\pp^n,H_0+\dots+H_n)$
is a finite group
and $(Y,\Delta)$ is a crepant model where $G$ 
regularizes.
In general, both the birational transformation
and the finite Galois cover are necessary to 
obtain the model that computes the alteration complexity (see Example~\ref{ex:alt-comp-two-ways}).

In~\cite{MM24}, Mauri and the author 
construct examples, in each dimension $n\geq 3$, of log Calabi--Yau pairs of
index one, coregularity zero,
and birational complexity $n$.
However, all these examples have alteration complexity zero. Indeed, they are quotients of toric log Calabi--Yau pairs. 
We expect that Corollary~\ref{cor:birsuper}
can be used to provide examples, in each dimension $n\geq 3$, of log Calabi--Yau pairs of index one, coregularity zero, and
alteration complexity $n$ (see Question~\ref{quest:max-alt-comp}).
Log Calabi--Yau pairs
of maximal alteration complexity
should be considered as building blocks
of more general log Calabi--Yau pairs
(see Theorem~\ref{theorem:coreg-0-decomp}).

\subsection{Conic complexity} 
The conic complexity $k(X,B)$ of a log Calabi--Yau pair $(X,B)$ measures the maximum number of strict conic fibrations that we can find in a sequence of fibrations starting from $X$. 
A {\em strict conic fibration}
is a conic fibration that respects the log Calabi--Yau structure (see Definition~\ref{def:strict-conic}).
The conic complexity $k(X,B)$ is zero precisely when we can find a sequence of
$\dim X$ strict conic fibrations starting from $X$.
Therefore, a log Calabi--Yau pair $(X,B)$ with $k(X,B)=0$ behaves like a Bott tower.
However, $X$ may be singular. In Definition~\ref{def:Q-bott}, we define a singular version of Bott towers, the so-called {\em Bott $\qq$-towers}. 
In general, a Bott $\qq$-tower is not necessarily a Bott tower (see Remark~\ref{rem:Q-bott}).
The following theorem states that log Calabi--Yau pairs with conic complexity zero are precisely Bott $\qq$-towers.

\begin{theorem}\label{theorem:conic-complexity-zero}
Let $(X,B)$ be a log Calabi--Yau pair.
Then, we have that $k(X,B)=0$ if and only if $X$ is a Bott $\mathbb{Q}$-tower and $B$ is the torus invariant divisor.
\end{theorem}

The {\em birational conic complexity} $k_{\rm bir}(X,B)$
and the {\em alteration conic complexity} 
$k_{\rm alt}(X,B)$ of a log Calabi--Yau pair $(X,B)$ are defined in a similar fashion as the birational complexity
and alteration complexity (see Definition~\ref{def:alt-comp}).
The birational conic complexity is the minimum among conic complexities of crepant birational equivalent models of the log Calabi--Yau pair.
The alteration conic complexity is the minimum among birational conic complexities
of crepant finite Galois covers.
Theorem~\ref{theorem:cbir-conic} can rephrased as $c_{\rm bir}(X,B)\geq k_{\rm bir}(X,B)$ for log Calabi--Yau pairs of index one. Further, the equality $c_{\rm alt}(X,B)\geq k_{\rm alt}(X,B)$ also holds (see Corollary~\ref{corollary:total-comparison}).
The following theorem characterizes
when $k_{\rm bir}(X,B)$ and $k_{\rm alt}(X,B)$ are zero.

\begin{theorem}\label{introcor:bir-alt-conic-comp-zero}
Let $(X,B)$ be a log Calabi--Yau pair of dimension $n$.
Then, the following statements hold:
\begin{enumerate}
\item[(i)] $k_{\rm bir}(X,B)=0$ if and only if 
$(X,B)$ is crepant birational equivalent to $(\pp^n,H_0+\dots+H_n)$.
\item[(ii)] $k_{\rm alt}(X,B)=0$ if and only if 
$(X,B)$ is crepant birational equivalent to $(\pp^n,H_0+\dots+H_n)/G$ where
$G\leqslant {\rm Bir}(\pp^n,H_0+\dots+H_n)$ is a finite subgroup.
\end{enumerate}
\end{theorem}

\subsection{Applications} In this subsection, we explain some applications of the tools 
used in the proof of the main theorem.
For the sake of briefness of the introduction, we only explain the results here and state the theorems before their proofs in Section~\ref{sec:proofs}.

\begin{enumerate}
\item In Theorem~\ref{theorem:k-alt-vs-reg}, we prove that for a log Calabi--Yau pair $(X,B)$ of index one, 
the dimension of its dual complex $\mathcal{D}(X,B)$ is at most $n-k_{\rm alt}(X,B)-1$. 
In other words, the coregularity of $(X,B)$ is bounded above by the alteration conic complexity $k_{\rm alt}(X,B)$.
\item In Theorem~\ref{theorem:from-pi-to-fun}, we study the fundamental group of $U:=X\setminus B$ where $X$ is a Fano variety and $(X,B)$ is a $1$-complement\footnote{A {\em $1$-complement} is a boundary $B$ for which $(X,B)$ is log Calabi--Yau and $K_X+B\sim 0$.}.
We show that $\pi_1^{\rm alg}(U)$ admits a normal abelian subgroup
of index at most $c(n)$
and rank at most $n-k_{\rm alt}(X,B)$.
\item In Theorem~\ref{theorem:bound-coreg-0-3-fold}, we consider log Calabai--Yau pairs $(X,B)$ of dimension $3$, index one, and coregularity zero. If their birational complexity is less than $3$, then we prove that either $c_{\rm alt}(X,B)=0$ or $X$ admits 
a birational del Pezzo fibration.
\item In Theorem~\ref{theorem:weil-index}, we study the Weil index of $n$-dimensional klt Fano varieties $X$ admitting a $1$-complement $(X,B)$. We show that if the index is large enough, compared with $n$, then the inequality  $k_{\rm alt}(X,B)< n$ holds.
\item In Theorem~\ref{theorem:klt-sing},
we consider $n$-dimensional klt singularities $(X;x)$ with $1$-complements $(X,B;x)$.
We prove that the birational complexity of log canonical places of $(X,B;x)$ is controlled by the local complexity of $(X,B;x)$.
\end{enumerate} 

\subsection*{Acknowledgements} 
The author would like to thanks
Adrien Dubouloz, 
Stefano Filipazzi,
Alvaro Liendo,
Mirko Mauri,
Hendrik S\"u{\ss}, 
Burt Totaro, and 
Charles Vial
for many useful discussions 
related to this article.

\section{Preliminaries}

We work over an uncountable algebraically closed field $\kk$ of characteristic zero.
All considered varieties are defined over $\kk$ unless otherwise stated.
The {\em rank} of a finite abelian group $G$ is the least number of generators of $G$.
A {\em contraction} $f\colon X\rightarrow Z$ is a proper morphism for which $f_*\mathcal{O}_X=\mathcal{O}_Z$.
A {\em fibration} 
$f\colon X\rightarrow Z$ is a contraction with positive dimensional general fiber.

\begin{definition}
{\em 
Let $f\colon X\rightarrow Z$ be a fibration.
We say that $f$ has {\em minimal relative dimension}
if for every birational contraction $X\dashrightarrow X'$ over $Z$
every fibration $X'\rightarrow Z'$ over $Z$ 
satisfies that
the morphism $Z'\rightarrow Z$ is birational.
In other words, every fibration $X'\rightarrow Z'$ over $Z$ has relative dimension
larger than or equal to 
the relative dimension of $X\rightarrow Z$.
For instance, a fibration of relative dimension $1$ is always a fibration
of minimal relative dimension.
}
\end{definition}

\begin{notation}
{\em
Let $(X,B)$ be a log pair.
We write ${\rm Aut}(X,B)$ 
for the subgroup of ${\rm Aut}(X)$ that consists of elements $g$ for which $g^*B=B$. 
Let $(X,B)$ be a log pair and $Z_1,\dots,Z_r$ be subvarieties of $X$.
We write ${\rm Aut}(X,B,Z_1,\dots,Z_r)$ for the subgroup of ${\rm Aut}(X,B)$ that consists of elements $g$ for which $g^*Z_i=Z_i$ for every $i \in \{1,\dots,r\}$.
Let $f\colon X\rightarrow W$ be a fibration and $(X,B)$ be a log pair.
We write ${\rm Aut}_W(X,B)$ for the subgroup of 
${\rm Aut}(X,B)$ that consists of elements $g$ making the following diagram commutative
\[
\xymatrix{
(X,B)\ar[d]^-{f} \ar[r]^-{g} & (X,B)\ar[d]^-{f} \\ 
W \ar[r]^-{{\rm id}_W} & W
}
\]
In other words, the group ${\rm Aut}_W(X,B)$ consists of automorphisms that respect the fibration
and induce the identity on the base.
}
\end{notation}

\subsection{Equivariant generalized pairs}
In this subsection, we recall the concept of equivariant generalized pairs. 

\begin{definition}
{\em 
Let $G$ be a finite group.
A {\em $G$-equivariant generalized sub-pair} $(X,B,\mathbf{M})$ consists of the data:
\begin{enumerate}
\item a normal quasi-projective variety $X$ with $G\leqslant {\rm Aut}(X)$,
\item a $G$-invariant divisor $B$, and 
\item a $G$-invariant b-nef divisor $\mathbf{M}$ on $X$, 
\end{enumerate} 
such that $K_X+B+\mathbf{M}_X$ is $\rr$-Cartier. We say that $(X,B,\mathbf{M})$ is a {\em $G$-equivariant generalized pair} if $B$ is an effective divisor.
By~\cite[Proposition 2.4]{Mor21} the $G$-equivariant b-nef divisor $\mathbf{M}$ is the pull-back of a b-nef divisor $\mathbf{N}$ on the quotient $X/G$.
If $Y$ is a birational model of $X/G$ on which $\mathbf{N}$ descends, 
then we say that $Y$ is a {\em quotient by $G$ where $\mathbf{M}$ descends}.
Indeed, $Y$ is the quotient by $G$ of a $G$-equivariant birational model of $X$ (see~\cite[Proposition 2.4]{Mor21}).
In the previous setting, we write $G\leqslant {\rm Aut}(X,B,\mathbf{M})$.

If $G$ is the trivial group, then we drop it from the notation.
If $\mathbf{M}$ is the trivial b-nef divisor, then we say that $(X,B)$ is a sub-pair and drop $\mathbf{M}$ from the notation.
}
\end{definition} 

\begin{definition}
{\em 
Let $(X,B,\mathbf{M})$ be a $G$-equivariant generalized sub-pair.
Let $\Omega$ be a set of real numbers.
We say that {\em the coefficients of $B$ belong to} $\Omega$ if we can write
$B=\sum_{i=1}^k \lambda_i B_i$, 
where the $B_i$'s are irreducible reduced divisors and each $\lambda_i$ is in $\Omega$.
We say that the coefficients of of $\mathbf{M}$ {\em belong to $\Omega$ in a quotient by $G$ where it descends} if the following conditions are satisfied:
\begin{enumerate}
\item the b-nef divisor $\mathbf{M}$ is the pull-back of a b-nef divisor $\mathbf{N}$ on $X/G$, 
\item the b-nef divisor $\mathbf{N}$ descends on $Y$, and 
\item we can write $\mathbf{N}_Y\sim \sum_i \lambda_i \mathbf{N}_{Y,i}$ 
where each $\mathbf{N}_{Y,i}$ is a nef Cartier divisor
and each $\lambda_i\in \Omega$.
\end{enumerate}
}
\end{definition}

\begin{definition}
{\em 
Let $(X,B,\mathbf{M})$ be a generalized pair. 
Let $\pi\colon Y\rightarrow X$ be a projective birational morphism from a normal variety.
Write
\[
\pi^*(K_X+B+\mathbf{M}_X)=K_Y+B_Y+\mathbf{M}_Y.
\]
Let $E\subset Y$ be a prime divisor.
The {\em generalized log discrepancy} of $(X,B,\mathbf{M})$ at $E$ is defined to be:
\[
a_E(X,B,\mathbf{M})=1-{\rm coeff}_E(B_Y).
\]
We say that $(X,B,\mathbf{M})$ is {\em generalized terminal} if all the generalized log discrepancies of exceptional divisors are larger than one.
We say that $(X,B,\mathbf{M})$ is {\em generalized canonical} 
if all its generalized log discrepancies are at least one.
We say that $(X,B,\mathbf{M})$ is {\em generalized Kawamata log terminal} (or gklt) if all its generalized log discrepancies are positive.
We say that $(X,B,\mathbf{M})$ is {\em generalized log canonical} if all its generalized log discrepancies are non-negative.

A {\em generalized non-terminal place} of $(X,B,\mathbf{M})$ is an exceptional prime divisor $E$ over $X$ for which $a_E(X,B,\mathbf{M})\leq 1$.
A {\em generalized non-canonical place} of $(X,B,\mathbf{M})$ is a prime divisor $E$ over $X$ for which $a_E(X,B,\mathbf{M})<1$.
A {\em generalized non-klt place} (or {\em non-gklt place}) of $(X,B,\mathbf{M})$ is a prime divisor $E$ over $X$ for which $a_E(X,B,\mathbf{M})\leq 0$. 
A {\em generalized non-klt center} of $(X,B,\mathbf{M})$ is the image on $X$ of a non-gklt place.
If $(X,B,\mathbf{M})$ is generalized log canonical and $E$ is a non-gklt place,
then we say that $E$ is a {\em generalized log canonical place} (or glc place).
A {\em generalized log canonical center} (or glcc) of a glc pair $(X,B,\mathbf{M})$ is the image on $X$ is a glc place. 
In the case that $(X,B,\mathbf{M})$ is a generalized sub-pair, we adopt the previous definitions and add the word {\em sub}.
}
\end{definition}

\begin{definition}
{\em 
Let $(X,B,\mathbf{M})$ be a generalized lc pair. 
We say that $(X,B,\mathbf{M})$ is {\em generalized divisorially log terminal} (or {\em gdlt}) if there exists an open set $U\subseteq X$ for which the
following conditions are satisfied:
\begin{enumerate}
\item the coefficients of $B$ are less than or equal to one,
\item $U$ is smooth and $B|_U$ has simple normal crossing support, and
\item all generalized non-klt centers of $(X,B,\mathbf{M})$ map into $U$ and consists of strata of $\lfloor B\rfloor$.
\end{enumerate} 
A generalized pair $(X,B,\mathbf{M})$ is said to be {\em generalized purely log terminal} (or gplt) if it is generalized dlt and all its glc centers are divisorial.
} 
\end{definition}

\begin{definition}
{\em 
Let $(X,B,\mathbf{M})$ be a generalized sub-pair. 
Let $p\colon Y\rightarrow X$ be a projective birational morphism from a normal variety.
The generalized sub-pair $(Y,B_Y,\mathbf{M})$ for which
\[
K_Y+B_Y+\mathbf{M}_Y=p^*(K_X+B+\mathbf{M}_X)
\]
is called the {\em log pull-back} of $(X,B,\mathbf{M})$ to $Y$.

Let $\pi\colon X\dashrightarrow X'$ be a birational map. 
Let $(X,B,\mathbf{M})$ and $(X',B',\mathbf{M})$ be two generalized sub-pairs.
Let $p\colon Z \rightarrow X$ 
and $q\colon Z \rightarrow X'$ be a resolution of indeterminacy of the birational morphism. 
If the log pull-back of $(X,B,\mathbf{M})$ to $Z$ agrees with the log pull-back of $(X',B',\mathbf{M})$ to $Z$, then we say that 
$\pi$ is a {\em crepant birational map} for such generalized sub-pairs.
Analogously, we say that $(X',B',\mathbf{M})$ is a {\em crepant birational model} of $(X,B,\mathbf{M})$ or that $X'\dashrightarrow X$ is a {\em crepant birational transformation} for such generalized sub-pairs.

Let $(X,B,\mathbf{M})$ be a generalized sub-pair. 
Let $f\colon Y\rightarrow X$ be a finite Galois morphism and $\mathbf{N}$ be the equivariant b-nef divisor obtained by pulling back $\mathbf{M}$ to $Y$.
The generalized sub-pair $(Y,B_Y,\mathbf{N})$ for which 
\[
K_Y+B_Y+\mathbf{N}_Y=p^*(K_X+B+\mathbf{M}_X)
\]
is called the {\em log pull-back of} $(X,B,\mathbf{M})$ to $Y$.
In the previous setting, we say that 
the finite morphism $f\colon (Y,B_Y,\mathbf{N})\rightarrow (X,B,\mathbf{M})$ is a {\em crepant finite Galois morphism}.
}
\end{definition} 

\begin{definition}
{\em 
Let $(X,B,\mathbf{M})$ be a generalized pair. 
A {\em generalized dlt modification} of $(X,B,\mathbf{M})$ is a projective birational morphism $\phi\colon Y\rightarrow X$ satisfying the following conditions:
\begin{enumerate}
\item the divisors extracted by $\phi$ are generalized log canonical places of $(X,B,\mathbf{M})$, and 
\item the log pull-back $(Y,B_Y,\mathbf{M})$ of $(X,B,\mathbf{M})$ to $Y$ is gdlt.
\end{enumerate} 
}
\end{definition}

The following statement is known as the existence of $G$-equivariant dlt modifications. 

\begin{lemma}\label{lem:g-equiv-dlt}
Let $(X,B,\mathbf{M})$ be a generalized pair and $G\leqslant {\rm Aut}(X,B,\mathbf{M})$ be a finite group.
Then $(X,B,\mathbf{M})$ admits a $G$-equivariant dlt modification.
\end{lemma}

The following lemma characterizes generalized pairs where the boundary divisor plus the moduli divisor is torsion.

\begin{lemma}\label{lem:b+m-tor}
Let $(X,B,\mathbf{M})$ be a generalized pair.
Let $X\dashrightarrow X'$ be a birational contraction
and $B'$ be the push-forward of $B$ on $X'$.
If $B'+\mathbf{M}_{X'}$ is a torsion divisor, then $B'=0$ and $\mathbf{M}$ is torsion in any model where it descends.
\end{lemma} 

\begin{proof}
If $B'\neq 0$, then we can find an ample curve\footnote{This means that the curve $C$ is the intersection of general ample divisors.} $C$ on $X'$ for which $(B'+\mathbf{M}_{X'})\cdot C>0$. This leads to a contradiction. 
Thus, we have $B'=0$.

Let $p\colon Y\rightarrow X'$ be a model where $\mathbf{M}$ descends.
This means that $\mathbf{M}_Y$ is a nef divisor.
By the negativity lemma, we have
$\mathbf{M}_Y+E=p^*\mathbf{M}_{X'}$ where $E$ is an effective divisor.
If $E\neq 0$, then we can find an ample curve $C_0$ on $Y$ for which $E\cdot C_0>0$ so 
$(\mathbf{M}_Y+E)\cdot C_0>0$.
If we define $C:=p_*C_0$, then we have $\mathbf{M}_{X'}\cdot C>0$, 
contradicting the fact that $B'+\mathbf{M}_{X'}=\mathbf{M}_{X'}$ is a torsion divisor.
We conclude that $E=0$ so 
$\mathbf{M}_Y$ is a torsion divisor.
This implies that $\mathbf{M}$ is torsion in any model where it descends.
\end{proof} 

\begin{definition}
{\em 
Let $X\rightarrow Z$ be a fibration and $(X,B,\mathbf{M})$ be a generalized pair.
The {\em generalized log general fiber}, denoted by $(X_z,B_z,\mathbf{M}_z)$, where $z\in Z$ is a general closed point, 
is the generalized pair structure induced on a general fiber $X_z$ by restricting both $B$ and $\mathbf{M}$ to $X_z$.
}
\end{definition}

\subsection{Fano type varieties and log Calabi--Yau pairs}
In this subsection, we recall the concepts of Fano type variety and
of log Calabi--Yau pair.

\begin{definition}
{\em 
Let $X\rightarrow W$ be a contraction.
A generalized sub-pair $(X,B,\mathbf{M})$ is said to be {\em generalized sub-log Calabi--Yau pair} over $W$ if $(X,B,\mathbf{M})$ is generalized sub-log canonical and $K_X+B+\mathbf{M}_X\sim_{\rr,W} 0$.
If $(X,B,\mathbf{M})$ is indeed a generalized pair, then we say that it is a {\em generalized log Calabi--Yau pair} over $W$.
If $(X,B,\mathbf{M})$ is a generalized log Calabi--Yau with $K_X+B+\mathbf{M}_X\sim_{\qq,W} 0$, then 
the {\em index} of $(X,B,\mathbf{M})$ is the smallest positive integer $I$ for which 
$I(K_X+B+\mathbf{M}_X)\sim_W 0$.
Furthermore, if $\mathbf{M}$ is the trivial b-nef divisor, then we say that $(X,B)$ is a log Calabi--Yau pair over $W$.

Let $X\rightarrow W$ be a contraction.
We say that $X$ is {\em of Fano type} over $W$ (or simply, {\em Fano type} over $W$)
if there exists a boundary $B$ on $X$ 
that is big over $W$ 
and $(X,B)$ is a klt Calabi--Yau pair over $W$. If $X\rightarrow W$ is a Fano type morphism, then $X$ is a relative Mori dream space over $W$ (see~\cite[Theorem 3.18]{BM24}).
In particular, the MMP for any divisor on $X$ over $W$ terminates either with a good minimal model over $W$ or a Mori fiber space over $W$.
}
\end{definition}

\begin{definition}
{\em 
Let $X\rightarrow W$ be a fibration and $(X,B,\mathbf{M})$ be a generalized pair.
We say that $(X,B,\mathbf{M})$ is {\em $\qq$-complemented} over $W$ 
if there exists a boundary divisor $\Gamma \geq B$
for which $(X,\Gamma,\mathbf{M})$ is log Calabi--Yau over $W$.
}
\end{definition}

We turn to recall the definition of 
the dual complex $\mathcal{D}(X,B,\mathbf{M})$ of a generalized log Calabi--Yau pair $(X,B,\mathbf{M})$.

\begin{introdef}
{\em 
Let $Y$ be a smooth variety
and let $B$ be a reduced divisor
with simple normal crossing support on $Y$. Write $B=\sum_{i\in I}B_i$ where each $B_i$ is irreducible and reduced.
The dual complex $\mathcal{D}(B)$
is the CW complex 
whose vertices
are in correspondence
with components of $B$
and whose $k$-cells
are in correspondence
with the irreducible components
of $\bigcap_{i\in J} B_i$ with $|J|=k+1$.

Let $(X,B,\mathbf{M})$ be a generalized
pair. 
Let $p\colon Y\rightarrow X$ be a log resolution on which $\mathbf{M}$ descends.
Let $(Y,B_Y,\mathbf{M})$
be the log pull-back of $(X,B,\mathbf{M})$ to $Y$.
We write $B_Y^{\geq 1}$ for the divisor which is the reduced sum of the components of $B_Y$ with coefficient at least one.
The {\em dual complex}
of $(X,B,\mathbf{M})$
with respect to 
the resolution $p$ is
the CW complex $\mathcal{D}(B^{\geq 1})$.
}
\end{introdef}

The dual complex, a priori,
depends on the choice
of a resolution.
However,~\cite[Lemma 2.34]{FS20}
shows that choosing a different resolution
only changes
the dual complex
up to simple homotopy equivalence.

\begin{proposition}\label{prop:sim-hom-equiv}
Let $(X,B,\mathbf{M})$ be a generalized pair. 
Let $p_i\colon Y_i\rightarrow X$, with $i\in \{1,2\}$,
be two log resolutions of 
$(X,B,\mathbf{M})$ on which the b-nef divisor $\mathbf{M}$ descends.
Let $(Y_i,B_{Y_i},\mathbf{M})$ be the log pull-back of $(X,B,\mathbf{M})$ to $Y_i$.
Then, the CW complex $\mathcal{D}(B_{Y_1}^{\geq 1})$ is simple homotopy equivalent to the CW complex $\mathcal{D}(B_{Y_2}^{\geq 2})$.
\end{proposition}

Thus, given a generalized pair
$(X,B,\mathbf{M})$, we may write
$\mathcal{D}(X,B,\mathbf{M})$
for any dual complex of a log resolution
on which the b-nef divisor descends.
Proposition~\ref{prop:sim-hom-equiv} implies that the CW complex $\mathcal{D}(X,B,\mathbf{M})$ is well-defined up to simple homotopy equivalence.

We turn to define the coregularity of a generalized pair. This concept is useful for the study of Fano varieties (see, e.g.,~\cite{FMP22,Mor24b}).

\begin{definition}\label{def:coreg}
{\em 
Let $(X,B,\mathbf{M})$ be a log Calabi--Yau pair of dimension $n$. 
The {\rm regularity} of $(X,B,\mathbf{M})$, denoted by ${\rm reg}(X,B,\mathbf{M})$, is defined to be 
$\dim \mathcal{D}(X,B,\mathbf{M})$.
The {\em coregulartiy} of the generalized log Calabi--Yau pair $(X,B,\mathbf{M})$ is defined to be:
\[
{\rm coreg}(X,B,\mathbf{M})=n-{\rm reg}(X,B,\mathbf{M})-1.
\]
The coregularity of a generalized log Calabi--Yau pair agrees with the dimension of any minimal log canonical center in a gdlt modification.
The coregularity of a $n$-dimensional generalized log Calabi--Yau pair is an integer in the set $\{0,\dots,n\}$.

Let $(X,B,\mathbf{M})$ be a generalized pair.
The {\rm coregularity} of $(X,B,\mathbf{M})$ is defined to be:
\begin{equation}\label{eq-def:coreg} 
{\rm coreg}(X,B,\mathbf{M}):=\min \{ 
{\rm coreg}(X,\Gamma,\mathbf{M}) \mid 
\text{ $(X,\Gamma,\mathbf{M})$ is generalized log Calabi--Yau and $\Gamma \geq B$}
\}.
\end{equation} 
If $(X,B,\mathbf{M})$ is generalized log Calabi--Yau, then the set in~\eqref{eq-def:coreg} consists of a single element, so the two previous definitions agree. In the case that the set in~\eqref{eq-def:coreg} is empty, we set ${\rm coreg}(X,B,\mathbf{M})=\infty.$
}
\end{definition}

The following lemma is well-known.
It gives a characterization of Fano type morphisms.

\begin{lemma}\label{lem:FT-big-nef} 
Let $X\rightarrow W$ be a contraction.
The variety $X$ is of Fano type over $W$ if and only there exists a boundary $B$
such that $(X,B)$ is klt and $-(K_X+B)$ is big and nef over $W$.
\end{lemma} 

The following lemma explains how the Fano type condition behaves under different kinds of morphisms.

\begin{lemma}\label{lem:FT-under-morphisms}
Let $X$ be a Fano type variety over $W$
and $(X,B,\mathbf{M})$ be a generalized log Calabi--Yau pair over $W$.
Then, the following statements hold:
\begin{enumerate}
\item If $X\dashrightarrow X'$ is a birational contraction over $W$, then $X'$ is Fano type over $W$.
\item If $X\rightarrow X'$ is a fibration over $W$, then $X'$ is Fano type over $W$.
\item If $Y\dashrightarrow X$ is a projective birational map over $W$
and all the divisors extracted by $Y\dashrightarrow X$ are non-canonical centers of $(X,B,\mathbf{M})$, then $Y$ is of Fano type over $W$.
\item If $X\rightarrow Y$ is a finite Galois morphism over $W$, then $Y$ is of Fano type over $W$.
\item If $Y\rightarrow X$ is a finite morphism over $W$ and its divisorial branch locus is contained in $\lfloor B\rfloor$, then $Y$ is of Fano type over $W$.
\end{enumerate}
\end{lemma}

\begin{proof}
By~\cite[Lemma 3.7]{MS21}, we may assume that the b-nef divisor $\mathbf{M}$ is trivial.

We prove the first statement.
Let $(X,\Delta)$ be a klt log Calabi--Yau pair over $W$ with $\Delta$ big over $W$.
If $\Delta_Y$ is the push-forward of $\Delta$ on $Y$, then the pair $(Y,\Delta_Y)$ is klt and log Calabi--Yau over $W$. Furthermore, $\Delta_Y$ is big over $W$, so $Y$ is of Fano type over $W$. 

The second statement is proved in~\cite[Lemma 2.12]{Bir19}.

We prove the third statement.
Let $(X,\Delta)$ be a klt pair for which $-(K_X+\Delta)$ is big and nef over $W$ provided by Lemma~\ref{lem:FT-big-nef}.
Replacing $\Delta$ with $(1-\epsilon)B+\epsilon \Delta$ for $\epsilon>0$ small enough, we may assume that $Y\dashrightarrow X$ only extracts non-canonical places of $(X,\Delta)$.
By~\cite[Theorem 1]{Mor20}, there exists a projective birational morphism $Z\rightarrow X$ extracting exactly the exceptional divisors of $Y\dashrightarrow X$.
Let $(Z,\Delta_Z)$ be the log pull-back of $(X,\Delta)$ to $Z$.
The pair $(Z,\Delta_Z)$ is klt and $-(K_Z+\Delta_Z)$ is big and nef over $W$.
Hence, $Z$ is of Fano type over $W$ by Lemma~\ref{lem:FT-big-nef}.
As $Y\dashrightarrow Z$ is small, we have that $Y\rightarrow W$ is of Fano type.

The fourth statement is proved in~\cite[Lemma 3.17]{Mor20c}.

We show the last statement. Let
$Y\rightarrow X$ be a finite morphism 
such that its divisorial branch locus is contained in $\lfloor B\rfloor$.
Let $Z\rightarrow X$ be the Galois closure of $Y\rightarrow X$.
Hence, the divisorial branch locus of $Z\rightarrow X$ is contained in $\lfloor B\rfloor$.
By Lemma~\ref{lem:FT-big-nef}, there is a klt pair $(X,\Delta)$ for which $-(K_X+\Delta)$ is big and nef over $W$.
For each prime divisor $P$ on $X$, we let $m_P$ be the ramification index of $Z\rightarrow X$ at $P$.
We replace $\Delta$ with $(1-\epsilon)B+\epsilon \Delta$ for $\epsilon>0$ small enough and
assume that ${\rm coeff}_P(\Delta)>1-\frac{1}{m_P}$ for every $P$.
Hence, the log pull-back $(Z,\Delta_Z)$ of $(X,\Delta)$ to $Z$ is a log pair, i.e., $\Delta_Z$ is an effective divisor.
Hence, the pair $(Z,\Delta_Z)$ is klt
and $-(K_Z+\Delta_Z)$ is big and nef over $W$.
We conclude that $Z$ is of Fano type over $W$ by Lemma~\ref{lem:FT-big-nef}.
Hence, $Y$ is of Fano type over $W$ by the fourth statement.
\end{proof}

\begin{lemma}\label{lem:general-element}
Let $N$ be a positive integer.
Let $(X,B,\mathbf{M})$ be a generalized sub-log Calabi--Yau pair.
Assume that $N(K_X+B+\mathbf{M}_X)\sim 0$.
Let $X'\rightarrow X$ be a projective birational morphism for which $N\mathbf{M}_{X'}$ is a base point free divisor.
Let $\Gamma' \in |N\mathbf{M}_{X'}|$ be a general element and $\Gamma$ its push-forward to $X$. Then, the following conditions are satisfied: 
\begin{enumerate} 
\item we have $N(K_X+B+\Gamma/N)\sim 0$, 
\item the sub-pair $(X,B+\Gamma/N)$ is sub-log canonical, and
\item every log canonical place of $(X,B+\Gamma/N)$ is a generalized log canonical place of $(X,B,\mathbf{M})$.
\end{enumerate} 
\end{lemma}

\begin{proof}
Let $Y\rightarrow X'\rightarrow X$ be a log resolution of $(X,B,\mathbf{M})$ on which $\mathbf{M}$ descends as a b-nef divisor.
The morphism $Y\rightarrow X'$ is also a log resolution of 
$(X',B')$.
Let $(X',B',\mathbf{M})$ be the log pull-back of $(X,B,\mathbf{M})$ to $X'$.
It suffices to show that:
\begin{enumerate} 
\item[(i)] we have  $N(K_{X'}+B'+\Gamma'/N)\sim 0$, 
\item[(ii)] the sub-pair $(X',B'+\Gamma'/N)$ is sub-log canonical, and 
\item[(iii)] every log canonical place of $(X',B'+\Gamma'/N)$ is a generalized log canonical place of $(X,B,\mathbf{M})$.
\end{enumerate} 
Indeed, the projective birational morphism $X'\rightarrow X$ is crepant birational equivalent for the pairs $(X',B'+\Gamma'/N)$ and 
$(X,B+\Gamma/N)$. 

The relation $N(K_{X'}+B'+\Gamma'/N)\sim 0$ holds as $\Gamma'\in |N\mathbf{M}_{X'}|$. This proves condition (i).
As we chose $\Gamma'$ general, we 
may assume that $q\colon Y\rightarrow X'$ is a 
log resolution for $(X',B'+\Gamma'/N)$.
In particular, the pull-back of $\Gamma'$ to $Y$ is just its strict transform.
Hence, we have 
\[
q^*(K_{X'}+B'+\Gamma'/N) = K_Y+B_Y + {q^{-1}}_*\Gamma'/N. 
\]
The coefficients of the divisor $B_Y$ are bounded above by one as 
$(X',B')$ is sub-log canonical.
Since the coefficients of $\Gamma'/N$ are bounded above by one, we conclude that $(X',B'+\Gamma'/N)$
is sub-log canonical. This proves (ii).
If $a_F(X',B'+\Gamma'/N)=0$, then $F$ is obtained by a toric blow-up of a strata of $\lfloor B_Y\rfloor$.
In particular, we have  $a_F(X',B')=0$, so $a_F(X',B',\mathbf{M})=0$.
This proves the last condition.
\end{proof}

\subsection{Complexity measures
for log Calabi--Yau pairs}
In this subsection, we introduce several measures of the complexity of log Calabi--Yau pairs.

\begin{definition}
\label{def:comp}
{\em 
Let $X\rightarrow W$ be a projective contraction of normal quasi-projective varieties.
Let $(X,B,\mathbf{M})$ be a generalized log Calabi--Yau pair over $W$.
The {\em complexity} of $(X,B,\mathbf{M})$ over $W$ is 
\[
c(X/W,B,\mathbf{M}):= \dim X + \dim_\qq {\rm Cl}(X/W) - |B|, 
\]
where $|B|$ is the sum of the coefficients of $B$.

Let $B$ be an effective divisor on a variety $X$ and $X\rightarrow W$ be a contraction.
A {\em decomposition} $\Sigma$ of $B$ is a finite sum $\sum_{i\in I} b_i B_i \leq B$,
where each $b_i$ is a non-negative real number
and each $B_i$ is an effective Weil divisor.
We define the {\em norm} of the decomposition, denoted by $|\Sigma|$, to be $\sum_{i\in I} b_i$.
We define the {\em span over $W$} of the decomposition $\Sigma$, denoted by $\langle \Sigma /W \rangle$, to be $\langle B_i \rangle \subseteq {\rm Cl}_\qq(X/W)$. 
The {\em relative rank} of the decomposition over $W$, denoted by $\rho(\Sigma/W)$, is $\dim_\qq \langle B\rangle$.
The {\em complexity} of a decomposition $\Sigma$ of a divisor $B$ on $X$ is defined to be
\[
c(\Sigma/W):= \dim X + \rho(\Sigma/W) - |\Sigma|.
\]
The {\em fine complexity} of $(X,B,\mathbf{M})$ over $W$ is
\[
\bar{c}(X/W,B,\mathbf{M}):=\inf 
\{ 
c(\Sigma/W) \mid \text{ 
$\Sigma$ is a decomposition of $B$
}
\}.
\]
In~\cite[Proposition 30]{MS21}, the authors show that the complexity is indeed a minimum, i.e., it is computed by some decomposition of $B$.
}
\end{definition}

\begin{definition}\label{def:bir-comp}
{\em 
Let $X\rightarrow W$ be a projective contraction of normal quasi-projective varieties.
Let $(X,B,\mathbf{M})$ be a generalized sub-log Calabi--Yau pair over $W$.
The {\em birational complexity} of $(X,B,\mathbf{M})$ over $W$ is defined to be
\[
c_{\rm bir}(X/W,B,\mathbf{M}):=
{\rm inf}\left\{ 
c(X'/W,B',\mathbf{M}) \middle| 
\begin{gathered}
\text{ 
$(X',B',\mathbf{M})$ is a crepant birational model} \\
\text{ of $(X,B,\mathbf{M})$ over $W$
and $B'\geq 0$ }
\end{gathered}
\right\}.
\]
If the set in the definition is empty,
then we set the birational complexity to be infinite.
The {\em fine birational complexity}, denoted by $\bar{c}_{\rm bir}(X/W,B,\mathbf{M})$, 
is defined similarly to the
birational complexity 
by replacing the infimum among complexities
with the infimum
among fine complexities.
}
\end{definition}

\begin{definition}\label{def:alt-comp}
{\em 
The {\em alteration complexity} of $(X,B,\mathbf{M})$ over $W$ is defined to be 
\[
c_{\rm alt}(X/W,B,\mathbf{M}):=
{\rm inf} \left\{ c_{\rm bir}(Y/W,B_Y,\mathbf{N}) 
\middle|
\begin{gathered}
\text{
$(Y,B_Y,\mathbf{N})\rightarrow (X,B,\mathbf{M})$ is a crepant finite Galois} \\ 
\text{morphism over $W$ and $Y\rightarrow W$ is a contraction}
\end{gathered}
\right\}. 
\]
In the previous definition, we emphasize that $Y\rightarrow W$ must have connected fibers to avoid dealing with disconnected varieties over ${\rm Spec}(\kk)$. 
The {\em fine alteration complexity},
denoted by $\bar{c}_{\rm alt}(X/W,B,\mathbf{M})$, 
is defined similarly to the alteration complexity by replacing the infimum among complexities
with the infimum among fine complexities.
}
\end{definition} 

\begin{definition}\label{def:strict-conic}
{\em 
Let $(X,B,\mathbf{M})$ be a generalized pair
with $G\leqslant {\rm Aut}(X,B,\mathbf{M})$ a finite group.
Let $\rho\colon X\rightarrow Z$ be a $G$-equivariant fibration.
We say that $\rho\colon X\rightarrow Z$ is a {\em $G$-equivariant strict conic fibration} for $(X,B,\mathbf{M})$ if the following conditions are satisfied:
\begin{enumerate}
\item the morphism $\rho$ is a conic fibration
with $\rho^G(X/Z)=1$,
\item the generalized pair $(X,B,\mathbf{M})$ is log Calabi--Yau over $Z$, 
\item the divisor $\lfloor B\rfloor$ has two $G$-invariant prime components $S_0$ and $S_\infty$ that are horizonal over $Z$, and 
\item the prime divisors $S_0$ and $S_\infty$ are disjoint.
\end{enumerate}
If $G$ is the trivial group, then we just say that $\rho$ is a {\em strict conic fibration}.
}
\end{definition}

\begin{definition} 
{\em 
Let $(X,B,\mathbf{M})$ be a generalized log Calabi--Yau pair.
Let 
\[
\xymatrix{ 
\mathcal{T}: 
X\ar[r]^-{\phi_1} &
X_1\ar[r]^-{\phi_2} & 
\dots\ar[r]^-{\phi_r} & 
X_r 
}
\] 
be a tower of fibrations for $X$.
Let $I(\mathcal{T})\subseteq \{1,\dots,r\}$ be the set of indices for which $\phi_i$ is a strict conic fibration for the generalized log Calabi--Yau pair induced on $X_i$ by the generalized canonical bundle formula.
We denote by $r(\mathcal{T})$ the cardinality of $I(\mathcal{T})$. The {\em conic complexity} of $(X,B,\mathbf{M})$,
denoted by $k(X,B,\mathbf{M})$, is
defined to be:
\[
k(X,B,\mathbf{M}):=\dim X - \max \{
r(\mathcal{T}) \mid 
\text{ 
$\mathcal{T}$ is a tower of fibrations for $X$
}
\}.
\]
}
\end{definition} 

\begin{definition}
{\em 
Let $(X,B,\mathbf{M})$ be a generalized sub-log Calabi--Yau pair. 
The {\em birational conic complexity} of $(X,B,\mathbf{M})$, denoted by 
$k_{\rm bir}(X,B,\mathbf{M})$, is defined to be 
\[
\min \{ k(X',B',\mathbf{M}) \mid 
\text{ 
$(X',B',\mathbf{M})$ is crepant birational equivalent to $(X,B,\mathbf{M})$ and $B'\geq 0$
}
\}.
\]
If the set in the definition is empty, 
we set the birational conic complexity to be infinite.
The {\em alteration conic complexity} of
$(X,B,\mathbf{M})$, denoted by $k_{\rm alt}(X,B,\mathbf{M})$, is defined to be
\[
\min\{ 
k_{\rm bir}(Y,B_Y,\mathbf{N})\mid 
\text{
$(Y,B_Y,\mathbf{N})\rightarrow (X,B,\mathbf{M})$ is a crepant finite Galois morphism
}
\}.
\]
}
\end{definition}

The following definition
is a slight generalization of
the concept of Bott towers.
As explained in Theorem~\ref{theorem:conic-complexity-zero} these are exactly the varieties for which the conic complexity equals zero.

\begin{definition}\label{def:Q-bott}
{\em 
We define Bott $\qq$-towers inductively.
A {\em Bott $\qq$-tower of height $1$} is a rational curve considered as a toric variety, i.e., with a fixed  $\mathbb{G}_m$-action.
A {\em Bott $\qq$-tower of height $n$} 
is a toric projective variety 
equivariantly isomorphic to 
\[
\mathbb{P}_Y(\mathcal{L}_1\oplus \mathcal{L}_2), 
\]
where $Y$ is a Bott $\qq$-tower variety of height $n-1$ and the $\mathcal{L}_i$'s are torus invariant $\qq$-line bundle on $Y$.
}
\end{definition}

\begin{remark}\label{rem:Q-bott}
{\em 
A Bott tower of height $2$ is a Hirzebruch surface.
More generally, every toric surface with a Mori fiber space onto $\pp^1$ is a Bott $\qq$-tower of height $2$. 
}
\end{remark} 

\subsection{Negativity lemma} In this subsection, we prove some applications of the negativity lemma.

\begin{lemma}\label{lem:neg0}
Let $\phi\colon X\rightarrow Y$ be a projective birational morphism from a normal projective variety to a normal $\qq$-factorial projective variety.
Let $E$ be an effective divisor contained in the exceptional locus of $\phi$. Then, we have 
$\supp E \subseteq {\rm Bs}_{-}(E)$.
\end{lemma}

\begin{proof}
Since $Y$ is $\qq$-factorial, 
we can find an effective divisor
$F$ on $X$ that is antiample over $Y$.
Let $A$ be an ample divisor on $Y$.
Then, the divisor $\phi^*(A)-\epsilon F$ is ample on $X$.
Let $\delta>0$ be a positive rational number.
Let $0\leq H \sim_\qq E+\delta \phi^*(A)-\delta \epsilon F$. 
Note that $H-E \sim_\qq \delta \epsilon F - \delta \phi^*(A)$ is anti-ample over $Y$.
Furthermore, $\phi_*(H-E)=\phi_*H \geq 0$.
By~\cite[Lemma 3.39]{KM92}, we conclude that $H\geq E$.
\end{proof} 

\begin{lemma}\label{lem:negativity}
Let $X$ be a normal projective variety.
Let $\pi\colon X\dashrightarrow Y$ be a birational contraction to a $\qq$-factorial variety.
Let $E\subset X$ be an effective $\qq$-Cartier divisor.
Assume that $E$ is contained in the exceptional locus of $\pi$.
Then, we have $\supp E \subseteq {\rm Bs}_{-}(E)$.
\end{lemma}

\begin{proof}
By contradiction, assume that a component $Q$ of $E$ is not contained in ${\rm Bs}_{-}(E)$.
Let $p\colon Z \rightarrow X$ and
$q\colon Z \rightarrow Y$ be a resolution of the birational map $\pi$.
Let $A$ be an ample divisor on $X$.
By assumption, for every $\epsilon>0$, the prime divisor $Q$ is not contained in ${\rm Bs}(E+\epsilon A)$.
Hence, for every $\epsilon>0$, the effective divisor $p^*Q$ is not contained in 
${\rm Bs}(p^*E+\epsilon p^*A)$.
By~\cite[Lemma 1.27]{Mor18a}, 
we have ${\rm Bs}_{-}(p^*E) \subseteq \bigcup_{\epsilon >0} {\rm Bs}(p^*E+\epsilon p^*A)$.
In particular, $p^*Q$ is not contained in ${\rm Bs}_{-}(p^*E)$.
However, since $\pi$ is a birational contraction, the divisor $p^*E$ is exceptional over $Y$.
Hence, by Lemma~\ref{lem:neg0}, we have 
$\supp p^*Q \subseteq \supp p^*E \subseteq {\rm Bs}_{-}(E)$. This leads to a contradiction.
\end{proof} 

\subsection{Generalized canonical bundle formula}
The following statement is known as the generalized canonical bundle formula.
The proof is verbatim from~\cite[Theorem 1.5]{FM20} by replacing~\cite[Theorem 1.4]{Bir12} 
with~\cite[Theorem 1.3]{LX22}.

\begin{proposition}\label{prop:gen-cbf}
Let $n$ be a positive integer
and $\Lambda$ be a set of rational numbers
satisfying the descending chain condition.
Let $X$ be a $n$-dimensional variety
and 
let $f\colon X\rightarrow Z$ be a Fano type fibration.
Let $(X,B,\mathbf{M})$ be a generalized log Calabi--Yau pair.
Assume that the coefficients of $B$ belong to $\Lambda$ and the coefficients of $\mathbf{M}$ belong to $\Lambda$ in a model where it descends.
Then, there exists a positive integer $q$
and a set of rational numbers $\Omega$ 
with the descending chain condition,
both only depending on $n$ and $\Lambda$, satisfying the following.
We can write
\[
K_X+B+\mathbf{M}_X \sim qf^*(K_Z+B_Z+\mathbf{N}_Z)
\] 
where $(Z,B_Z,\mathbf{N})$ is a generalized log Calabi--Yau pair, 
the coefficients of $B_Z$ belong to $\Omega$, and the 
coefficients of $\mathbf{N}$ belong to $\Omega$ in a model
where it descends.
\end{proposition} 

We will use an equivariant version of the previous statement.
The following statement is known as 
the equivariant generalized canonical bundle formula for Fano type morphisms.
The proof is verbatim from the proof of~\cite[Lemma 2.33]{Mor21} by replacing~\cite[Theorem 1.5]{FM20}
with Proposition~\ref{prop:gen-cbf}.

\begin{proposition}\label{prop:equiv-gen-cbf}
Let $n$ be a positive integer,
$\Lambda$ be a set of rational numbers satisfying the descending chain condition,
and $G$ be a finite group.
Let $X$ be a $n$-dimensional variety.
Let $f\colon X \rightarrow Z$ be a $G$-equivariant Fano type fibration.
Let $(X,B,\mathbf{M})$ be a generalized log Calabi--Yau pair with $G\leqslant {\rm Aut}(X,B,\mathbf{M})$.
Assume that the coefficients of $B$ 
belong to $\Lambda$
and the coefficients
of $\mathbf{M}$ belong to $\Lambda$
in a quotient by $G$ where it descends.
Let $G_Z$ be the homomorphic image of $G$ acting on $Z$.
Then, 
there exists a positive integer $q$
and a set of rational numbers $\Omega$
with the descending chain condition, 
both of them only depending on $n$ and $\Lambda$, 
satisfying the following.
We can write
\[
K_X+B+\mathbf{M}_X \sim qf^*(K_Z+B_Z+\mathbf{N}_Z)
\]
where $(Z,B_Z,\mathbf{N})$ is a $G_Z$-equivariant generalized log Calabi--Yau pair, 
the coefficients of $B_Z$ belong to $\Omega$, and the coefficients of $\mathbf{N}$ belong to $\Omega$ in a quotient by $G_Z$ where it descends.
\end{proposition}

\subsection{The two ray game}
The following proposition is an application of the $2$-ray game.
The proof follows from~\cite[Lemma 2.11 and Lemma 2.12]{MM24}.

\begin{proposition}\label{prop:2-ray}
Let $G$ be a finite group.
Let $X\rightarrow W$ be a $G$-equivariant Fano type morphism, let $\pi\colon X\rightarrow Z$ be a $G$-equivariant fibration over $W$, and let $G_Z$ be the homomorphic image of $G$ acting on $Z$.
Let $(X,B,\mathbf{M})$ be a $G$-equivariant generalized log Calabi--Yau pair over $W$.
Let $(Z,B_Z,\mathbf{N})$ be the $G_Z$-equivariant generalized log Calabi--Yau pair induced on $Z$ by the equivariant canonical bundle formula.
Let $\phi_Z \colon Z'\dashrightarrow Z$ be a $G_Z$-equivariant projective birational map from a $\qq$-factorial variety $Z'$ that only extracts generalized log canonical places of $(Z,B_Z,\mathbf{N})$.
Let $(Z',B_{Z'},\mathbf{N})$ be the generalized log Calabi--Yau pair induced on $Z'$.
Then, there exists a commutative diagram:
\[
\xymatrix{
(X,B,\mathbf{M})\ar[d]_-{\pi} & (X',B',\mathbf{M})\ar@{-->}[l]_-{\phi}\ar[d]^-{\pi'}\\
(Z,B_Z,\mathbf{N}) & (Z',B_{Z'},\mathbf{N})\ar@{-->}[l]^-{\phi_Z}
}
\]
satisfying the following conditions:
\begin{enumerate}
\item the birational map $\phi$ is $G$-equivariant, a crepant birational equivalence, and only extracts generalized log canonical places of $(X,B,\mathbf{M})$, 
\item the morphism $\pi'$ is a $G$-equivariant Fano type fibration, 
\item we have ${\pi'}^{-1}(\lfloor B_{Z'}\rfloor) \subseteq \lfloor B'\rfloor$, 
\item if $\pi$ is a $G$-equivariant Mori fiber space, then so is $\pi'$, and 
\item if $\pi$ is a $G$-equivariant strict conic fibration, then so does is $\pi'$.
\end{enumerate}
\end{proposition} 

\begin{proof}
Properties (1)-(4) are already proved in~\cite[Lemma 2.11 and Lemma 2.12]{MM24}.
It suffices to show statement (5).
Assume that $\pi$ is a $G$-equivariant Mori fiber space of relative dimension one
and that the divisor $\lfloor B\rfloor$ has two $G$-invariant disjoint prime components that are horizontal over $Z$. Call $S_0$ and $S_\infty$ these $G$-invariant prime components. 
Consider a commutative diagram as in the statement:
\[
\xymatrix{
(X,B,\mathbf{M})\ar[d]_-{\pi} & (X',B',\mathbf{M})\ar@{-->}[l]_-{\phi}\ar[d]^-{\pi'}\\
(Z,B_Z,\mathbf{N}) & (Z',B_{Z'},\mathbf{N})\ar@{-->}[l]^-{\phi_Z}
}
\]
satisfying properties (1)-(4).
Let $S'_0$ and $S'_\infty$ be the strict transforms of $S_0$ and $S_\infty$ on $X'$, respectively.
By condition (4), we know that $\pi'$ is a $G$-equivariant Mori fiber space of relative dimension one.
It suffices to prove that $S'_0$ and $S'_\infty$ are disjoint.
First, note that every glcc of $(X',B',\mathbf{M})$ which is different from $S'_0$ and $S'_\infty$ is vertical over $Z'$. This holds as $X'\rightarrow Z'$ has relative dimension one, so the only horizontal glc centers of
$(X',B',\mathbf{M})$ are $S'_0$ and $S'_\infty$.
Then, the support of ${\pi'}^*\lfloor B_{Z'}\rfloor$ contains all the vertical generalized log canonical centers of $(X',B',\mathbf{M})$.
By condition (3), the divisor
$B'-\epsilon {\pi'}^*\lfloor B_{Z'}\rfloor$ is effective for $\epsilon>0$ small enough.
In particular, 
$(X',B'-\epsilon {\pi'}^*\lfloor B_{Z'}\rfloor,\mathbf{M})$
is a generalized pair for $\epsilon>0$ small enough.
The only generalized log canonical centers of $(X',B'-\epsilon {\pi'}^*\lfloor B_{Z'}\rfloor,\mathbf{M})$ are $S'_0$ and $S'_\infty$. 
Hence, the divisors $S'_0$ and $S'_\infty$ are disjoint.
Indeed, the intersection of glc centers
is union of glc centers.
This finishes the proof of (5).
\end{proof}

\subsection{Finite equivariant geometry} 
In this section, we prove some results regarding $G$-equivariant geometry.
The first result is proved in~\cite[Theorem 5.1]{Mor21}

\begin{theorem}\label{thm:action-dual-comp}
Let $n$ be a positive integer.
There exists a constant $c(n)$,
only depending on $n$,
satisfying the following.
Let $(X,B,\mathbf{M})$ be a generalized log Calabi--Yau pair of dimension $n$.
If $G\leqslant {\rm Aut}(X,B,\mathbf{M})$ is a finite group, then $G$ admits a normal subgroup $A\leqslant G$ of index at most $c(n)$ that acts trivially on $\mathcal{D}(X,B,\mathbf{M})$.
\end{theorem} 

The following lemma allows us to perform equivariant extraction of generalized log canonical places.

\begin{lemma}\label{lem:g-extraction}
Let $(X,B,\mathbf{M})$ be a generalized log canonical pair and $G\leqslant {\rm Aut}(X,B,\mathbf{M})$ be a finite group.
Assume that $X$ is a klt variety.
Let $\mathcal{V}$ be a set of $G$-invariant generalized log canonical places of $(X,B,\mathbf{M})$.
Then, there exists a $G$-equivariant projective birational morphism $Y\rightarrow X$ that extracts precisely the divisors corresponding to the elements of $\mathcal{V}$.
\end{lemma} 

\begin{proof}
The generalized pair $(X,(1-\epsilon)B,(1-\epsilon)\mathbf{M})$ is $G$-equivariant and klt.
If $\epsilon>0$ is small enough, then 
the valuations in $\mathcal{V}$ are generalized non-canonical places of $(X,(1-\epsilon)B,(1-\epsilon)B)$.
Hence, we may find a $G$-invariant effective divisor $\Gamma\sim_\qq (1-\epsilon)B+(1-\epsilon)\mathbf{M}_X$ such that $(X,\Gamma)$ is klt and every valuation
in $\mathcal{V}$ is a non-canonical valuation of $(X,\Gamma)$ (see, e.g.,~\cite[Lemma 3.7]{MS21}).
Thus, the statement follows from~\cite[Theorem 1]{Mor20c}.
\end{proof}

\begin{lemma}\label{lem:extending-fibration}
Let $G$ be a finite group.
Let $X\rightarrow Z$ be a $G$-equivariant Fano type fibration and $(X,B)$ be a $G$-equivariant log Calabi--Yau pair.
Let $G_Z$ be the homomorphic image of $G$ acting on $Z$.
Let $U$ be an open $G_Z$-invariant subset of $U$ and $X_U$ its preimage on $X$.
Let $B_U$ be the restriction of $B$ to $X_U$.
Assume that there exists a $G$-equivariant crepant birational model $(Y_U,B_{Y_U})$ of $(X_U,B_U)$ over $U$ that admits a $G$-equivariant fibration $Y_U\rightarrow W_U$ over $U$ of relative dimension $d$.
Furthermore, assume that $Y_U\dashrightarrow X_U$ only extracts log canonical places of $(X_U,B_U)$.
Then, there exists a $G$-equivariant crepant birational model of $(X,B)$ over $Z$ that admits a $G$-equivariant fibration over $Z$ of relative dimension $d$.
\end{lemma}

\begin{proof}
By Lemma~\ref{lem:g-extraction}, we can find a $G$-equivariant crepant birational morphism 
$(X',B') \rightarrow (X,B)$ 
such that $X'_U$ is $G$-equivariantly small birational to $Y_U$.
In the previous sentence, $X'_U$ is the preimage of $U$ in $X'$.
Let $A_{W_U}$ be a $G$-invariant ample divisor on $W_U$ and $A_{Y_U}$ be its preimage on $Y_U$.
Then, $A_{Y_U}$ is a semiample $G$-invariant divisor
of relative Iitaka dimension equal to $\dim W_U$ over $U$.
Let $A_{X'_U}$ be the push-forward of $A_{Y_U}$ 
to $X'_U$.
Let $A_{X'}$ be a $G$-equivariant closure of $A_{X'_U}$ on $X'$.
Then, $A_{X'}$ is a $G$-invariant pseudo-effective divisor over $Z$.
Furthermore, the relative Iitaka dimension of $A_{X'}$ over $Z$ is $\dim W_U$.
Indeed, the relative Iitaka dimension over $Z$ can be computed over the open set $U\subseteq Z$.
By Lemma~\ref{lem:FT-under-morphisms}, we know that $X'\rightarrow Z$ is a $G$-equivariant Fano type morphism.
So we can run a $G$-equivariant $A_{X'}$-MMP over $Z$.
Let $X'\dashrightarrow X''$ be the $G$-equivariant MMP over $Z$.
Let $B''$ be the image of $B'$ on $X''$.
The log pair $(X'',B'')$ is 
$G$-equivariantly crepant birational equivalent to $(X',B')$,
hence it is $G$-equivariantly crepant birational equivalent to $(X,B)$.
Let $A_{X''}$ be the push-forward of $A_{X'}$ to $X''$. 
The divisor $A_{X''}$ is $G$-invariant and semiample.
Let $X''\rightarrow X^{(3)}$ be the $G$-equivariant ample model of $A_{X''}$.
Then, the variety $X^{(3)}$ has dimension equal
to such of $W_U$.
Thus $X''$ admits a $G$-equivariant fibration
of relative dimension $d$. 
\end{proof}

\subsection{Finite actions on conic fibrations}

In this subsection, we study finite actions on conic fibrations.

\begin{lemma}\label{lem:finite-action-conic-bundle}
Let $X\rightarrow Z$ be a conic fibration.
Let $(X,B)$ be a log pair which is log Calabi--Yau over $Z$.
Assume that the log general fiber of $(X,B)\rightarrow Z$ is isomorphic to $(\pp^1,\{0\}+\{\infty\})$.
Assume that $\zz_m \leqslant {\rm Aut}_Z(X,B)$ and $m\geq 7$.
Then $\lfloor B\rfloor$ has two $\zz_m$-invariant prime components that are horizontal over $Z$.
\end{lemma}

\begin{proof}
Assume otherwise that $\lfloor B\rfloor$ has a unique prime component $S$ that is horizontal over $Z$. Note that $S$ is $\zz_m$-invariant.
Then, there is a $\zz_m$-equivariant commutative diagram as follows: 
\[
\xymatrix{
(X,B)\ar[d] & (X',B')\ar[d]\ar[l]_-{/\zz_2} \\
Z & Z' \ar[l]_{/\zz_2}
}
\]
obtained by taking the normalization of the base change induced by the $\zz_m$-equivariant finite morphism $S\rightarrow Z$.
Here, the morphism $X'\rightarrow Z'$ is a conic fibration, the pair 
$(X',B')$ is the log pull-back of $(X,B)$ to $X'$, and $\lfloor B'\rfloor$ has two prime components $S_0'$ and $S'_\infty$ that are $\zz_m$-invariant and horizontal over $Z$. 
Let $i$ be the involution acting on $X'$ that swaps $S'_0$ and $S'_\infty$.
Let $G$ be the group of automorphisms of $(X',B')$ generated by $\zz_2$ and $\zz_m$. 
We have an exact sequence
\begin{equation}\label{ses:2-m}
1\rightarrow \zz_2 \rightarrow G \rightarrow \zz_m \rightarrow 1.
\end{equation} 
As $m\geq 7$, the group $\zz_m$ is acting as multiplication by $m$-roots of unity on the fibers of $X'\rightarrow Z'$.
Let $u\in Z'$ be a closed general point and $i(u)$ its image under $i$. By abuse of notation, we are denoting by $i$ the automorphism on $Z'$ induced by the involution on $X'$.
Let 
\[
(X'_u,B'_u) 
\text{ and } 
\left(X'_{i(u)},B'_{i(u)} \right)
\]
be the log fibers over $i$ and $i(u)$, respectively.
Fix isomorphisms of log pairs $f\colon (\pp^1,\{0\}+\{\infty\})\rightarrow (X'_u,B'_u)$
and $g\colon (\pp^1,\{0\}+\{\infty\})\rightarrow (X'_{i(u)},B'_{i(u)})$ that map zero to $S'_0$ and infinity to $S'_\infty$.
Then $i_0:=f^{-1}\circ i \circ g$ is an involution of $(\pp^1,\{0\}+\{\infty\})$.
Thus, we can write $i_0([x:y])=([\alpha y:\beta x])$ for certain $\alpha,\beta\in \mathbb{K}^*$.
Let $\mu_m$ be the multiplication by an $m$-root of unity $\lambda_m$ on the first coordinate of $\pp^1$.
From~\eqref{ses:2-m}, the $\zz_2$-action commutes with the $\zz_m$-action. So, we conclude that 
$i_0 \circ \mu_m =\mu_m \circ i_0$.
Thus, we have  
\[
i_0\circ \mu_m([x:y])=i_0([\lambda_m x:y])=[\alpha y: \lambda_m\beta x]
\]
and 
\[
\mu_m\circ i_0([x:y])=
\mu_m([\alpha y:\beta x])=
[\lambda_m \alpha y:\beta x].
\]
From the previous equalities, we deduce that $m=2$. This contradicts the fact that $m\geq 7$. Hence, $\lfloor B\rfloor$ has two $\zz_m$-invariant prime components.
\end{proof}

\begin{lemma}\label{lem:making-sections-disjoints}
Let $X$ be a $\qq$-factorial variety and $X\rightarrow W$ be a Fano type fibration.
Let $(X,B,\mathbf{M})$ be a generalized log Calabi--Yau pair and $G\leqslant {\rm Aut}(X,B,\mathbf{M})$ be a finite group.
Let $f\colon X\rightarrow Z$ be a conic fibration over $W$
and assume that $\lfloor B\rfloor$ has two $G$-invariant prime divisors that are horizontal over $Z$.
Let $G_Z$ be the homomorphic image of $G$ acting on $Z$.
Let $(Z,B_Z,\mathbf{N})$ be the $G_Z$-equivariant generalized log Calabi--Yau pair induced by the generalized canonical bundle formula.
Then, there exists a commutative diagram as follows:
\[
\xymatrix{ 
(X,B,\mathbf{M})\ar[d]_-{f} & (X',B',\mathbf{M})\ar[d]^-{f'}\ar@{-->}[l]_-{\phi} \\
(Z,B_Z,\mathbf{N}) &
(Z',B_{Z'},\mathbf{N})\ar[l]^-{\phi_Z}
}
\]
satisfying the following conditions:
\begin{enumerate}
    \item the morphism $\phi_Z$ is a $G_Z$-equivariant crepant birational modification,
    \item the birational map $\phi$ is crepant, $G$-equivariant, and only extracts generalized log canonical places of $(X,B,\mathbf{M})$, 
    \item the morphism $f'$ is a $G$-equivariant strict conic fibration.
\end{enumerate}
\end{lemma} 

\begin{proof}
Let $S_0$ and $S_\infty$ be the $G$-invariant prime components of $\lfloor B\rfloor$ that are horizontal over $Z$.
First, we run a $G$-equivariant $K_X$-MMP over $Z$. 
We get a $G$-equivariant commutative diagram as follows:
\[
\xymatrix{ 
(X,B,\mathbf{M})\ar@{-->}[r]^-{\pi}\ar[d]_-{f} & (X_0,B_0,\mathbf{M})\ar[d]^-{f_0} \\
(Z,B_Z,\mathbf{N}) &
(Z_0,B_{Z_0},\mathbf{N})\ar[l]^-{\psi_Z}
}
\]
where $f_0$ is a $G$-equivariant Mori fiber space of relative dimension one over $Z$.
Further, $\psi_Z$ is a $G_Z$-equivariant crepant birational modification.
As usual, the generalized pair $(X_0,B_0,\mathbf{M})$ is the $G$-equivariant generalized log Calabi--Yau pair induced on $X_0$
and $(Z_0,B_{Z_0},\mathbf{N})$ is the $G_Z$-equivariant generalized log Calabi--Yau pair induced by the equivariant generalized canonical bundle formula.
Let $p_Z\colon Z'\rightarrow Z_0$ be a gdlt modification of $(Z_0, B_{Z_0},\mathbf{N})$ and 
$(Z',B_{Z'},\mathbf{N})$ be the generalized pair induced on $Z'$.
By Proposition~\ref{prop:2-ray}, there is a $G$-equivariant commutative diagram as follows:
\[
\xymatrix{ 
(X,B,\mathbf{M})\ar@{-->}[r]^-{\pi} \ar[d]_-{f} & (X_0,B_0,\mathbf{M})\ar[d]^-{f_0} & (X',B',\mathbf{M}) \ar@{-->}[l]_-{\phi_1} \ar[d]^-{f'} \\
(Z,B_Z,\mathbf{N}) &
(Z_0,B_{Z_0},\mathbf{N})\ar[l]^-{\psi_Z} & (Z',B_{Z'},\mathbf{N})\ar[l]^-{p_Z}
}
\]
where $\phi_1$ only extracts generalized log canonical places of $(X,B,\mathbf{M})$
and $f'$ is a $G$-equivariant conic fibration
with $\rho^G(X_1/Z_1)=1$.
Further, we have ${f'}^{-1}\lfloor B_{Z'}\rfloor \subseteq \lfloor B'\rfloor$ by Proposition~\ref{prop:2-ray}.(3).
Note that the generalized pair
$(Z',B_{Z'}-\epsilon \lfloor B_{Z'}\rfloor,\mathbf{N})$ is gklt.
Hence, all the log canonical centers of the generalized pair 
$(X',B'-\epsilon {f'}^*\lfloor B'\rfloor,\mathbf{M})$ are horizontal over $Z'$.
In particular, the only log canonical centers of 
$(X',B'-\epsilon {f'}^*\lfloor B'\rfloor,\mathbf{M})$ are $S'_0$
and $S'_{\infty}$.
Here, $S'_0$ and $S'_\infty$ are the strict transforms of $S_0$ and $S_\infty$ in $X'$, respectively.
If $S'_{0}$ and $S'_{\infty}$ intersect, then there would be a generalized log canonical center of $(X',B'-\epsilon {f'}^*\lfloor B'\rfloor,\mathbf{M})$ that is vertical over $Z'$, leading to a contradiction.
Set $\phi:= \pi^{-1}\circ \phi_1$
and $\phi_Z := \psi_Z\circ p_Z$.
By construction, $\phi_Z$ is a crepant $G_Z$-equivariant birational modification, so condition (1) holds.
The birational map $\phi$ is the composition of a $G$-equivariant contraction and a $G$-equivariant birational map that only extracts generalized log canonical places of $(X,B,\mathbf{M})$. Thus, the second condition holds.
The morphism $f'$ is a $G$-equivariant conic fibration with $\rho^G(X'/Z')=1$.
Finally, by construction, $\lfloor B'\rfloor$ has two $G$-invariant disjoint prime components that are horizontal over $Z$. This proves condition (3).
\end{proof}

\subsection{Generalized pairs on toric varieties} 
In this subsection, we prove some results 
regarding generalized pairs on toric varieties.

\begin{lemma}\label{lem:toric-indeterminacy-locus}
Let $X$ be a projective toric variety
and let $X\dashrightarrow Y$ be a birational map.
Then, there exists a projective toric birational morphism $X'\rightarrow X$ such that
the indeterminacy locus of $X'\dashrightarrow Y$ does not contain any toric strata of $X'$.
Furthermore, we may choose $X'$ to be smooth.
\end{lemma} 

\begin{proof}
Let $B$ be the reduced toric boundary of $X$ so $(X,B)$ is log Calabi--Yau toric pair.
Let $\pi\colon Z\rightarrow X$ be a log resolution of $(X,B)$ for which 
the birational map $X\dashrightarrow Y$
extends to a morphism $\pi\colon Z\rightarrow Y$.
Let $(Z,B_Z)$ be the pair obtained by log pull-back of $(X,B)$ to $Z$.
Note that $(Z,B_Z)$ is a sub-log Calabi--Yau pair.
Let $\Gamma_Z:=B_Z^{=1}$.
The prime components of $\Gamma_Z$ are precisely the toric valuations extracted by $Z\rightarrow X$ (see Lemma~\ref{lem:G_m-action-lift-to-gdlt}).
Fix $0<\epsilon<1$.
Let $\Delta_Z$ be the boundary divisor obtained from $B_Z$ by increasing the coefficients of the prime components in ${\rm Ex}(\pi)\setminus \Gamma_Z$ to $\epsilon$.
Then, we can write $\Delta_Z:=\Gamma_Z+\epsilon F_Z$
where $F_Z$ is reduced.
The log pair $(Z,\Delta_Z)$ is dlt and every prime component of $F_Z$ is exceptional over $X$.
Therefore, the divisor $K_Z+\Gamma_Z+\epsilon F_Z$ is
$\qq$-linearly equivalent to 
an effective divisor that is exceptional over $X$.
The pair $(Z,\Gamma_Z+F_Z)$ is dlt 
as $\Gamma_Z+F_Z$ has simple normal crossing support. 
So the divisor $F_Z$ contains no strata of $\Gamma_Z$.
We run a $(K_Z+\Gamma_Z+\epsilon F_Z)$-MMP over $X$ with scaling of an ample divisor. By Lemma~\ref{lem:negativity}, this MMP terminates in a model $W$ after contracting all the components of $F_Z$.
Let $(W,\Gamma_W)$ be the dlt pair induced on $W$.
Note that $(W,\Gamma_W)$ is a dlt modification of $(X,B)$. Hence, $W$ is a smooth projective toric variety.
Let $q\colon Z\dashrightarrow W$ be the induced morphism.
By monotonicity of log discrepancies, the subvariety $V:={\rm Ex}(q^{-1})$ does not contain
any log canonical center of $(W,\Gamma_W)$.
Hence, $V$ does not contain any dlt center of $(W,\Gamma_W)$.
The dlt centers of $(W,\Gamma_W)$ are precisely the toric strata of $W$.
Thus, $V$ does not contain any toric strata of $W$.
Note that $W \dashrightarrow Z$ is an isomorphism on the complement of $V$.
Hence, the indeterminacy locus of $W\dashrightarrow Y$ is contained in $V$.
So, the indeterminacy locus of $W\dashrightarrow Y$ does not contain any toric strata of $W$.
Thus, it suffices to set $X':=W$.
In our construction, $(W,\Gamma_W)$ is a dlt toric pair so it is log smooth. In particular, $X'$ is smooth.
\end{proof}

\begin{lemma}\label{lem:toric-gen}
Let $(X,B,\mathbf{M})$ be a generalized pair on a projective toric variety $X$.
Assume that $N\mathbf{M}$ is nef Cartier in a model where it descends.
Then, there exists a projective birational toric morphism
$X'\rightarrow X$ for which $N\mathbf{M}_{X'}$ is a base point free divisor.
\end{lemma} 

\begin{proof}
Let $Y$ be a model where $\mathbf{M}$ descends.
Hence, $N\mathbf{M}_Y$ is a nef Cartier divisor.
By Lemma~\ref{lem:toric-indeterminacy-locus}, 
there exists a projective birational toric morphism $X'\rightarrow X$
such that the indeterminacy locus $V$ of $X'\dashrightarrow Y$ does not contain any toric strata of $X'$.
We may assume that $X'$ is a smooth projective toric variety.
Let $p\colon Z\rightarrow X'$
and $q\colon Z \rightarrow Y$ be a resolution of the indeterminacy locus of the birational map $X'\dashrightarrow Y$. 
We may assume that $p({\rm Ex}(p))=V$
as $V$ is the indeterminacy locus of $X'\dashrightarrow Y$.
As $X'$ is smooth, the divisor
$N\mathbf{M}_{X'}$ is Cartier. 
Let $C\subset X'$ be a torus invariant curve. By construction, $C$ is not contained in $V$.
Let $C_Z$ be the strict transform of $C$ in $Z$.
Write $p^*\mathbf{M}_{X'}=\mathbf{M}_Z+E_Z$ where $E_Z$ is an effective divisor supported on the exceptional locus of $p$. 
In particular, we have $C_Z\cdot E_Z\geq 0$.
Since $Z\rightarrow Y$ is a projective birational morphism, we have 
$\mathbf{M}_Z=q^*\mathbf{M}_Y$, so $\mathbf{M}_Z$ is a nef divisor.
Then, we can write
\[
\mathbf{M}_{X'} \cdot C  =
p^*\mathbf{M}_{X'} \cdot C_Z = 
(\mathbf{M}_Z+E_Z)\cdot C_Z  \geq 0.
\]
We conclude that $N\mathbf{M}_{X'}$ is a Cartier divisor that intersects every torus invariant curve non-negatively.
By~\cite[Theorem 6.3.12]{CLS11}, we conclude that $N\mathbf{M}_{X'}$ is a base point free Cartier divisor.
\end{proof}

\subsection{Algebraic torus actions} In this section, we prove some results regarding algebraic torus actions on normal projective varieties and automorphisms of log Calabi--Yau pairs.

\begin{definition}
{\em 
Let $X$ be a normal projective variety with a $\mathbb{G}_m$-action.
Let $Y$ be the normalized Chow quotient of $X$ by the $\mathbb{G}_m$-action.
By~\cite[Theorem 5.6]{AHS08}, there is a divisorial fan $\mathcal{S}$ on $(Y,\qq)$ for which 
there is a $\mathbb{G}_m$-equivariant isomorphism $X(\mathcal{S})\simeq X$.
The tail-fan of $\mathcal{S}$ is $\Sigma=\{ \rho_{-}, \{0\}, \rho_{+}\}$,
where $\rho_{-}:=\qq_{\leq 0}$ and $\rho_{+}:=\qq_{\geq 0}$.
Ther rays $\rho_{-}$ and $\rho_{+}$ induce two $\mathbb{G}_m$-invariant  divisorial valuations $E_{-}$ and $E_{+}$ over $X$.
We call $E_{-}$ the {\em divisorial sink}
and $E_{+}$ {\em divisorial source}
of the $\mathbb{G}_m$-action, respectively.

Let $\widetilde{X}\rightarrow X$ be a $\mathbb{G}_m$-equivariant projective birational morphism for which there is a $\mathbb{G}_m$-quotient $\widetilde{X}\rightarrow Y$.
The center of the divisorial sink and the divisorial source of $\mathbb{G}_m$ on $\widetilde{X}$ are the two $\mathbb{G}_m$-invariant prime divisors that dominate $Y$.
These are precisely the only two prime divisors on $\widetilde{X}$ on which $\mathbb{G}_m$ acts as the identity (see, e.g.,~\cite[Theorem 10.1]{AH06}). 
Hence, the sink and the source of the $\mathbb{G}_m$-action on $X$ are the centers of the divisorial sink and the divisorial source, respectively.
}
\end{definition}

The following lemma implies that on an affine variety either the center of the divisorial sink
or the center of the divisorial source have codimension at least two.

\begin{lemma}\label{lem:gm-fixed-divisors}
Let $X$ be a normal affine variety with a faithful $\mathbb{G}_m$-action.
Then, there is at most one prime divisor 
on $X$ that is fixed pointwise by the $\mathbb{G}_m$-action.
\end{lemma} 

\begin{proof}
By~\cite[Theorem 1]{Wat81}, we have a $\mathbb{G}_m$-equivariant isomorphism
\[
X\simeq {\rm Spec}\left(\bigoplus_{m\in \zz_{\geq 0}} H^0(Y,\mathcal{O}_X(mD_Y))\right),
\]
where $Y$ is a semiprojective variety\footnote{A semiprojective variety is a variety that is projective over an affine variety.} and $D_Y$ is an ample $\qq$-Cartier $\qq$-divisor.
Consider the variety 
\[
\widetilde{X}:={\rm Spec}_X\left( 
\bigoplus_{m\in \zz_{\geq 0}}
\mathcal{O}_X(mD_Y)
\right).
\]
Then, there is a $\mathbb{G}_m$-equivariant projective birational morphism $\widetilde{X}\rightarrow X$.
The variety $\widetilde{X}$ has at most one $\mathbb{G}_m$-invariant prime divisor $E$ that dominates $Y$.
The morphism $\widetilde{X}\rightarrow Y$ is a quotient for the $\mathbb{G}_m$-action.
Hence, the prime divisor $E$ is the only divisor on which $\mathbb{G}_m$ can act as the identity.
Therefore, there is at most one prime divisor on $X'$ that is fixed pointwise by the $\mathbb{G}_m$-action. 
Since $\widetilde{X}\rightarrow X$ is a $\mathbb{G}_m$-equivariant projective birational morphism, the same statement holds for $X$.
\end{proof} 

\begin{lemma}\label{lem:G_m-action-lift-to-gdlt}
Let $X$ be a klt variety and
$(X,B)$ be a log canonical pair.
Assume that $\mathbb{G}_m\leqslant {\rm Aut}(X,B)$.
Let $Y\rightarrow X$ be a projective birational morphism that only extracts non-canonical places of $(X,B)$.
Then, the morphism $Y\rightarrow X$ is $\mathbb{G}_m$-equivariant.
\end{lemma}

\begin{proof}
Let $p\colon Z\rightarrow X$ be a $\mathbb{G}_m$-equivariant log resolution of $(X,B)$
and let $p^*(K_X+B)=K_Z+B_Z$.
By blowing up strata of $B_Z$, we may assume that every non-canonical place of $(X,B)$ is extracted on $Z$.
Thus, we have a birational contraction $Z\dashrightarrow Y$.
Let $\Gamma_Z$ be the divisor obtained from $B_Z$ by increasing all its negative coefficients to one.
Hence, $(Z,\Gamma_Z)$ is a log pair.
Further, we have $\mathbb{G}_m\leqslant {\rm Aut}(Z,\Gamma_Z)$.
We run a $\mathbb{G}_m$-equivariant $(K_Z+\Gamma_Z)$-MMP over $X$.
By Lemma~\ref{lem:negativity}, this $\mathbb{G}_m$-equivariant MMP terminates with a $\mathbb{G}_m$-equivariant dlt modification $(Y',B_{Y'})$ of $(X,B)$
that admits a birational contraction $Y'\dashrightarrow Y$ over $X$.
Note that $Y'\rightarrow X$ is a Fano type morphism by Lemma~\ref{lem:FT-under-morphisms}.
Hence, the birational contraction $Y'\dashrightarrow Y$ is the outcome of an MMP over $X$ for a divisor $D$ on $Y'$. The divisor $D$ is linearly equivalent to a $\mathbb{G}_m$-invariant divisor $D'$.
Thus, the MMP $Y'\dashrightarrow Y$ over $X$ is $\mathbb{G}_m$-equivariant. Hence, $Y$ admits a $\mathbb{G}_m$-action and $Y\rightarrow X$ is $\mathbb{G}_m$-equivariant.
\end{proof}

\begin{lemma}\label{lem:iit-dim}
Let $X$ be a Fano type variety of dimension $n$.
Let $(X,B,\mathbf{M})$ be a generalized log Calabi--Yau pair with ${\rm reg}(X,B,\mathbf{M})\geq 1$.
Assume that $\mathbb{G}_m$ acts faithfully on $X$ and both the divisorial sink and the divisorial source of the action are prime divisors $S_1$ and $S_2$ on $X$.
Further, assume that $S_1$ and $S_2$ are contained in $\lfloor B\rfloor$.
Then, we have that
\[
0<k(B+\mathbf{M}_X-S_1-S_2)<n.
\]
\end{lemma} 

\begin{proof}
Let $\pi\colon X'\rightarrow X$ be a $\mathbb{G}_m$-equivariant projective birational morphism for which there is
a good quotient $q\colon X'\rightarrow Y$ for the torus action.
Let $(X',S'_1+S'_2)$ be the log pull-back
of $(X,S_1+S_2)$ to $X'$.
Since $S_1$ and $S_2$ are the sink and the source of the torus action, the restriction of the log pair $(X',S'_1+S'_2)$
to a general fiber of $q$ is isomorphic to $(\pp^1,\{0\}+\{\infty\})$.
In particular, if $C$ is a general fiber of $q$, 
then $\pi^*(K_X+S_1+S_2)\cdot C =0$.
We conclude that 
$\pi^*(B+\mathbf{M}_X-S_1-S_2)\cdot C=0$.
The image of $C$ on $X$ is a movable curve.
Thus, the divisor 
$B+\mathbf{M}_X-S_1-S_2$ intersects a movable curve trivially.
This implies that the divisor $B+\mathbf{M}_X-S_1-S_2$ is not big.

Now, it suffices to show that the Iitaka dimension of $B+\mathbf{M}_X-S_1-S_2$ is non-zero.
We proceed by contradiction.
For simplicity, we write $B_0:=B-S_1-S_2$.
Note that $B_0+\mathbf{M}_X$ is a pseudo-effective divisor so its Iitaka dimension is non-negative as $X$ is a Mori dream space.
Assume that the Iitaka dimension of $B_0+\mathbf{M}_X$ is zero.
Let $X\dashrightarrow Z$ be the outcome of a $(B_0+\mathbf{M}_X)$-MMP.
Let $B_{Z,0}+\mathbf{M}_Z$ be the image of $B_0+\mathbf{M}_X$ on $Z$.
Let $S_{Z,1}$ and $S_{Z,2}$ be the image of $S_1$ and $S_2$ on $Z$, respectively.
By construction, the divisor $B_{Z,0}+\mathbf{M}_Z$ is semiample and has Iitaka dimension zero.
Hence, the divisor $B_{Z,0}+\mathbf{M}_Z$ is torsion. 
By Lemma~\ref{lem:b+m-tor}, 
we conclude that $B_{Z,0}=0$
and $\mathbf{M}$ is torsion in any model where it descends.
This means that $\mathbf{M}$ is a torsion b-divisor.
Hence, the pair
$(Z,S_{Z,1}+S_{Z,2})$ is log Calabi--Yau and has regularity at least $1$.
Indeed, by construction, the log Calabi--Yau pair $(Z,S_{Z,1}+S_{Z,2})$ has the same log discrepancies as the pair $(X,B,\mathbf{M})$.
Since $X$ is of Fano type, we have that $Z$ is of Fano type (see Lemma~\ref{lem:FT-under-morphisms}).
In particular, $(Z,S_{Z,1}+S_{Z,2})$ is klt in the complement of the support of $S_{Z,1}+S_{Z,2}$.
Therefore, every log canonical center of $(Z,S_{Z,1}+S_{Z,2})$ is contained in either the support of $S_{Z,1}$ or the support of $S_{Z,2}$.
We argue that the log Calabi--Yau pair
$(Z,S_{Z,1}+S_{Z,2})$ admits an effective $\mathbb{G}_m$-action.
Indeed, the birational contraction $X\dashrightarrow Z$ is equivariant as
the divisor $B_0+\mathbf{M}_X$ is $\qq$-linearly equivalent to a $\mathbb{G}_m$-invariant divisor.
Furthermore, $\mathbb{G}_m$-acts as the identity on $S_{Z,1}$ and $S_{Z,2}$.
Indeed, this statement holds in $X$ and 
$X\dashrightarrow Z$ is an isomorphism over the generic point of both $S_1$ and $S_2$.
If $S_{Z,1}\cap S_{Z,2}=\emptyset$, then
$\mathcal{D}(Z,S_{Z,1}+S_{Z,2})$ is disconnected.
This contradicts the fact that 
${\rm reg}(Z,S_{Z,1}+S_{Z,2})\geq 1$ (see, e.g.,~\cite[Theorem 1.1]{FS20}).
On the other hand, if
$S_{Z,1}\cap S_{Z,2}\neq \emptyset$, we can take $z\in S_{Z,1}\cap S_{Z,2}$
and find a $\mathbb{G}_m$-equivariant affine neighborhood of $U$ of $z$ (see, e.g.,~\cite{Sum74}).
The normal affine variety $U$ admits a $\mathbb{G}_m$-action and it has two prime divisors $S_{Z,1}\cap U$ and $S_{Z,2}\cap U$ that are fixed pointwise by the $\mathbb{G}_m$-action. This leads to a contradiction of Lemma~\ref{lem:gm-fixed-divisors}.

We conclude that the Iitaka dimension of $B+\mathbf{M}_X-S_1-S_2$ is larger than zero and less than the dimension of $X$.
\end{proof}

The following lemma characterizes connected automorphism groups of sub-log Calabi--Yau pairs.

\begin{lemma}\label{lem:aut-sub-log-CY}
Let $X$ be a klt rationally connected variety.
Let $(X,B)$ be a sub-log Calabi--Yau pair.
Then, the group ${\rm Aut}^0(X,B)$ is an algebraic torus.
\end{lemma} 

\begin{proof}
Let $G:={\rm Aut}^0(X,B)$. 
The group $G$ is a smooth connected algebraic group.
We may take a $G$-equivariant log resolution of $(X,B)$. 
Thus, we may assume that $(X,B)$ is log smooth and admits an effective $G$-action.
Since $X$ is rationally connected, the group $G$ is a linear algebraic group.
Let $\Gamma$ be the effective divisor obtained from $B$ by increasing all its negative coefficients to one.
Then, $G$ acts on the log smooth pair $(X,\Gamma)$ and
$K_X+\Gamma$ is pseudo-effective.
By~\cite[Theorem 1.1]{Hu18}, we conclude that $G$ is an algebraic torus.
\end{proof} 

We conclude this section 
proving that a bounded family of projective toric varieties is indeed finite.

\begin{lemma}\label{lem:toric-bound-finite}
Let $\mathcal{T}$ be a bounded family of projective toric varieties. Then, the bounded family $\mathcal{T}$ is indeed a finite family up to isomorphism.
\end{lemma}

\begin{proof}
Let $\mathcal{X}\rightarrow U$ be a bounding family for $\mathcal{T}$, i.e., for every $T\in \mathcal{T}$ we have $T\simeq \mathcal{X}_u$ for some closed point $u\in U$.
Let $n$ be the dimension of the general fiber of $\mathcal{X}\rightarrow U$.
By Noetherian induction, it suffices to prove that only finitely many isomorphisms classes of projective toric varieties appear as fibers of $\mathcal{X}\rightarrow U$
after possibly shrinking $U$.
Let $\mathcal{L}$ be a very ample divisor on $\mathcal{X}$. 
Shrinking $U$, we may assume that 
there exists a constant $k$ such that
$h^0(\mathcal{X}_u,\mathcal{L}_u)=k$
for all $u\in U$.
Assume that $\mathcal{X}_u$ is a projective toric variety. 
Then, there is a torus invariant divisor $T_u \sim \mathcal{L}_u$ that is very ample
and for which $h^0(\mathcal{X}_u,T_u)=k$.
Let $P_u \subset \qq^n$ be the lattice polytope corresponding to the pair $(\mathcal{X}_u,T_u)$ (see~\cite[Theorem 6.2.1]{CLS11}).
The torus equivariant isomorphism class of $\mathcal{X}_u$ is determined by the lattice polytope $P_u$ up to unimodular transformation of $\qq^n$.
The polytope $P_u$ contains exactly $k$ lattice points of $\zz^n$ (see, e.g.,~\cite[Proposition 5.4.1]{CLS11}).
By~\cite[Theorem 1]{LZ91}, we conclude that $\vol(P_u)\leqslant f(n,k)$ for a constant $f(k,n)$ only depending on $n$ and $k$.
By~\cite[Theorem 2]{LZ91}, we conclude that the polytope $P_u$ belongs to a finite set, that only depends on $n$ and $k$, up to unimodular transformation.
Thus, there are only finitely many equivariant isomorphism classes for the projective toric varieties that appear as fibers of $\mathcal{X}\rightarrow U$.
\end{proof}

\section{Finite actions on Fano fibrations}
In this section, we prove results regarding finite abelian actions on Fano fibrations. 
First, we prove a couple of statements about finite subgroups of algebraic tori over fields.

\begin{lemma}\label{lem:GLr-action}
Let $r$ be a positive integer.
There exists a constant $m:=m(r)$,
only depending on $r$, satisfying the following.
Let $\mathbb{F}$ be a field of characteristic zero, 
let $\mathbb{T}$ be an $r$-dimensional split torus over $\mathbb{F}$, and 
let $G \leqslant {\rm GL}_r(\zz) \leqslant {\rm Aut}(\mathbb{T})$ be a finite subgroup.
If $G$ commutes with $\zz_\ell \leqslant \mathbb{T}$
and $\ell \geq m$, then 
there exist
\[
C \leqslant \mathbb{T}_0  \leqslant \mathbb{T}
\]
where $C$ is a subgroup of $\zz_\ell$ of index at most $\frac{m}{2}$ and 
$\mathbb{T}_0$ is a one-dimensional split torus over $\mathbb{F}$ that commutes with $G$.
\end{lemma}

\begin{proof}
By the Jordan-Zasenhaus Theorem,
there are only finitely many finite subgroups of ${\rm Gl}_r(\zz)$ up to conjugation.
Let $H_1,\dots,H_u$ be a finite set 
of finite subgroups of ${\rm GL}_r(\zz)$ such that
every finite subgroup of ${\rm GL}_r(\zz)$ is conjugate with one of the $H_i$'s.
Hence, we have a finite set 
\[
\{
C(H_i)\cap \mathbb{T} \mid i\in \{1,\dots,u\}
\} 
\]
of algebraic groups.
In particular, there is a constant $m:=m(r)$ satisfying the following.
In the previous, $C(H_i)$ denotes the centralizer of $H_i$.
If $\zz_\ell \leqslant C(H_i)\cap \mathbb{T}$
and $\ell \geq m$, then there exist $C\leqslant \mathbb{T}_0 \leqslant C(H_i) \cap \mathbb{T}$, 
where $C$ is a subgroup of $\zz_\ell$ of index at most $\frac{m}{2}$ and $\mathbb{T}_0$ is a one-dimensional split torus over $\mathbb{F}$.
We fix such an $m$.

Now, let $G\leqslant {\rm GL}_r(\zz)$ 
be a finite group commuting with $\zz_\ell$ and $\ell\geq m$. We have 
$\zz_\ell \leqslant C(G) \cap \mathbb{T}$. 
There exists $h\in {\rm GL}_r(\zz)$ such that
$hGh^{-1}=H_i$ for some $i$.
In particular, we have 
\[
h\zz_\ell h^{-1} \leqslant hC(G)h^{-1} \cap \mathbb{T} = C(H_i) \cap \mathbb{T}.
\]
By the first paragraph, there exists a subgroup 
$C'\leqslant h\zz_\ell h^{-1}$ of 
index at most $\frac{m}{2}$
and a 
one-dimensional torus $\mathbb{T}_0' \geqslant C$ such that $\mathbb{T}'_0 \leqslant C(H_i) \cap \mathbb{T}$. 
We set $C:=h^{-1}C'h$ and $\mathbb{T}_0:=h^{-1}\mathbb{T}_0'h$.
Hence, $C$ is a subgroup of index at most $\frac{m}{2}$ in $\zz_\ell$ for which 
$C \leqslant \mathbb{T}_0 \leqslant \mathbb{T}$
and $\mathbb{T}_0$ commutes with $G$.
\end{proof}

\begin{lemma}\label{lem:non-split-r-torus}
Let $r$ be a positive integer.
There is a constant $m:=m(r)$ satisfying the following. 
Let $\kk$ be an algebraically closed field of characteristic zero and $\mathbb{F}\supseteq \kk$ be a field extension.
Let $\mathbb{T}$ be an $r$-dimensional torus over $\mathbb{F}$, let $\mathbb{F}'\supset \mathbb{F}$ be its splitting field, let $G$ be the Galois group of the field extension $\mathbb{F}'\supseteq \mathbb{F}$, and 
let 
$\mathbb{T}'$ be the corresponding split torus over $\mathbb{F}'$. 
Assume that $\zz_\ell \leqslant \mathbb{T}_1' \leqslant \mathbb{T}'$, where:
\begin{itemize}
\item we have $\ell \geq m$, 
\item the group $\zz_\ell$ commutes with $G$, and 
\item the torus $\mathbb{T}_1'$ is invariant under the conjugation $G$-action.
\end{itemize} 
Then, there exists 
\[
C \leqslant \mathbb{T}'_0 \leqslant \mathbb{T}'_1\leqslant \mathbb{T}'
\]
where $C$ is a subgroup of $\zz_\ell$ of index at most $\frac{m}{2}$ and $\mathbb{T}'_0$ is a one-dimensional split torus over $\mathbb{F}'$ that is invariant under the conjugation $G$-action.
\end{lemma}

\begin{proof}
Write $\mathbb{T}':={\rm Spec}(\mathbb{F}'[M])$
where $M$ is a free finitely generated
abelian group of rank $r$.
There is a homomorphism
$\tau \colon G \rightarrow {\rm GL}_r(\zz)$
such that the action of $G$
on $\mathbb{F}'[M]$ is given by  
\begin{equation}\label{eq:Gal-action}
g(\lambda \chi^m)=
\gamma(\lambda)\chi^{\tau(g)(m)}
\end{equation} 
with $\gamma \in {\rm Gal}(\mathbb{F}'/\mathbb{F})$.
For the previous statement see~\cite[Remark 3.6]{Gil22}.
Let $H$ be the image of $\tau$.
By equality~\eqref{eq:Gal-action}, every element $g$ of $G$ can be written as
$g=\gamma \circ h$ where $h\in H$ 
and $\gamma \in {\rm Gal}(\mathbb{F}'/\mathbb{F})$ acts trivially on the weighted monoid of $\mathbb{T}'$.

We argue that $\mathbb{T}'_1$ is invariant under the conjugation $H$-action.
For every $h\in H$, we can find $g\in G$ for which $g=\gamma \circ h$. Therefore, we have
\[
\gamma \circ h \circ \mathbb{T}'_1 \circ h^{-1} \circ \gamma^{-1} = \mathbb{T}'_1 
\]
as $\mathbb{T}'_1$ is invariant under the conjugation $G$-action. 
Thus, we have
\[
h \circ \mathbb{T}'_1 \circ h^{-1} =
\gamma^{-1} \circ \mathbb{T}'_1 \circ \gamma = \mathbb{T}'_1,
\]
where the last equality holds as $\gamma$ acts trivially on the weighted monoid of $\mathbb{T}'$.
In particular, we have that $H\leqslant {\rm GL}_s(\zz)\leqslant {\rm Aut}(\mathbb{T}'_1)$ 
where $s$ is the rank of $\mathbb{T}'_1$ over $\mathbb{F}'$.

Now, we argue that $\zz_\ell$ commutes with $H$.
Let $\lambda \in \zz_\ell$ be a generator.
Since $\lambda$ commutes with $G$, we have 
\[
h^{-1}\circ \gamma^{-1} \circ \lambda \circ \gamma \circ h = \lambda,
\]
for each $g\in G$.
Hence, we have the following equalities
\[
h \circ \lambda \circ h^{-1} =
\gamma^{-1} \circ \lambda \circ \gamma = \lambda,
\]
where the last equality holds as $\gamma$ acts as the identity on $\ell$-roots of unity.
Here, we used the fact that $\gamma\in {\rm Gal}(\mathbb{F}'/\mathbb{F})$ and $\mathbb{F}$ contains $\mathbb{K}$, so $\mathbb{F}$ contains all roots of unity.
We conclude that $\zz_{\ell}\leqslant \mathbb{T}_1$ commutes with $H$.

Thus, we have $H\leqslant {\rm GL}_s(\zz) \leqslant {\rm Aut}(\mathbb{T}'_1)$ and $H$ commutes with $\zz_\ell \leqslant \mathbb{T}'_1$.
By Lemma~\ref{lem:GLr-action}, for $\ell \geq m$, we have 
\[
C \leqslant \mathbb{T}'_0 \leqslant \mathbb{T}'_1 \leqslant \mathbb{T}',
\]
where $C$ is a subgroup of $\zz_\ell$ of index at most $\frac{m}{2}$ and $\mathbb{T}'_0$ is a one-dimensional split torus over $\mathbb{F}'$ that commutes with $H$.
It suffices to show that $\mathbb{T}'_0$ is invariant under the conjugation $G$-action.
Let $g\in G$ and choose a decomposition $g=\gamma \circ h$ as above.
Let $t\in \mathbb{T}'_0$. 
Since $H$ commutes with $\mathbb{T}'_0$, we can write
\[
g\circ t\circ g^{-1} =
\gamma \circ h \circ t \circ h^{-1} \circ \gamma^{-1} = 
\gamma \circ t \circ \gamma^{-1}.
\]
Let $(x_1,\dots,x_r) \in ({\mathbb{F}'}^*)^r$.
We have
\[
t(x_1,\dots,x_r)=
(t^{\alpha_1}x_1,\dots,t^{\alpha_r}x_r),
\]
for certain integers $\alpha_1,\dots,\alpha_r$.
Then, there is a sequence of equalities
\begin{align*} 
\gamma \circ t \circ \gamma^{-1}(x_1,\dots,x_r)& =
\gamma \circ t(\gamma^{-1}(x_1),\dots,\gamma^{-1}(x_r)) =\\ 
\gamma(t^{\alpha_1}\gamma^{-1}(x_1),\dots,
t^{\alpha_r}\gamma^{-1}(x_r)) &= 
(\gamma(t)^{\alpha_1}x_1,\dots,
\gamma(t)^{\alpha_r}x_r).
\end{align*}
We conclude that the following equality holds
\[
g \circ t \circ g^{-1} = \gamma(t) \in \mathbb{T}'_0.
\]
Hence, the torus $\mathbb{T}'_0$ is closed under the conjugation $G$-action.
\end{proof}

Now, we turn to prove the main technical proposition of this section. 
The following theorem states that if we have a large cyclic action on a Fano fibration with bounded fibers, 
then general fibers admit tori actions in such a way that the divisorial sink and the divisorial source are compatible in the fibration.

\begin{theorem}\label{thm:cyc-act-Fano-fib}
Let $n$ and $N$ be positive integers
and $\epsilon$ be a positive real number.
There exists a constant $m:=m(n,N,\epsilon)$,
only depending on $n,N$ and $\epsilon$,
satisfying the following.
Let $f\colon X \rightarrow Z$ be a Fano fibration from a $n$-dimensional variety.
Assume that $X$ has $\epsilon$-lc singularities.
Let $(X,B,\mathbf{M})$ be a generalized log Calabi--Yau pair over $Z$ of index $N$.
Furthermore, assume that $\zz_\ell \leqslant {\rm Aut}_Z(X,B,\mathbf{M})$ 
and $N\mathbf{M}$ is nef Cartier in a quotient by $\zz_\ell$ where it descends.
If $\ell\geq m$, then there exist:
\begin{enumerate}
\item a projective birational morphism $Y\rightarrow X$ that only extracts 
generalized log canonical places of $(X,B,\mathbf{M})$, 
\item a reduced effective divisor $S$ on $Y$, and 
\item a subgroup $C\leqslant \zz_\ell$ of index at most $\frac{m}{2}$,
\end{enumerate}
satisfying the following.
For a general closed point $z\in Z$,
there exists a one-dimensional torus $C \leqslant \mathbb{T}_z \leqslant {\rm Aut}(X_z,B_z)$
for which $S|_{Y_z}$ is the reduced sum of the divisorial sink and the divisorial source of the $\mathbb{T}_z$-action.
\end{theorem} 

\begin{proof}
By assumption, a general fiber $X_z$ is
and $\epsilon$-lc Fano variety
and $N(K_{X_z}+B_z+\mathbf{M}_{z,X_z})\sim 0$.
In particular, 
the pairs $(X_z,B_z)$
as in the statement are log bounded in terms of $n,N$, and $\epsilon$ (see, e.g.,~\cite{Bir21}).
By~\cite[Theorem 4.1]{Mor20c}, there is a constant $m_0:=m_0(n,N,\epsilon)$ satisfying the following.
Whenever $\zz_\ell \leqslant {\rm Aut}(X,B)$ acts fiberwise over $Z$ and $\ell\geq m_0$, we can find a subgroup $C\leqslant \zz_\ell$ of index at most $\frac{m_0}{2}$ 
such that $C\leqslant \mathbb{T}_z\leqslant {\rm Aut}^0(X_z,B_z)$ for a general closed point $z\in Z$. Here, $\mathbb{T}_z$ is a maximal torus of ${\rm Aut}^0(X_z,B_z)$.
Let $\eta$ be the generic point of $Z$ and $\bar{\eta}$ be the geometric generic point.
We write $(X_{\eta},B_{\eta})$ and $(X_{\bar{\eta}},B_{\bar{\eta}})$ for the log pair induced on the generic fiber and the geometric generic fiber, respectively.
By~\cite[Lemma 2.1]{Via13}, we conclude that $C\leqslant \mathbb{T}_{\bar{\eta}} \leqslant {\rm Aut}(X_{\bar{\eta}},B_{\bar{\eta}})$, 
where $\mathbb{T}_{\bar{\eta}}$ is a split algebraic torus over $\bar{\eta}$ of rank $r\leq n$.
We write $\mathbb{T}_{\eta}$ for the corresponding algebraic torus over $\eta$.
In particular, up to shrinking $Z$ and assuming that the base is affine, we have $C\leqslant \mathbb{T}_Z \leqslant {\rm Aut}_Z(X,B)$
where $\mathbb{T}_Z$ is an algebraic torus over $Z$.

Let $\eta'\supseteq \eta$ be the splitting field of $\mathbb{T}_\eta$.
We know that 
$G:={\rm Gal}(\eta'/\eta)$ is a finite group.
Let $Z'\rightarrow Z$ be the corresponding Galois \'etale cover.
Then, we have a commutative diagram as follows:
\[
\xymatrix{
(X,B,\mathbf{M})\ar[d]_-{f} & (X',B',\mathbf{M}') \ar[l]^-{/G} \ar[d]_-{f'} & (X_{\bar{\eta}},B_{\bar{\eta}},\mathbf{M}_{\bar{\eta}})\ar[l]\ar[d] \\ 
Z & Z'\ar[l]^-{/G} & \bar{\eta} \ar[l].
}
\]
Furthermore, we have 
$C\leqslant \mathbb{T}_{Z'}\leqslant {\rm Aut}_{Z'}(X',B')$ where
$\mathbb{T}_{Z'}$ is an $r$-dimensional split torus and $r\leq n$.
Hence, there is a commutative diagram:
\[
\xymatrix{
1\ar[r] &  C\ar@{^{(}->}[d] \ar[r] & \tilde{G}\ar[r]\ar@{^{(}->}[d] & G\ar[d]^-{{\rm id}_G} \ar[r] & 1 \\
1\ar[r] & \mathbb{T}_{Z'}\ar[r] & H\ar[r] & G\ar[r] & 1 \\
}
\]
where $\tilde{G}$ is the finite group generated by $G$ and $C$. Since the $C$-action descends to the quotient $X=X'/G$, the $C$-action normalizes $G$.
By~\cite[Appendix B.1.]{Gil22}, we may assume that $G\hookrightarrow {\rm GL}_r(\zz)$. 
In particular, the order of $|G|$ is bounded above by a constant only depending on $n$.
Replacing $C$ with a subgroup of index at most $m_1(r)$, we may assume that $G$ commutes with $C$.
So $C\leqslant \mathbb{T}_{Z'}$ commutes with the Galois action.\\

\noindent\textit{Claim 1:} There exists a subtorus 
$\mathbb{T}_{Z',1}\leqslant \mathbb{T}_{Z'}$ 
with $C\leqslant \mathbb{T}_{Z',1}$
satisfying the following conditions:
\begin{itemize}
\item the torus $\mathbb{T}_{Z',1}$ is invariant under the conjugation $G$-action, and
\item for any one-dimensional subtorus
$\mathbb{H}\leqslant \mathbb{T}_{Z',1}$ 
the divisorial sink and the divisorial source of the $\mathbb{H}$-action are generalized log canonical centers of
$(X',B',\mathbf{M}')$.
\end{itemize}

\begin{proof}[Proof of Claim 1]
Let $W$ be the normalized Chow quotient of $X'$ 
for the $\mathbb{T}_{Z'}$-action.
Since $H$ normalizes $\mathbb{T}_{Z'}$,
we have an $H$-equivariant rational map
$X'\dashrightarrow W$. 
Then, we can find an $H$-equivariant resolution 
$\bar{X}\rightarrow X'$ for which 
$\bar{X}\rightarrow W$ is an $H$-equivariant morphism.
Since $\mathbb{T}_{Z'}$-acts trivially on $W$, 
the morphism
$\bar{X}\rightarrow W$ is a quotient for the $\mathbb{T}_{Z'}$-action.
Let $(\bar{X},\bar{B},\mathbf{M}')$ be the log pull-back of $(X',B',\mathbf{M}')$ to $\bar{X}$.
Hence, the generalized sub-pair
$(\bar{X},\bar{B},\mathbf{M}')$ is a generalized
sub-log Calabi--Yau pair.
Let $U_W\subseteq W$ be an open subset over which all fibers of
$\bar{X}\rightarrow U_W$ are toric.
By possibly shrinking $U_W$, we may assume that it is $G$-invariant.
Thus, we have an $H$-equivariant commutative diagram as follows:
\[
\xymatrix{
\bar{X}\ar[rd]\ar[rr]^-{q} & & \bar{X}/C \ar[ld] \\
& U_W & 
}
\]
By performing toric blow-ups on $\bar{X}$, we may assume that 
$\bar{X}/C$ is smooth.
Recall that $N\mathbf{M}'$ is nef Cartier in a quotient where it descends.
Hence, there is a b-nef divisor $\mathbf{N}$ on $X'/C$ that pull-backs to $\mathbf{M}$ for which $N\mathbf{N}_{\bar{X}/C}$ is nef Cartier on $\bar{X}/C$.
By Lemma~\ref{lem:toric-gen}, 
after possible shrinking $U_W$, 
we can perform a sequence of toric blow-ups on $\bar{X}/C$ and assume that $N\mathbf{N}_{\bar{X}/C}$ is a base point free Cartier divisor over $U_W$.
By possibly shrinking $U_W$, we may assume that $N\mathbf{N}_{\bar{X}/C}$ is base point free on the preimage of $U_W$.
Let $\Gamma$ be such that
$N\Gamma \in |N\mathbf{N}_{\bar{X}/C}|$ is a general element.
We consider the divisor 
$\Gamma^G:=\sum_{g\in G} g^*N\Gamma/|G|$.
We replace $N$ with $N|G|$ and 
$\Gamma$ with $\Gamma^G$, so
that $\Gamma$ is $G$-invariant.
Hence, the following conditions are satisfied:
\begin{enumerate}
\item[(i)] the pair $(\bar{X},\bar{B}+\bar{\Gamma})$
is a sub-log Calabi--Yau pair, 
\item[(ii)] we have $N(K_{\bar{X}}+\bar{B}+\bar{\Gamma})\sim_{U_W} 0$, 
\item[(iii)] the pair $(\bar{X},\bar{B}+\bar{\Gamma})$ is $\tilde{G}$-invariant, and 
\item[(iv)] every log canonical place of $(\bar{X},\bar{B}+\bar{\Gamma})$ is a generalized log canonical place of $(\bar{X},\bar{B},\mathbf{M}')$.
\end{enumerate}
The first two statements are clear from the construction.
The third statement follows as $\Gamma$ is $G$-invariant. 
The last statement follows from Lemma~\ref{lem:general-element}. 

Let $\bar{X}_w$ be the fiber over $w$ of $\bar{X}\rightarrow U_W$. 
The variety $\bar{X}_w$ is a projective toric variety.
We argue that there 
is a finite set of projective toric varieties $\mathcal{T}_0:=\{T_i\mid i\in I\}$, only depending on $n,N$ and $\epsilon$, such that there is a projective contraction
$\bar{X}_w \rightarrow T_i$ for some $i\in I$.
For $u\in Z'$ a general closed point, 
the log pair
$(X'_u,B'_u)$ belongs to a bounded family that only depends on $n,N$ and $\epsilon$.
Hence, we may fix a $\mathbb{T}$-equivariant projective birational morphism
$\widetilde{X}'_u \rightarrow X'_u$ for which there is a good quotient $\widetilde{X}'_u \rightarrow V$.
Here, $\mathbb{T}$ is a maximal torus of ${\rm Aut}^0(X'_u,B'_u)$. 
By construction, the general fiber of $\widetilde{X}'_u\rightarrow V$ belongs to a
bounded family that only depends on $n,N$, and $\epsilon$.
A bounded family of projective toric varieties is indeed finite (see Lemma~\ref{lem:toric-bound-finite}). 
Hence, the general fiber of $\widetilde{X}'_u\rightarrow V$ belongs to a finite set of projective toric varieties, that only depends on $n,N$ and $\epsilon$.
We define $\mathcal{T}_0$ to be the finite set of projective toric varieties that appear as general fibers of
morphisms of the form $\widetilde{X}'_u\rightarrow V$.
By further blowing up $\bar{X}$, we may assume that the general fiber of $\bar{X}\rightarrow U_W$ dominates such of $\widetilde{X}'_u\rightarrow V$.
Hence, a general closed fiber of $\bar{X}\rightarrow U_W$ admits a projective birational morphism
to a projective toric variety in $\mathcal{T}_0$.

If the generalized pair $(\widetilde{X}'_u,\widetilde{B}'_u,\mathbf{M}'_u)$
is the log pull-back of $(X'_u,B'_u,\mathbf{M}'_u)$ to $\widetilde{X}'_u$, then the coefficients of $\widetilde{B'}_u$ belong to a finite set that only depends on $n,N$, and $\epsilon$.
In particular, 
there is a finite set of toric sub-pairs
$\mathcal{T}_1:=\{ (T_i,B_{T_i}) \mid i \in I'\}$,
only depending on $n,N$ and $\epsilon$,
such that 
the restriction of 
$(\widetilde{X}'_u,\widetilde{B}'_u)$ 
to $T_i$ is isomorphic to a sub-pair in $\mathcal{T}_1$.
Let $(\bar{X}_w,\bar{B}_w+\bar{\Gamma}_w)$ be the restriction of $(\bar{X},\bar{B}+\bar{\Gamma})$ to a general fiber of $\bar{X}\rightarrow W$. 
Let $(T_i,B_{T_i}+\Gamma_{T_i})$ be the push-forward of $(\bar{X}_w,\bar{B}_w+\bar{\Gamma}_w)$
to $T_i$.
By construction, 
the sub-pair $(T_i,B_{T_i}+\Gamma_{T_i})$ is a $C$-invariant sub-log Calabi--Yau pair for which $N(K_{T_i}+B_{T_i}+\Gamma_{T_i})\sim 0$.
Since the coefficients of $B_{T_i}$ are bounded below, the sub-pairs $(T_i,B_{T_i}+\Gamma_{T_i})$ are log bounded in terms of $n,N$, and $\epsilon$.
In particular, 
there is a finite set $\mathcal{A}:=\mathcal{A}(n,N,\epsilon)$ of
algebraic groups, 
only depending on $n,N$ and $\epsilon$,
such that ${\rm Aut}(T_i,B_{T_i},\Gamma_{T_i})\in \mathcal{A}$.
Hence, up to replacing $C$ with a subgroup of index at most $m_3:=m_3(n,N,\epsilon)$, we may assume that 
\[
C
\leqslant {\rm Aut}^0(T_i,B_{T_i},\Gamma_{T_i}).
\]
Since $\bar{X}_w \rightarrow T_i$ is a  projective birational toric morphism and $C\leqslant {\rm Aut}^0(T_i)$, we have 
\[
C
\leqslant {\rm Aut}^0(\bar{X}_w,\bar{B}_w+\bar{\Gamma}_w) \leqslant \mathbb{T}_{Z'}
\]
for a general closed point $w\in U_W$.
This implies that
\[
C\leqslant {\rm Aut}^0_{U_W}(\bar{X},\bar{B}+\bar{\Gamma})\leqslant \mathbb{T}_{Z'}.
\]
In particular, ${\rm Aut}^0_{U_W}(\bar{X},\bar{B}+\bar{\Gamma})$ is a positive-dimensional split torus over $\kk$ (see Lemma~\ref{lem:aut-sub-log-CY}).
We set $\mathbb{T}_{Z',1}:= {\rm Aut}^0_{U_W}(\bar{X},\bar{B}+\bar{\Gamma})$.
The torus $\mathbb{T}_{Z',1}$ is invariant under the conjugation $G$-action as it is the connected component of ${\rm Aut}_{U_W}(\bar{X},\bar{B}+\bar{\Gamma})$
and $G\leqslant {\rm Aut}(\bar{X},\bar{B}+\bar{\Gamma})$
preserves the fibration $\bar{X}\rightarrow U_W$.

Now, it suffices to show the second statement.
Let $\mathbb{H}\leqslant \mathbb{T}_{Z',1}$ be a one-dimensional subtorus.
Let $S_0$ and $S_\infty$ be the divisorial sink and the divisorial source
of the $\mathbb{H}$-action.
We can find an $\mathbb{H}$-equivariant projective birational morphism $M\rightarrow \bar{X}$
for which $M$ admits a good quotient $M\rightarrow L$ for the $\mathbb{H}$-action.
The general fiber of $M\rightarrow L$ is isomorphic to $\pp^1$.
Let $(M,B_M+\Gamma_M)$ be the log pull-back of
$(\bar{X},\bar{B}+\bar{\Gamma})$ to $M$.
Since $(\bar{X},\bar{B}+\bar{\Gamma})$ is
$\mathbb{H}$-invariant, then the restriction of  $B_M+\Gamma_M$ to the general fiber of $M\rightarrow L$ equals $\{0\}+\{\infty\}$.
Indeed, the restriction of $(M,B_M+\Gamma_M)$ to the general fiber of $M\rightarrow L$ must be $\mathbb{H}$-invariant and sub-log Calabi--Yau.
We conclude that both $S_0$ and $S_\infty$ appear with coefficient one in $B_M+\Gamma_M$.
Indeed, both  
$M\rightarrow \bar{X}$
and $M\rightarrow L$
morphisms
are over $U_W$.
Hence, $S_0$ and $S_\infty$ are log canonical centers of $(\bar{X},\bar{B}+\bar{\Gamma})$.
By (iv), we conclude that $S_0$ and $S_\infty$ are generalized log canonical places of $(X',B',\mathbf{M}')$. 
Thus, the restriction of the divisorial sink and the divisorial source of the $\mathbb{H}$-action to $X'_z$ are generalized log canonical places
of $(X'_z,B'_z,\mathbf{M}'_z)$ for $z\in Z$ a general closed point.
This finishes the proof of the claim.
\end{proof} 

Let $\mathbb{T}_{Z',1}\leqslant \mathbb{T}_{Z'}$ be the torus provided by the first claim.
This torus is induced by a split torus $\mathbb{T}_{\eta',1}\leqslant \mathbb{T}_{\eta'}$ which is invariant under the conjugation $G$-action.
By Lemma~\ref{lem:non-split-r-torus}, up to replacing $C$ with a subgroup of index at most $m_2(r)$, 
we may assume that
\[
C\leqslant \mathbb{T}_{\eta',0} 
\leqslant \mathbb{T}_{\eta',1} 
\leqslant \mathbb{T}_{\eta'}.
\]
Here, $\mathbb{T}_{\eta',0}$ is a one-dimensional split torus over $\eta'$ that is invariant under the conjugation $G$-action.
Hence, up to shrinking $Z$ and $Z'$, we have 
\[
C\leqslant \mathbb{T}_{Z',0} \leqslant \mathbb{T}_{Z',1} \leqslant \mathbb{T}_{Z'}.
\]
Here, $\mathbb{T}_{Z',0}$ is a one-dimensional torus over $Z'$ that is invariant under the conjugation $G$-action.
Let $H_0$ be the group generated by $\mathbb{T}_{Z',0}$ and $G$,
so we have a short exact sequence:
\[
1\rightarrow \mathbb{T}_{Z',0}\rightarrow H_0 \rightarrow G \rightarrow 1.
\]
\noindent\textit{Claim 2:}
Up to shrinking $Z$, we can construct a log pair $(X',B'+\Gamma')$ that satisfies the following properties:
\begin{itemize}
\item the log pair $(X',B'+\Gamma')$ is log Calabi--Yau over $Z$, 
\item we have $N(K_{X'}+B'+\Gamma')\sim_Z 0$, and 
\item the log pair $(X',B'+\Gamma')$ is $H_0$-invariant.
\end{itemize} 

\begin{proof}[Proof of Claim 2]
Let $\mathbb{T}_{\bar{\eta},0}$ be the one-dimensional torus induced by $\mathbb{T}_{Z',0}$ on the geometric generic fiber. So, we have 
$\mathbb{T}_{\bar{\eta},0}\leqslant {\rm Aut}(X_{\bar{\eta}},B_{\bar{\eta}})$.
For each $m\geq 1$, we let $H_{0,m}$ be the subgroup of $H_0$ generated by $\zz_m$ and $G$. 
Note that $(X',B')$ is $H_{0,m}$-invariant and $\qq$-complemented over $Z$.
By~\cite[Lemma 2.14]{Mor21}, for each $m\geq 1$, we can 
shrink $Z$ and construct an effective divisor $\Gamma_m$ for which:
\begin{itemize}
\item the log pair $(X',B'+\Gamma'_m)$ is log Calabi--Yau over $Z$,
\item we have $N(K_{X'}+B'+\Gamma'_m)\sim 0$, and 
\item the log pair $(X',B'+\Gamma'_m)$ is $H_{0,m}$-invariant.
\end{itemize}

Let $(X_{\bar{\eta}},B_{\bar{\eta}}+\Gamma_{\bar{\eta},m})$ be the log Calabi--Yau pair induced on the geometric generic fiber.
By Lemma~\ref{lem:aut-sub-log-CY}, the automorphism groups ${\rm Aut}(X_{\bar{\eta}},B_{\bar{\eta}}+\Gamma_{\bar{\eta},m})$ are finite extensions of algebraic tori.
As $N(K_{X_{\bar{\eta}}}+B_{\bar{\eta}}+\Gamma_{\bar{\eta},m})\sim 0$, the log pairs $(X_{\bar{\eta}},B_{\bar{\eta}}+\Gamma_{\bar{\eta},m})$ belong to a bounded family that only depends on $n,N$, and $\epsilon$.
Thus, the algebraic groups ${\rm Aut}(X_{\bar{\eta}},B_{\bar{\eta}}+\Gamma_{\bar{\eta},m})$ form a finite set of algebraic groups only depending on $n,N$ and $\epsilon$. In particular, this set is independent of $m$.
For each $m$, we have 
$\zz_m\leqslant {\rm Aut}(X_{\bar{\eta}},B_{\bar{\eta}}+\Gamma_{\bar{\eta},m})\cap \mathbb{T}_{\bar{\eta},0}$.
Thus, for $m$ large enough, we have 
$\mathbb{T}_{\bar{\eta},0} \leqslant  {\rm Aut}(X_{\bar{\eta}},B_{\bar{\eta}}+\Gamma_{\bar{\eta},m})$.
Fix $\Gamma':=\Gamma'_m$ for some $m$ large enough.
Up to shrinking $Z$, we may assume that the pair $(X',B'+\Gamma')$ is $\mathbb{T}_{Z',0}$-invariant.
By construction, the log pair $(X',B'+\Gamma')$ is $G$-invariant. 
Therefore, the log pair $(X',B'+\Gamma')$ is $H_0$-invariant. 
\end{proof}

Now, we return to the proof of the proposition.
Let $\pi'\colon Y'\rightarrow X'$ be a $G$-equivariant projective birational morphism
that extracts the divisorial sink and the divisorial source of the $\mathbb{T}_{Z',0}$-action.
By construction, both the divisorial sink and the divisorial source are log canonical places of $(X',B',\mathbf{M}')$.
Furthermore, the divisorial sink and the divisorial source are log canonical places of $(X',B'+\Gamma')$.
By Lemma~\ref{lem:g-extraction}, we may assume that $\pi'$ only extract generalized log canonical places of $(X',B',\mathbf{M}')$.
Thus, it only extracts log canonical places of $(X',B'+\Gamma')$.
In particular, by Lemma~\ref{lem:G_m-action-lift-to-gdlt} applied to $(X',B'+\Gamma')$, the projective birational morphism
$\pi'$ is $\mathbb{T}_{Z',0}$-equivariant.
Let $S' \subset Y'$ be the sum of the divisorial sink and the divisorial source of the $\mathbb{T}_{Z',0}$-action.
Let $t\in \mathbb{T}_{Z',0}$,
$g\in G$, and $p\in S'$.
We have 
\[
g(p)= g\circ t (p) = t'\circ g(p),
\]
for some $t'\in \mathbb{T}_{Z',0}$.
Indeed, the first equality holds as $p$ is fixed by $\mathbb{T}_{Z',0}$ and the second equality holds as the $\mathbb{T}_{Z',0}$-action is invariant under the conjugation $G$-action.
So for every $p\in S'$, we have $t'\cdot g(p)=g(p)$.
Thus, the one-dimensional torus
$\mathbb{T}_{Z',0}$ acts as the identity on every prime component of $g(S')$.
We conclude that $g(S')=S'$ for every $g\in G$.
Hence, $G(S')=S'$.
This means that $G$ either preserves the divisorial sink or the divisorial source
or permutes them.

Let $Y$ be the quotient of $Y'$ by the $G$-action.
Let $S$ be the image of $S'$ on $Y$.
Hence, we have a commutative diagram
\[
\xymatrix{ 
Y\ar[d] & Y'\ar[d]\ar[l]_-{/G} \\
X & X' \ar[l]_-{/G}.
}
\]
The $\mathbb{T}_{Z',0}$-action
induces a $\mathbb{G}_m$-action 
on $X'_{z'}$ for $z'\in Z'$ a general closed point.
By construction, the divisor $S'|_{Y'_{z'}}$ is precisely the reduced sum of the divisorial sink and the divisorial source of the $\mathbb{G}_m$-action on $X'_{z'}$.
The quotient by $G$ maps fibers isomorphically. 
Hence, there is an induced $\mathbb{G}_m$-action on $X_z$
for which $S|_{Y_z}$ is the reduced sum of the divisorial
sink and the divisorial source.
Further, by construction, the prime components of $S|_{Y_z}$ are generalized log canonical places of $(X_z,B_z,\mathbf{M}_z)$.
\end{proof} 

\section{Finite abelian actions and conic fibrations}

In this section, we the study geometric consequences of the existence of large abelian actions on log Calabi--Yau pairs. 

\begin{theorem}\label{thm:fin-act-conic}
Let $n$ be a positive integer
and $\Lambda$ be a set of rational numbers satisfying the descending chain condition.
There is a constant $m_0:=m_0(n,\Lambda)$,
only depending on $n$ and $\Lambda$, satisfying the following.
Let $X$ be a $n$-dimensional variety
and $X\rightarrow W$ be a Fano type morphism.
Let $(X,B,\mathbf{M})$ be a generalized log Calabi--Yau pair.
Assume there exists a finite abelian group $A$ for which $\bigoplus_{i=1}^r \zz_{\ell_i} \leqslant A \leqslant {\rm Aut}_W(X,B,\mathbf{M})$
with $\ell_1 \mid \ell_2 \mid \dots \mid \ell_r$
and $\ell_1\geq m_0$.
Furthermore, assume that the coefficients of $B$ belong to $\Lambda$ and the coefficients of $\mathbf{M}$ belong to $\Lambda$ in a quotient by $A$ where it descends.
Then, there exists a subgroup
$A_0 \leqslant A$ of index at most $m_0/2$
satisfying the following.
There is an $A_0$-equivariant crepant birational model 
$(X',B',\mathbf{M})$ of
$(X,B,\mathbf{M})$ over $W$ for which:
\begin{enumerate}
\item the birational map
$X'\dashrightarrow X$ only extracts generalized log canonical places of $(X,B,\mathbf{M})$, 
\item there is a tower of equivariant Mori fiber spaces over $W$
\[ 
\xymatrix{
X'\ar[r]^-{\psi_0} & 
 X_1 \ar[r]^-{\psi_1} & \dots  
\ar[r]^-{\psi_{r-1}}
& X_r \ar[r]^-{\psi_r} & W
}
\]
and a subset $I\subseteq \{0,1,\dots,r\}$ with $|I|=k$ such that each $\psi_i$ with $i\in I$ is a conic fibration, and 
\item if $i\in I$, $A_i$ is the homomorphic image of $A_0$ acting on $X_i$,
and $(X_i,B_i,\mathbf{N}_i)$ is the generalized pair on $X_i$ induced by the equivariant generalized canonical bundle formula, then 
$\psi_i$ is an $A_i$-equivariant strict conic fibration for $(X_i,B_i,\mathbf{N}_i)$.
\end{enumerate}
\end{theorem}

\begin{proof}
The proof proceeds by induction on the dimension $n$ of the variety $X$. 
The case $n=1$ is simple and the details are explained in Example~\ref{ex:p1}.
We will proceed in six steps.
In Step 1, we will control the index of the generalized log Calabi--Yau pair $(X,B,\mathbf{M})$.
In Step 2, we take an equivariant gdlt modification and run an equivariant minimal model program that terminates with an equivariant Mori fiber space. 
In Step 3, we analyze the group acting on the fibers of the Mori fiber space 
and the regularity of the generalized log Calabi--Yau pair induced on such fibers. 
In Steps 4-5, we prove the statement of the theorem provided that the group acting fiberwise is large enough compared to $n$ and $\Lambda$.
In Step 6, we prove the statement of the theorem when the order of the group acting on the fibers is bounded in terms of $n$ and $\Lambda$.\\

\textit{Step 1:} In this step, we show that $I(K_X+B+\mathbf{M}_X)\sim 0$ for some $I:=I(n,\Lambda)$.\\

By~\cite[Theorem 1.6]{BZ16}, there exists a finite subset $\mathcal{F}\subseteq \Lambda$ such that the coefficients of $B$ belong to $\mathcal{F}$ 
and the coefficients of $\mathbf{M}$ belong to $\mathcal{F}$ in a quotient where it descends.
Hence, by~\cite[Theorem 1.2]{FM20}, we conclude that $I(K_X+B+\mathbf{M}_X)\sim 0$ for some $I$ that only depends on $n$ and $\mathcal{F}$.
In particular, the positive integers $I$ only depends on $n$ and $\Lambda$.\\ 

\textit{Step 2:} In this step, we take an equivariant dlt modification of $(X,B,\mathbf{M})$ and run an equivariant minimal model program.\\

Let $(X_0,B_0,\mathbf{M})$ be an $A$-equivariant crepant model of $(X,B,\mathbf{M})$ that admits an $A$-equivariant fibration $X_0\rightarrow Z_1$ over $W$ of minimal relative dimension. 
By the first step, we have $I(K_{X_0}+B_0+\mathbf{M}_{X_0})\sim 0$.
Let $(X_1,B_1,\mathbf{M})$ be an $A$-equivariant generalized dlt modification
of $(X_0,B_0,\mathbf{M})$.
The variety $X_1$ is $\frac{1}{I}$-lc. 
We run an $A$-equivariant MMP for $K_{X_1}$ over $Z_1$. 
This MMP terminates with an $A$-equivariant Mori fiber space $X_2\rightarrow Z$ over $Z_1$.
As $X_0\rightarrow Z_1$ is chosen to be an equivariant fibration of minimal relative dimension, we conclude that $\dim Z_1=\dim Z$. Let $B_2$ be the push-forward of $B_1$ on $X_2$. Then, the generalized pair
$(X_2,B_2,\mathbf{M})$ is a generalized log Calabi--Yau pair and $X_2\rightarrow Z$ is an equivariant Mori fiber space.
Let $A_f \leqslant A$ be the subgroup acting fiberwise on $X_2\rightarrow Z$ and $A_Z$ the homomorphic image of $A$ acting on $Z$.
We may replace $(X,B,\mathbf{M})$ with $(X_2,B_2,\mathbf{M})$ and assume that $X$ itself admits an $A$-equivariant Mori fiber space
$X\rightarrow Z$. Further, the general fiber of $X\rightarrow Z$ is an $A_f$-equivariant $\frac{1}{I}$-lc Fano variety of dimension at most $n$. Let $z\in Z$ be a general closed point and let
$(X_z,B_z,\mathbf{M}_z)$ be the generalized pair induced by adjunction to a general fiber of $X\rightarrow Z$. In what follows, we will analyze the equivariant geometry of the generalized log Calabi--Yau pair $(X_z,B_z,\mathbf{M}_z)$.\\

\textit{Step 3:} In this step, we show that either 
$|A_f|\leq m_1(n,I,\epsilon)$ or
${\rm reg}(X_z,B_z,\mathbf{M}_z)=0$.\\

Assume that ${\rm reg}(X_z,B_z,\mathbf{M}_z)\geq 1$.
This condition implies that the dimension of $X_z$ is at least $2$.
Let $U\subseteq Z$ be an open set for which
there is an $A$-invariant $N$-complement
$\Gamma_U$ of $(X_U,B_U)$ over $Z$.
Here, we may choose $N$ so it only depends on $n$ and $\Lambda$ (see~\cite[Theorem 1.2]{FM20}).
Further, the pair $(X_U,B_U)$ is the restriction of $(X,B)$ to the preimage of $U$.
We may assume that ${\rm reg}(X_z,\Gamma_z)\geq {\rm reg}(X_z,B_z,\mathbf{M}_z)\geq 1$.
Let $m_2:=m_2(n,N,I^{-1})$ be as in Proposition~\ref{lem:finite-action-conic-bundle}.
Assume that $|A_f|>m_1(n,I,\epsilon):=m_2^n$.
Since $A_f$ is a finite abelian group of rank at most $n$, we conclude that there is a cyclic subgroup $C\leqslant A_f$ of order at least $m_2$.
Note that $C\leqslant {\rm Aut}(X_U,\Gamma_U)$ is acting fiberwise over $U$, i.e., we have $C\leqslant {\rm Aut}_U(X_U,\Gamma_U)$.
By Theorem~\ref{thm:cyc-act-Fano-fib}, 
there exists a subgroup $C_0\leqslant C$ of index at most $\frac{m_2}{2}$ satisfying the following.
There exist:
\begin{itemize} 
\item a projective birational morphism $Y\rightarrow X$ only extracting generalized log canonical places of $(X,B,\mathbf{M})$, and
\item a reduced effective divisor $S$ on $Y$
\end{itemize}
with the following property. For a general closed point $z\in Z$,
we can find a one-dimensional torus $C_0\leqslant \mathbb{T}_z \leqslant {\rm Aut}(X_z,\Gamma_z)$ for which $S|_{Y_z}$ is the reduced sum of the divisorial sink and the divisorial source for the $\mathbb{T}_z$-action.
Each component of $S$ is a log canonical place of both $(X_U,\Gamma_U)$ and $(X_U,B_U,\mathbf{M}_U)$.
Indeed, each component of $S_z$ is a log canonical center of both $(X_z,\Gamma_z)$ and 
$(X_z,B_z,\mathbf{M}_z)$ by Theorem~\ref{thm:cyc-act-Fano-fib}.
By Theorem~\ref{thm:action-dual-comp},
there exists a subgroup $A'\leqslant A$ of index at most $m_3(n,\Lambda)$ such that $A'$ acts trivially on $\mathcal{D}(X,B,\mathbf{M})$.
Hence, for each $i\in \{1,\dots,k\}$
there exists $\ell_{i,0} \mid \ell_i$ 
with $\ell_i \ell_{i,0}^{-1} \leq m_3(n,\Lambda)$ 
such that $A' \geqslant \bigoplus_{i=1}^k \zz_{\ell_i\ell_{i,0}^{-1}}$ acts
trivially on $\mathcal{D}(X,B,\mathbf{M})$.
We may replace $A$ with $A'$ and
$\bigoplus_{i=1}^k \zz_{\ell_i}$
with $\bigoplus_{i=1}^k \zz_{\ell_i\ell_{i,0}^{-1}}$.
After this replacement, 
we may 
assume
that $A$ acts trivially on $\mathcal{D}(X,B,\mathbf{M})$.
By Lemma~\ref{lem:g-extraction}, 
there exists an $A$-equivariant  
crepant birational modification $(X'_U,B'_U,\mathbf{M}_U)$
of $(X_U,B_U,\mathbf{M}_U)$
that extracts precisely
the divisorial valuations corresponding to the prime components of $S$.
We denote by $S'$ the image of $S$ on $X'_U$.
Further, each component of $S'$ is invariant under the $A$-action. 
Let $z\in U$ be a general closed point.
By Lemma~\ref{lem:G_m-action-lift-to-gdlt}, the variety $X'_z$ admits a $\mathbb{G}_m$-action. Further, by construction, the divisor $S'|_{X'_z}$ is the reduced sum of the divisorial sink and the divisorial source of this action.
Indeed, the crepant birational morphism
$X'_z\rightarrow X_z$ only extracts log canonical places of the $\mathbb{G}_m$-invariant log pair $(X_z,\Gamma_z)$. So, the morphism $X'_z\rightarrow X_z$ is $\mathbb{G}_m$-equivariant (see Lemma~\ref{lem:G_m-action-lift-to-gdlt}).
By Lemma~\ref{lem:iit-dim}, we have 
\[
0 < k(B'_U + \mathbf{M}_{X'_U}-S'/U)< \dim X'_z.
\]
In particular, there exists an $A$-equivariant crepant birational model of $(X_U,B_U,\mathbf{M}_U)$ that admits an equivariant fibration of relative dimension strictly less than $\dim X_z$.
The previous statement is obtained by running an $A$-equivariant $(B'_U+\mathbf{M}_{X'_U}-S')$-MMP over $U$ and taking its ample model over $U$.
Hence, by Lemma~\ref{lem:extending-fibration}, 
we conclude that there is an $A$-equivariant crepant birational model 
of $(X,B,\mathbf{M})$ that admits an equivariant fibration
of relative dimension less than such of $X\rightarrow Z$. This contradicts the minimality of the fibration $X\rightarrow Z$. We conclude that either $|A_f|\leq m_1(n,I,\epsilon)$ or ${\rm reg}(X_z,B_z,\mathbf{M}_z)$.\\

\textit{Step 4:} In this step, we show that 
$X\rightarrow Z$ is a conic fibration provided that $|A_f|>m_1(n,I,\epsilon)$ and ${\rm reg}(X_z,B_z,\mathbf{M}_z)=0$.\\

In this step, we assume that
${\rm reg}(X_z,B_z,\mathbf{M}_z)=0$
for a general closed point $z\in U$
and that 
$|A_f|>m_1(n,I,\epsilon)$.
By Theorem~\ref{thm:cyc-act-Fano-fib}, up to replacing $m_1(n,I,\epsilon)$, we conclude that $\mathcal{D}(X_z,B_z,\mathbf{M}_z)$ is two points.
Let $(Y,B_Y,\mathbf{M})$ be an $A$-equivariant dlt modification of $(X,B,\mathbf{M})$.
Then, for $z\in U$ general, the divisor $\lfloor B_{Y,z}\rfloor$ is the sum
of two disjoint reduced divisors.
We run an $A$-equivariant 
$K_Y$-MMP over $Z$.
This MMP must terminate with an equivariant conic fibration over $Z$.
By the minimality assumption
on the relative dimension of $X\rightarrow Z$, we conclude that $X\rightarrow Z$ is itself a conic fibration over $W$.\\

\textit{Step 5:} In this step, we prove the statement of the theorem
when $X\rightarrow Z$ is a conic fibration
provided $|A_f| > m_1(n,I,\epsilon)$.\\

Since $|A_f|>m_1(n,I,\epsilon)$ and 
$X\rightarrow Z$ is an $A$-equivariant conic fibration, we conclude $(X_z,B_z,\mathbf{M}_z)\simeq (\pp^1,\{0\}+\{\infty\})$.
In particular, $A_f$ is a cyclic group.
By Proposition~\ref{lem:finite-action-conic-bundle}, we know that $\lfloor B\rfloor$ has two sections over $Z$ provided
that $|A_f|>m_1(n,I,\epsilon)$.
By Lemma~\ref{lem:making-sections-disjoints}, up to an equivariant crepant birational transformation that only extracts log canonical places, we may assume that the sections are disjoint.
Let $(Z,B_Z,\mathbf{N})$ be the generalized log Calabi--Yau pair
obtained by the equivariant generalized canonical bundle formula (see Proposition~\ref{prop:equiv-gen-cbf}).
There exists a set of rational numbers $\Omega$ with the DCC property satisfying the following.
The coefficients of $B_Z$ belong to $\Omega$
and the coefficients of $\mathbf{N}$ belong to $\Omega$ in a quotient by $A_Z$ where it descends.
As $A_f$ is cyclic, we have that 
$\bigoplus_{i=1}^{k-1} \zz_{b_i} \leqslant A_Z \leqslant {\rm Aut}(Z,B_Z,\mathbf{N})$,
where $b_1\mid \dots \mid b_{k-1}$
and $b_1\geq \ell_1$.
By Lemma~\ref{lem:FT-under-morphisms}.(1), we have that $Z$ is of Fano type over $W$.
Thus, provided
that 
$\ell_1 \geq m_0(\dim Z,\Omega)$, 
we may assume that the conditions of the theorem
hold for the generalized 
log Calabi--Yau pair $(Z,B_Z,\mathbf{N})$.
By the induction hypothesis, 
there exists a subgroup $A_{Z,0}\leqslant A_Z$ of index at most $m_0(\dim Z,\Omega)/2$
satisfying the following.
There is an $A_{Z,0}$-equivariant crepant birational model 
$(Z',B_{Z'},\mathbf{N})$
of $(Z,B_Z,\mathbf{N})$ over $W$
that admits a tower of Mori fiber spaces
\[
\xymatrix{
Z' \ar[r]^-{\phi_1} & Z_1 
\ar[r]^-{\phi_2} & \dots \ar[r]^-{\phi_{r}} & Z_r \ar[r]^-{\phi_{r+1}} & W.
}
\]
By induction, we know that there exists 
a subset $I\subseteq \{1,\dots,r+1\}$ such that
$|I|=k-1$ and each $\psi_i$ with $i\in I$ is a conic fibration.
Furthermore, if $i\in I$, $A_{Z,i}$ is the homomorphic image of $A_{Z,0}$ acting on $Z_i$,
and $(Z_i,B_{Z_i},\mathbf{N}_i)$ is induced by the generalized canonical bundle formula on $Z_i$, then $\lfloor B_{Z_i}\rfloor$ admits two disjoint $A_{Z,i}$-invariant prime components that are horizontal over $Z_{i+1}$.
Let $A_0\leqslant A$ be the preimage of $A_{Z,0}$ in $A$.
Hence $A_0$ has index at most
$m_0(\dim Z,\Omega)/2$ in $A$.
By Proposition~\ref{prop:2-ray}, 
there exists an $A_0$-equivariant 
crepant birational model $(X',B',\mathbf{M})$ 
that admits an equivariant conic fibration $\phi_0$ onto $Z'$.
By Proposition~\ref{prop:2-ray}.(5), we have the following commutative diagram:
\[
\xymatrix{
(X,B,\mathbf{M}) \ar[d] &  (X',B',\mathbf{M}) \ar@{-->}[l] \ar[d]^-{\phi_0} \\
(Z,B_Z,\mathbf{N}) & (Z',B_{Z'},\mathbf{N})\ar@{-->}[l] \ar[d]^-{\phi_1} \\
& \vdots \ar[d]^-{\phi_{r+1}}  \\ 
& W \\
}
\]
By construction, the divisor $\lfloor B'\rfloor$ admits two disjoint $A_0$-invariant prime divisors over $Z'$.
Hence, it suffices to take $I':=I\cup \{0\}$. 
This finishes the proof of the theorem provided that $|A_f|>m_1(n,I,\epsilon)$.\\

\textit{Step 6:} In this step, we prove the statement of the theorem when $|A_f|\leq m_1(n,I,\epsilon)$.\\

The proof of this step is very similar to the proof of the previous step.
We spell out the details for the sake of completeness.

Since $|A_f|\leq m_1(n,I,\epsilon)$, 
we have that 
$\bigoplus_{i=1}^k \zz_{b_i} \leqslant A_Z$
where $b_1\mid \dots \mid b_k$ and 
$b_1 \geq \ell_1/m_1$.
Let $(Z,B_Z,\mathbf{N})$ be the generalized
log Calabi--Yau pair obtained by the equivariant generalized 
canonical bundle formula (see Proposition~\ref{prop:equiv-gen-cbf}).
Then, we have that $A_Z \leqslant {\rm Aut}_W(Z,B_Z,\mathbf{N})$.
There exists a set of rational numbers $\Omega$  
with the DCC property satisfying the following.
The coefficients of $B_Z$ 
belong to $\Omega$
and the coefficients of $\mathbf{N}$
belong to $\Omega$ in a quotient by $A_Z$ where it descends.
Furthermore, observe that $Z$ is of Fano type over $W$ by Lemma~\ref{lem:FT-under-morphisms}.(1).
Provided that $\ell_1 \geq m_1 m_0(\dim Z, \Omega)$, 
we may assume that the statement
of the theorem holds
for the generalized log Calabi--Yau pair
$(Z,B_Z,\mathbf{N})$.
By the induction hypothesis,
there exists a subgroup $A_{Z,0}\leqslant A_Z$ 
of index at most $m_0(\dim Z,\Omega)/2$ satisfying
the following.
There is an $A_{Z,0}$-equivariant crepant birational model $(Z',B_{Z'},\mathbf{N})$ of $(Z,B_Z,\mathbf{N})$ over $W$ that admits a tower of Mori fiber spaces
\[
\xymatrix{
Z' \ar[r]^-{\phi_1} & Z_1 
\ar[r]^-{\phi_2} & \dots \ar[r]^-{\phi_{r}} & Z_r \ar[r]^-{\phi_{r+1}} & W.
}
\]
By the induction hypothesis,
there exists a subset $I\subseteq \{1,\dots,r+1\}$
such that $|I|=k$ 
and each $\psi$ with $i\in I$ is a conic fibration.
Furthermore, if $i\in I$, $A_{Z,i}$ is the homomorphic image of $A_{Z,0}$ acting on $Z_i$,
and $(Z_i,B_{Z_i},\mathbf{N}_i)$ is
induced by the generalized canonical bundle formula,
then $\lfloor B_{Z_i}\rfloor$ admits two disjoint
$A_{Z,i}$-invariant prime components
that are horizontal over $Z_{i+1}$.
Let $A_0\leqslant A$ be the preimage
of $A_{Z,0}$ in $A$.
By construction,
$A_0$ has index at most $m_0(\dim Z,\Omega)/2$ in $A$.
By Proposition~\ref{prop:2-ray}, there exists an $A_0$-equivariant crepant birational model
$(X',B',\mathbf{M})$ that admits an equivariant Mori fiber space
$\phi_0$ onto $Z'$.
Thus, we have the following commutative diagram:
\[
\xymatrix{
(X,B,\mathbf{M}) \ar[d] &  (X',B',\mathbf{M}) \ar@{-->}[l] \ar[d]^-{\phi_0} \\
(Z,B_Z,\mathbf{N}) & (Z',B_{Z'},\mathbf{N})\ar@{-->}[l] \ar[d]^-{\phi_1} \\
& \vdots \ar[d]^-{\phi_{r+1}}  \\ 
& W \\
}
\]
Hence, it suffices to take $I\subseteq \{1,\dots,r+1\}$
as defined above.
This finishes the proof of the theorem.
\end{proof}

\section{Proof of the theorems}
\label{sec:proofs}
In this section, we prove the theorems of this article.

\subsection{Finite actions on rationally connected varieties}
In this subsection, we prove a geometric Jordan property for finite actions on rationally connected varieties. 

\begin{proof}[Proof of Theorem~\ref{thm:finite-actions}]
We proceed by induction on the dimension.
Let $n$ be the dimension of $X$.
The statement for $n=1$ by using the Jordan property for $\mathbb{P}{\rm GL}(2,\kk)$. 
By using the Jordan property~\cite{PS14}, we may assume that $G$ is indeed abelian. 
Let $X\dashrightarrow X'$ be a $G$-equivariant birational map 
with a $G$-equivariant Fano type morphism 
$X'\rightarrow Z$ for which $\dim(X'/Z)$ is minimized.
Let $G_f<G$ be the normal subgroup acting fiberwise over $Z$.
Note that $X'$ admits an $G$-invariant $N$ complement $(X',B')$ over $Z$, so we are in the situation of Theorem~\ref{thm:fin-act-conic}.
Let $m_0(n,\frac{1}{N}\zz\cap [0,1])$ be as in the statement of Theorem~\ref{thm:fin-act-conic}.
If $|G_f|\leq m_0(n,\frac{1}{N}\zz\cap [0,1])$, 
then we may replace $G$ with $G^{|G_f|}$ and the statement follows by induction on the dimension.
If $|G_f|>m_0(n,\frac{1}{N}\zz\cap[0,1])$, then we can apply Theorem~\ref{thm:fin-act-conic} with $W=Z$~\footnote{The proof applies as long as we have a relative complement over $W$.}.
Let $G_0\leqslant G$ be the normal subgroup of bounded index given by Theorem~\ref{thm:fin-act-conic}. Here, the index is bounded in terms of $n$ and $N$.
There are two possibilities:
\begin{enumerate}
    \item[(i)] If $\psi_0$ is a strict conic fibration, then $X'\simeq_{\rm bir} \pp^1 \times X_0$ is an $G_0$-equivariant isomorphism.
In this case the statement follows by induction on the dimension.
    \item[(ii)] If $\psi_0$ is not a strict conic fibration.
\end{enumerate} 
If (ii) holds, then we replace $G$ with $G_0$ so $\dim(X'/Z)$ drops. 
We can iterate the argument and (ii) can only hold finitely many times as in each iteration we replace $A$ with a subgroup of bounded index and the dimension of $X'/Z$ drops. 
Thus, eventually, we are in the situation of (i) so the statement follows from induction on the dimension.
\end{proof} 

\subsection{Complexity measures}
In this subsection, we prove the results regarding the non-negativity of the complexity measures and characterizations of the zero case. 
First, we show that the alteration complexity is non-negative.

\begin{theorem}\label{theorem:alt-comp-noneg}
Let $(X,B)$ be a log Calabi--Yau pair.
Then, we have 
\[
c(X,B) \geq 
c_{\rm bir}(X,B) \geq 
c_{\rm alt}(X,B)\geq 0.
\]
\end{theorem}

\begin{proof}[Proof of Theorem~\ref{theorem:alt-comp-noneg}]
Let $(X,B)$ be a log Calabi--Yau pair.
The inequalities 
\[
c(X,B) \geq c_{\rm bir}(X,B) \geq 
c_{\rm alt}(X,B)
\]
follow from the definitions.
It suffices to check that the alteration complexity is non-negative.
However, if $c_{\rm alt}(X,B)<0$, then we can find a log Calabi--Yau pair $(Y,B_Y)$, obtained from $(X,B)$ by an alteration, such that $c(Y,B_Y)<0$. This contradicts the non-negativity of the complexity (see~\cite[Corollary 1.3]{BMSZ18}).
\end{proof}

We recall the following theorem due to Mauri and the author
characterizing log Calabi--Yau pairs
of index one and birational complexity zero.

\begin{theorem}\label{theorem:bir-comp-zero}
Let $(X,B)$ be a log Calabi--Yau pair of dimension $n$ and index one. Then, we have that $c_{\rm bir}(X,B)=0$ if and only if $(X,B)$ is crepant birational equivalent to $(\mathbb{P}^n,H_0+\dots+H_n)$.
\end{theorem} 

If we do not impose the condition $K_X+B\sim 0$ in the previous theorem, then it is not true that $(X,B)$ is crepant birational equivalent to a log Calabi--Yau pair in $\pp^n$ (see Example~\ref{ex:not-index-1}). However, $(X,B)$ is still crepant birational equivalent to a log Calabi--Yau pair $(T,B_T)$ where $T$ is a projective toric variety.

Now, we turn to show that the alteration complexity is indeed computed by a log Calabi--Yau pair. We will need the following lemma.

\begin{lemma}\label{lem:finite-act-regularization}
Let $(X,B)$ be a log Calabi--Yau pair.
Let $(Y,\Delta)\rightarrow (X,B)$
be a crepant finite Galois morphism. 
Assume that $(Y,\Delta)$ is crepant birational equivalent to a log Calabi--Yau pair $(Y_0,\Delta_0)$. 
Then, there exists:
\begin{enumerate}
\item[(i)] a log Calabi--Yau pair $(X',B')$ that is crepant birational equvialent to $(X,B)$, and 
\item[(ii)] a log Calabi--Yau pair $(Y',\Delta')$ that admits a crepant finite Galois morphism to $(X',B')$
\end{enumerate} 
such that $(Y',\Delta')$ is crepant birational equivalent to $(Y_0,\Delta_0)$.
\end{lemma} 

\begin{proof}
Let $f\colon (Y,\Delta)\rightarrow (X,B)$ be the crepant finite Galois morphism.
Note that $(Y,\Delta)$ is a sub-log Calabi--Yau pair that is crepant birational equivalent to $(Y_0,\Delta_0)$.
A priori, the divisor $\Delta$ may have negative coefficients.
Let $G$ be the finite group acting on $(Y,\Delta)$.
We may replace $(Y,\Delta)$ with a $G$-equivariant resolution of singularities
and assume that $Y\dashrightarrow Y_0$
is a birational contraction.
Consider the boundary $\Gamma$ on $Y$
that is obtained from $\Delta$ by increasing all its negative coefficients to $1$.
Then, the pair $(Y,\Gamma)$ is a $G$-invariant log pair.
Furhtermore, we can write 
\begin{equation}\label{eq:eff-div} 
K_Y+\Gamma \sim_\qq F \geq 0.
\end{equation} 
The divisor $F$ is effective and exceptional over $Y_0$.
By Lemma~\ref{lem:negativity}, we have 
$\supp F \subseteq {\rm Bs}_{-}(F)$.
We run a $G$-equivariant $(K_Y+\Gamma)$-MMP.
Due to~\eqref{eq:eff-div} and the
fact that ${\rm supp}(F)\subseteq {\rm Bs}_{-}(F)$ this MMP terminates after contracting all the components of $F$.
Let $Y\dashrightarrow Y'$ be the corresponding $G$-equivariant birational map.
Let $\Delta'$ be the push-forward of $\Delta$ to $Y'$.
Then, the log pair $(Y',\Delta')$ is log Calabi--Yau and crepant birational equivalent to $(Y,\Delta)$.
In particular, the log Calabi--Yau pair $(Y',\Delta')$ is crepant birational equivalent to $(Y_0,\Delta_0)$.
Let $Y'\rightarrow X'$ be the quotient by $G$.
Let $(X',B')$ be the quotient by $G$ of the log pair $(Y',\Delta')$.
Then, the log Calabi--Yau pair $(X',B')$ is crepant birational equivalent to $(X,B)$.
\end{proof}

The following theorem shows that the alteration complexity is indeed computed by a log Calabi--Yau pair.

\begin{theorem}\label{theorem:alt-comp-computation}
Let $(X,B)$ be a log Calabi--Yau pair.
Then, there exist:
\begin{enumerate}
\item[(i)] a log Calabi--Yau pair $(X',B')$ which is crepant birational equivalent to $(X,B)$, and 
\item[(ii)] a log Calabi--Yau pair $(Y,B_Y)$ 
with a crepant finite Galois morphism $(Y,B_Y)\rightarrow (X',B')$,
\end{enumerate}
such that $c_{\rm alt}(X,B)=c_{\rm bir}(Y,B_Y)$. 
\end{theorem} 

\begin{proof}[Proof of Theorem~\ref{theorem:alt-comp-computation}]
Let $(X,B)$ be a log Calabi--Yau pair of dimension $n$.
First, we argue that there exists a crepant finite Galois morphism
$f\colon (Y_0,B_{Y_0})\rightarrow (X,B)$
such that $(Y_0,B_{Y_0})$ is sub-log Calabi--Yau
and $c_{\rm bir}(Y_0,B_{Y_0})=c_{\rm alt}(X,B)$.

A priori, the alteration complexity of a log Calabi--Yau pair $(X,B)$ is an infimum.
Let $f_i\colon (Y_i,B_i)\rightarrow (X,B)$ be a sequence of crepant finite Galois morphisms 
for which $c_{\rm bir}(Y_i,B_i)$ converges to the alteration complexity of $(X,B)$.
Note that each $(Y_i,B_i)$ is sub-log Calabi--Yau.
By definition of the birational complexity, we know that $(Y_i,B_i)$ is crepant birational
to a log Calabi--Yau pair $(Y'_i,B'_i)$.

We argue that the coefficients of $B'_i$ 
belong to a finite set $\mathcal{B}$
that is independent of $i$. We show that the set $\mathcal{B}$ only depends on the initial log Calabi--Yau pair $(X,B)$.
Indeed, a prime divisor $P$ in the support of $B'_i$ induces a non-terminal valuation of
$(Y_i,B_i)$.
By Riemann-Hurwitz, we have 
\begin{equation}\label{eq:non-term-val} 
a_P(Y_i,B_i)=
ra_{f_i(P)}(X,B)
\end{equation} 
where $r$ is a positive integer. 
In particular, $f_i(P)$ is a non-terminal valuation of $(X,B)$.
By~\cite[Corollary 2.13]{KM98}, there exists a finite set $\mathcal{B}_0$ such that every non-terminal valuation $Q$ over $(X,B)$ satisfies $a_Q(X,B)\in \mathcal{B}_0$.
Hence, $a_{f_i(P)}(X,B) \in \mathcal{B}_0$.
Since $a_P(Y_i,B_i)<1$, by equality~\eqref{eq:non-term-val}, we conclude that $a_P(Y_i,B_i)$ belongs to a finite set that only depends on $\mathcal{B}_0$. 
In particular, there is a lower bound $\epsilon_0>0$ for the coefficients of $B'_i$
which is independent of $i$.

By~\cite[Lemma 2.30]{MM24}, for each $i$, there exists a crepant birational model
$(Y''_i,B''_i)$ of $(Y'_i,B'_i)$ and a fibration $Y''_i\rightarrow W_i$, satisfying the following conditions:
\begin{enumerate}
\item the variety $Y''_i$ is $\qq$-factorial,
\item every divisor extracted by
$Y''_i\dashrightarrow Y'_i$ is a log canonical place of $(Y'_i,B'_i)$,
\item every component of $B''_i$ dominates $W_i$, and  
\item the relative Picard rank $\rho(Y''_i/W_i)$ is bounded above by $n$.
\end{enumerate} 
By~\cite[Theorem 2.25]{MM24}, condition (3), and condition (4), we conclude that $B_i''$ has at most $\lfloor 2\epsilon_0^{-1}n\rfloor$ components. 
By condition (1) and condition (2), we have 
\begin{equation}\label{eq:comp}
c(Y'_i,B'_i) \geq c(Y''_i,B''_i).
\end{equation} 
In particular, the sequence 
$c(Y''_i,B''_i)$ must converge to $c_{\rm alt}(X,B)$.
However, observe that 
\[
c(Y''_i,B''_i)= n + \rho(Y''_i) - |B''_i|. 
\]
Since the number of components of $B''_i$ is bounded above 
and the coefficients of $B''_i$ belong to $\mathcal{B}_0$, 
we conclude that 
there are only finitely many possible values for $|B''_i|$.
The sequence $\rho(Y''_i)$ can either stabilize or diverge.
The upper bound given in~\eqref{eq:comp} implies that $\rho(Y''_i)$ stabilizes.
Hence, up to passing to a subsequence, we may assume that $c(Y''_i,B''_i)$ stabilizes
and equals $c_{\rm alt}(X,B)$.
We set $(Y_0,B_{Y_0}):=(Y_i,B_i)$ for $i$ large enough.

Now, assume that there is a crepant finite Galois morphism $f\colon (Y_0,B_{Y_0})\rightarrow (X,B)$ such that 
$(Y_0,B_{Y_0})$ is sub-log Calabi--Yau and 
$c_{\rm bir}(Y_0,B_{Y_0})=c_{\rm alt}(X,B)$.
Since $c_{\rm alt}(X,B)$ is finite,
the sub-log Calabi--Yau pair
$(Y_0,B_{Y_0})$ must be crepant birational equivalent to a log Calabi--Yau pair.
By Lemma~\ref{lem:finite-act-regularization}, 
there exists:
\begin{itemize}
\item a log Calabi--Yau pair $(X',B')$
which is crepant birational equivalent to $(X,B)$, and 
\item a log Calabi--Yau pair $(Y,B_Y)$ with a crepant finite Galois morphism
$(Y,B_Y)\rightarrow (X',B')$,
\end{itemize} 
such that $(Y,B_Y)$ is crepant birational equivalent to $(Y_0,B_{Y_0})$.
In particular, we have 
$c_{\rm bir}(Y,B_Y)=c_{\rm bir}(Y_0,B_{Y_0})=c_{\rm alt}(X,B)$. 
\end{proof}

Applying Theorem~\ref{theorem:bir-comp-zero} and Theorem~\ref{theorem:alt-comp-computation}, we obtain a characterization of log Calabi--Yau pairs
with index one and 
alteration complexity zero.

\begin{proof}[Proof of Theorem~\ref{theorem:alt-comp-zero}]
Let $(X,B)$ be a log Calabi--Yau pair of dimension $n$ and index one. 

First, assume that $c_{\rm alt}(X,B)=0$.
By Theorem~\ref{theorem:alt-comp-computation}, there exists:
\begin{itemize}
\item a log Calabi--Yau pair $(X',B')$ 
which is crepant birational to $(X,B)$, and 
\item a log Calabi--Yau pair $(Y,B_Y)$ 
with a crepant finite Galois morphism $(Y,B_Y)\rightarrow (X',B')$, 
\end{itemize}
such that $c_{\rm bir}(Y,B_Y)=c_{\rm alt}(X,B)=0$.
Note that $K_{X'}+B'\sim 0$ as 
the index of a log Calabi--Yau pair is a birational invariant.
On the other hand, since 
$(Y,B_Y)\rightarrow (X',B')$ is a finite crepant morphism, we have 
$K_Y+B_Y\sim 0$.
Thus, we have that $K_Y+B_Y\sim 0$ and $c_{\rm bir}(Y,B_Y)=0$.
By Theorem~\ref{theorem:bir-comp-zero}, we conclude that $(Y,B_Y)$ is crepant birational equivalent to $(\pp^n,H_0+\dots+H_n)$.

Now, assume that there exists:
\begin{itemize}
\item a log Calabi--Yau pair $(X',B')$
which is crepant birational equivalent 
to $(X,B)$, and 
\item a log Calabi--Yau pair $(Y,B_Y)$ with a crepant finite Galois morphism
$(Y,B_Y)\rightarrow (X',B')$, 
\end{itemize}
such that $(Y,B_Y)$ is crepant birational equivalent to
$(\pp^n,H_0+\dots+H_n)$.
Then, the crepant finite Galois morphism 
$(Y,B_Y)\rightarrow (X',B')$ induces a crepant finite Galois morphism
$(Y',B_{Y'})\rightarrow (X,B)$.
Hence,
the sub-pair $(Y',B_{Y'})$ is crepant birational equivalent to $(Y,B_Y)$.
We have that $c_{\rm alt}(X,B)\leq c_{\rm bir}(Y',B_{Y'})=c_{\rm bir}(Y,B_Y) \leq c(\pp^n,H_0+\dots+H_n)=0$. 
Thus, we have $c_{\rm alt}(X,B)\leq 0$.
On the other hand, by Theorem~\ref{theorem:alt-comp-noneg}, we know that the alteration complexity is non-negative.
This finishes the proof.
\end{proof} 

\begin{proof}[Proof of Theorem~\ref{theorem:conic-complexity-zero}]
Let $(X,B)$ be a log Calabi--Yau pair with conic complexity zero.
By definition, we have a conic fibration $f\colon (X,B)\rightarrow Z$.
Further, we have two prime divisors $S_1$ and $S_2$ in $\lfloor B\rfloor$, horizontal over $Z$, such that 
$S_1\cap S_2=\emptyset$.
Since $S_1\rightarrow Z$ is generically one-to-one, we conclude that the moduli divisor of the canonical bundle formula is $\qq$-trivial.
Let $(Z,B_Z)$ be the log pair induced by the canonical bundle formula.
By induction on the dimension, we may assume that $Z$ is a Bott $\qq$-tower
and $B_Z$ is the torus invariant divisor.

We argue that $X$ is a toric variety and $B$ is the torus invariant divisor.
Let $P$ be a prime component of $B_Z$.
We show that the divisor $f^*P$ appears with coefficient one in $B$.
Indeed, since $P$ is a log canonical center of $(Z,B_Z)$, there is a codimension two subvariety $P_X$ on $X$ satisfying the following:
\begin{itemize}
\item the subvariety $P_X$ maps to $P$ under $f$, and
\item the subvariety $P_X$ is either contained in $S_1$ or $S_2$.
\end{itemize} 
If $f^*P$ is not a log canonical center of $(X,B)$, then we get a contradiction of the connectedness theorem~\cite[Theorem 1.2]{FS20}.
Hence, the divisor $f^*P$ appears with coefficient one in $B$.
Since $Z$ is a Bott $\qq$-tower, we have that $\rho(Z)=\dim Z = n-1$
and $B_Z$ has $2n-2$ prime components.
We conclude that $B$ has $2n$ prime components and $\rho(X)=\dim X = n$.
Henceforth, the pair $(X,B)$ is toric
by~\cite[Theorem 1.2]{BMSZ18}.

We conclude that $X\rightarrow Z$ is a toric Mori fiber space of relative dimension one.
Hence, $X$ is the projectivization of a rank $2$ torus invariant split $\qq$-vector bundle on $Z$.
This implies that $X$ is a Bott $\qq$-tower provided that $Z$ is a Bott $\qq$-tower. 
\end{proof} 

\begin{proof}[Proof of Theorem~\ref{introcor:bir-alt-conic-comp-zero}]
Let $(X,B)$ be a log Calabi--Yau pair of dimension $n$. 

We prove the first statement.
Assume that $(X,B)$ is crepant birational equivalent to $(\pp^n,H_0+\dots+H_n)$. 
Then, $(X,B)$ is crepant birational equivalent to 
$((\pp^1)^n,B_T)$ where $B_T$ is the torus invariant divisor. 
It is clear that the pair $((\pp^1)^n,B_T)$ admits a sequence of $n$ conic fibrations with disjoint sections. Thus $k_{\rm bir}(X,B)=0$.
On the other hand, if $k_{\rm bir}(X,B)=0$, then by Theorem~\ref{theorem:conic-complexity-zero}, we know that $(X,B)$ is crepant birational equivalent to a toric log Calabi--Yau pair.
Hence, it is crepant birational equivalent to $(\pp^n,H_0+\dots+H_n)$.

The second statement follows from Lemma~\ref{lem:finite-act-regularization}
and
Theorem~\ref{theorem:conic-complexity-zero}.
\end{proof}

\subsection{Birational complexity and conic fibrations}

In this subsection, we prove that the birational complexity of a log Calabi--Yau pair is at least its birational conic complexity.
First, we introduce the following definition
motivated by the concept 
of orbifold divisors.

\begin{definition}\label{def:orbifold-divisors}
{\em 
Let $X$ be a normal quasi-projective variety, 
$\Gamma$ be a reduced divisor on $X$, and $n$ be a positive integer. Set $B:=(1-n^{-1})\Gamma$.
We define $\mathbb{K}(X,B)$ to be the $\mathbb{K}(X)$-algebra generated by the elements $f\in \overline{\mathbb{K}(X)}$ for which $f^n\in \mathbb{K}(X)$ and $\supp ({\rm div}(f))\subseteq \supp(B)$.
Let $g\in \kk(X,B)$ be an element for which $g^n\in \kk(X)$. We write ${\rm div}(g):={\rm div}(g^n)/n$. For a $\qq$-divisor $D$ on $X$, we
define the sheaf $\mathcal{O}_{(X,B)}(D)$ by:
\[
\mathcal{O}_{(X,D)}(D)|_U:=
\{ f\in \kk(X,B) \mid 
f^n \in \kk(X) \text{ and }
({\rm div}(f)+D)|_U \geq 0
\}
\]
for every open $U\subseteq X$.
}
\end{definition}

The following theorem implies Theorem~\ref{theorem:cbir-conic}.

\begin{theorem}\label{thm:cbir-conic}
Let $(X,B)$ be a log Calabi--Yau pair
of index one.
Let $n$ be the dimension of $X$ and $c=\bar{c}_{\rm bir}(X,B)$.
Then, there exists a crepant birational model
$(X',B')$ of $(X,B)$ satisfying the following:
\begin{itemize}
\item the birational map $X'\dashrightarrow X$ only extract log canonical places of $(X,B)$, 
\item the variety $X'$ admits a sequence of Mori fiber spaces 
\[ 
\xymatrix{
X'\ar[r]^-{\psi_0} & 
 X_1 \ar[r]^-{\psi_1} & \dots  
\ar[r]^-{\psi_{r-1}}
& X_r
}
\]
\item 
there is a subset $I\subseteq \{0,\dots,r-1\}$
with $|I|\geq n-c$
such that each $\psi_i$ with $i\in I$ is a conic fibration, 
\item if $(X_i,B_i,\mathbf{M}_i)$ is obtained by the canonical bundle formula and $i\in I$, then $\lfloor B_i\rfloor$ has two disjoint prime components that are horizontal over $X_{i+1}$.
\end{itemize} 
\end{theorem}

\begin{proof}
Let $(X_0,B_0)$ be a birational model of $(X,B)$
for which $\bar{c}(X_0,B_0)=\bar{c}_{\rm bir}(X,B)=c$.
Without loss of generality, we may assume that $(X_0,B_0)$ is $\qq$-factorial and dlt.
If $c\geq n$, then the statement is vacuous.
We assume that $c < n$.
By~\cite[Theorem 2.25]{FFMP22}, we may assume that there is a Fano type fibration
$X_0\rightarrow W$ such that all 
components of $B_0$ are horizontal over $W$.
We can find Weil divisors $D_1,\dots,D_{n-c}$
satisfying the following conditions:
\begin{enumerate}
\item each $D_i$ is supported on $\supp(B_0)$, 
\item for each $i\in \{1,\dots, n-c \}$ we have 
$D_i\sim 0$, 
\item for each $i\in \{1,\dots, n-c\}$ there is a prime component of $\supp(D_i)$ that is not contained in the set $\bigcup_{j\neq i}\supp(D_j)$.
\end{enumerate} 
For each $p\in \zz_{>0}$, we consider the following variety
\[
Y_{0,p}:={\rm Spec}_{X_0}\left(
\bigoplus_{\vec{\alpha}\in \zz_p^{n-c}}
\mathcal{O}_{\left(X_0,\left(1-\frac{1}{p}\right)B_0\right)} 
\left( 
\alpha_1 \frac{D_1}{p} +\dots + \alpha_{n-c}\frac{D_{n-c}}{p}
\right) 
\right). 
\]
By condition (2), the variety $Y_{0,p}$ is a $\zz_p^{n-c}$-quasi-torsor over $X_0$ (see~\cite[Definition 2.7]{BM22} and ~\cite[Proposition 4.10]{BM24}).
By condition (1), we know that $Y_{0,p}\rightarrow X_0$ only ramifies along $B_0$.
By condition (3), if $p$ is a large prime number, then coefficients of each $\alpha_iD_i/p$ are fractional
provided that $\vec{\alpha}\neq 0$ in $\zz_p^{n-c}$.
In particular, the induced morphism
$Y_{0,p}\rightarrow W$ has connected fibers for $p\gg 0$.
Let $(Y_{0,p},B_{Y_{0,p}})$ be the log pull-back of $(X_0,B_0)$.
By the above considerations, the pair $(Y_{0,p},B_{Y_{0,p}})$ is log Calabi--Yau.
Since $Y_{0,p}\rightarrow X_0$ only ramifies along $\lfloor B_0\rfloor=B_0$, we conclude that 
$Y_{0,p} \rightarrow W$ is a Fano type morphism (see Lemma~\ref{lem:FT-under-morphisms}).
By construction, the group $\zz_p^{n-c}$ acts on $(Y_{0,p},B_{Y_{0,p}})$ fiberwise over $W$.
Thus, if $p\geq p_0(n)$, then we can apply Theorem~\ref{thm:fin-act-conic}.
From now on, we just fix $p\gg 0$ and set 
$(Y_0,B_{Y_0}):=(Y_{0,p},B_{Y_{0,p}})$.
By Theorem~\ref{thm:fin-act-conic}, 
there exists a 
$\zz_p^{n-c}$-equivariant crepant birational model
$(Y',B_{Y'})$ of $(Y_0,B_{Y_0})$ over $W$ that admits a tower of Mori fiber spaces
\[ 
\xymatrix{
Y' \ar[rrrd] \ar[r]^-{\phi_0} & 
 Y_1 \ar[rrd] \ar[r]^-{\phi_1} & \dots  
\ar[r]^-{\phi_{r-1}}
& Y_r \ar[d]^-{\phi_r} \\ 
& & & W 
}
\]
satisfying the following conditions:
\begin{enumerate}
\item there is a subset $I\subseteq \{0,\dots,r\}$
with $|I|\geq n-c$ such that each $\psi_i$ with $i\in I$ is a conic fibration, and
\item if $i\in I$, $\zz_p^{k_i}$ is the homomorphic image of $\zz_p^{n-c}$ acting on $Y_i$, and 
$(Y_i,B_{Y_i},\mathbf{N}_i)$ is induced by the equivariant generalized canonical bundle formula on $Y_i$, then $\lfloor B_{Y_i}\rfloor$ contains two prime disjoint $\zz_p^{k_i}$-invariant components that are horizontal over $Y_{i+1}$.
\end{enumerate} 
Let $(X',B')$ be the quotient of $(Y',B_{Y'})$ by $\zz_p^{n-c}$ and $X_i$ be the quotient of $Y_i$ by $\zz_p^{k_i}$ for each $i$.
Then, we obtain a commutative diagram as follows:
\[
\xymatrix{
& (Y_0,B_{Y_0}) \ar@{-->}[r]\ar[d]_-{/\zz_p^{n-c}} & 
(Y',B_{Y'}) \ar[d]_-{/\zz_p^{n-c}}\ar[r]^-{\phi_0} & 
Y_1 \ar[r]^-{\phi_1} \ar[d]_-{/\zz_p^{k_1}} & 
\dots \ar[r]^-{\phi_{r-1}} & 
Y_r \ar[d]_-{/\zz_p^{k_r}} \ar[rddd]^-{\phi_r} \\
(X,B) \ar@{-->}[r] &
(X_0,B_0) \ar[rrrrrdd] \ar@{-->}[r] & 
(X',B') \ar[rrrrdd] \ar[r]^-{\psi_0} & 
X_1 \ar[rrrdd] \ar[r]^-{\psi_1} & 
\dots \ar[r]^-{\psi_{r-1}} & X_r \ar[rdd]_-{\psi_r} &  \\ 
&  & & & & &  \\ 
& & & & & &  W
}
\]
If $\phi_i$ is a conic fibration, then $\psi_i$ is a conic fibration as well.
This implies that there is a subset $I\subseteq \{0,\dots,r\}$ with $|I|\geq n-c$ such that each $\psi_i$ with $i\in I$ is a conic fibration.
The generalized pair $(X_i,B_i,\mathbf{M}_i)$ obtained by the canonical bundle formula
is the quotient of $(Y_i,B_{Y_i},\mathbf{N}_i)$ by $\zz_p^{k_i}$.
For each $i\in I$, the divisor $\lfloor B_{Y_i}\rfloor$ has two disjoint $\zz_p^{k_i}$-invariant prime components horizontal over $Y_{i+1}$.
Hence, for each $i\in I$, the divisor
$\lfloor B_i \rfloor$, which is the $\zz_p^{k_i}$-quotient of $\lfloor B_{Y_i}\rfloor$, has two disjoint prime components that are horizontal over $X_{i+1}$.
\end{proof} 

Now, we turn to prove the two corollaries regarding the birational complexity and the conic fibrations.

\begin{proof}[Proof of Theorem~\ref{cor:hyperf}]
Let $X_d\subset \pp^n$ be a smooth hypersurface of degree $d\leq n$.
Since $X_d$ is smooth, we have $\rho(X_d)=1$.
Let $H_1,\dots,H_{n-d}$ be general hyperplanes in $\mathbb{P}^n$.
Then, the pair
$(X_d, H_1|_{X_d}+\dots + H_{n+1-d}|_{X_d})$ is log Calabi--Yau.
Hence, we have 
\[
c_{\rm bir}(X_d) 
\leq 
c(X_d, H_1|_{X_d}+\dots + H_{n+1-d}|_{X_d})
= n - (n+1-d) = d-1. 
\]
We conclude that the birational complexity of $X_d$ is at most $d-1$.
Then, the statement follows from Theorem~\ref{theorem:cbir-conic}.
\end{proof} 

\begin{proof}[Proof of Corollary~\ref{cor:birsuper}]
By~\cite[Corollary 1.4]{MM24}, we have that $c_{\rm bir}(X,B)\in \{0,\dots,n\}$.
If $c_{\rm bir}(X,B)<n$, then Theorem~\ref{thm:cbir-conic} implies that $X$ admits a birational non-trivial Mori fiber space. 
This contradicts the fact that $X$ is birationally superrigid. 
We conclude that $c_{\rm bir}(X,B)=n$.
\end{proof} 

\begin{corollary}\label{cor:bsuper}
Let $X$ be a Fano variety of Picard rank one.
Assume that $X$ admits a $1$-complement $B$. 
If the divisor $B$ has at least $2$ prime components, then $X$ is not birationally superrigid.
\end{corollary}

\begin{proof}[Proof of Corollary~\ref{cor:bsuper}]
Let $X$ be a klt Fano variety of Picard rank $1$.
Let $(X,B)$ be a $1$-complement.
If $B$ has at least two components, then
$c(X,B) < n$. 
Theorem~\ref{theorem:cbir-conic} implies that $X$ admits a birational non-trivial Mori fiber space, i.e., 
a Mori fiber space whose image is a positive dimensional variety.
\end{proof} 

\subsection{Dimension of the dual complex}
In this subsection, we prove a theorem relating the dimension of the dual complex
and the alteration conic complexity.
In~\cite[Theorem 1.12]{MM24}, the authors proved
that for a log Calabi--Yau pair $(X,B)$
the invariant 
${\rm reg}(X,B)$ is bounded below by
$n-c_{\rm bir}(X,B)-1$, or analogously, 
that ${\rm coreg}(X,B)\leq c_{\rm bir}(X,B)$.
The following theorem gives an improvement 
using the concept of alteration conic complexity.

\begin{theorem}\label{theorem:k-alt-vs-reg}
Let $(X,B)$ be a log Calabi--Yau pair of index one.
Then, we have 
\[
n-k_{\rm alt}(X,B)-1\leq 
{\rm reg}(X,B)
\]
or analogously
\[
{\rm coreg}(X,B)\leq k_{\rm alt}(X,B).
\]
\end{theorem}

\begin{proof}[Proof of Theorem~\ref{theorem:k-alt-vs-reg}]
Let $(X,B)$ be a log Calabi--Yau pair of dimension $n$.
If $k_{\rm alt}(X,B)=n$, then there is nothing to prove.
Indeed, we have ${\rm reg}(X,B)\geq -1$ for every log Calabi--Yau pair.

Assume that $k_{\rm alt}(X,B)=k<n$. 
Then, we may apply Lemma~\ref{lem:finite-act-regularization}.
Thus, there exists:
\begin{itemize}
\item a crepant birational map
$(X',B')\dashrightarrow (X,B)$, and 
\item a crepant finite Galois morphism
$(Y,\Delta)\rightarrow (X',B')$,
\end{itemize}
such that $k_{\rm alt}(X,B)=k_{\rm bir}(Y,\Delta)=k$.
Note that ${\rm reg}(X,B)={\rm reg}(X',B')={\rm reg}(Y,\Delta)$. 
Hence, it suffices to show that the inequality holds for $(Y,\Delta)$.
Since $k_{\rm bir}(Y,\Delta)=k$, there
exists a crepant birational map
$(Y',\Delta')\dashrightarrow (Y,\Delta)$
for which $k(Y',\Delta')=k<n$.
Hence, we have two Fano type morphisms
\begin{equation}\label{eqq:1} 
\xymatrix{ 
Y' \ar[r]^-{\phi_1} & Y_1 \ar[r]^-{\phi_2} & Y_2
}
\end{equation}
where $\phi_2$ is a conic fibration. 
In the previous composition, the morphism $\phi_1$ is possibly the identity.
Let $(Y_1,\Delta_1)$ and $(Y_2,\Delta_2)$ be the log Calabi--Yau pairs induced by the canonical bundle formula on $Y_1$ and $Y_2$, respectively.
Then, we know that $\lfloor \Delta_1\rfloor$ has two disjoint sections over $Y_2$.
By construction, we have 
${\rm reg}(Y',\Delta')\geq {\rm reg}(Y_1,\Delta_1)={\rm reg}(Y_2,\Delta_2)+1$.
Indeed, the dual complex of $(Y_1,\Delta_1)$ is the joint of the dual complex
of $(Y_2,\Delta_2)$ with two points.
On the other hand, we have 
\[
k(Y',\Delta')=
k(Y_1,\Delta_1)+\dim Y' - \dim Y_1=
k(Y_2,\Delta_2) -1 +\dim Y' - \dim Y_2.
\]
By induction on the dimension, we have 
\begin{equation}\label{eqq:2}
{\rm reg}(Y_2,\Delta_2) \geq \dim Y_2 - k_{\rm alt}(Y_2,\Delta_2)-1 
\geq 
\dim Y_2 - k(Y_2,\Delta_2) - 1.
\end{equation}
Putting equalities~\eqref{eqq:1}
and inequalities~\eqref{eqq:2} together, we obtain
\[
{\rm reg}(Y_2,\Delta_2) \geq 
\dim Y' - k(Y',\Delta') - 2. 
\]
Thus, we have 
\[
{\rm reg}(X,B) \geq {\rm reg}(Y_2,\Delta_2)+1 \geq \dim Y' - k(Y',\Delta')-2 = n-k-1.
\]
The second inequality follows from the definition of coregularity.
\end{proof} 

The following corollary makes plain
that Theorem~\ref{theorem:k-alt-vs-reg} implies~\cite[Theorem 1.12]{MM24}.
Indeed, among all the defined invariants
the alteration conic complexity $k_{\rm alt}(X,B)$ is the one that takes the smallest value.

\begin{corollary}\label{corollary:total-comparison}
Let $(X,B)$ be a log Calabi--Yau pair
of dimension $n$ 
of index one. 
Then, we have 
\[
c(X,B) \geq c_{\rm bir}(X,B) \geq 
\bar{c}_{\rm bir}(X,B)
\geq 
\bar{c}_{\rm alt}(X,B) \geq
k_{\rm alt}(X,B) \geq 0.
\]
Furthermore, the dimension of $\mathcal{D}(X,B)$ is at least $n-k_{\rm alt}(X,B)-1$.
\end{corollary} 

\begin{proof}[Proof of Corollary~\ref{corollary:total-comparison}]
Follows from Theorem~\ref{theorem:cbir-conic} and Theorem~\ref{theorem:alt-comp-computation}.
\end{proof}

In Section~\ref{sec:ex-and-quest}, we give plenty of examples that show that the inequalities in Corollary~\ref{corollary:total-comparison} can be strict.

\subsection{Fundamental groups and alteration conic complexity}
In this subsection, we prove theorems regarding fundamental groups
and alteration conic complexity.

The fundamental group of the smooth locus of a Fano variety is finite~\cite{Bra20}.
Indeed, the fundamental groups of the smooth locus
of $n$-dimensional Fano varieties 
satisfy the Jordan property~\cite{BFMS20}.
Similar theorems are expected
to hold
for open Calabi--Yau subvarities of Fano varieties, i.e.,
varieties of the form $U=X\setminus B$
where $X$ is Fano and $B$ a $1$-complement of $X$.
In this case, the fundamental group may be infinite
as the log pair $(\pp^1,\{0\}+\{\infty\})$
makes clear.
Nevertheless, in the previous context, the group
$\pi_1^{\rm reg}(U)$ is expected to contain a normal abelian subgroup of finite index and rank at most $n$.
The following theorem states that this expectation holds
for the algebraic fundamental group.
Further, we can control the rank of the abelian group
using the alteration conic complexity.

\begin{theorem}\label{theorem:from-pi-to-fun}
Let $n$ be a positive integer. 
There exists a constant $c(n)$, 
only depending on $n$, 
satisfying the following.
Let $X$ be a $n$-dimensional Fano type variety.
Assume that $X$ admits a $1$-complement $B$ and let $U=X^{\rm reg}\setminus B$.
Then, there is a short exact sequence:
\[
1\rightarrow A \rightarrow \pi_1^{\rm alg}(U) \rightarrow N \rightarrow 1,
\]
where $A$ is an inverse limit of finite abelian groups of rank at most $n-k_{\rm alt}(X,B)$ and $|N|\leq c(n)$.
\end{theorem}

\begin{proof}[Proof of Theorem~\ref{theorem:from-pi-to-fun}]
First, we introduce some notation about boundaries with standard coefficients.
Let $P_1,\dots,P_b$ be the prime components of the divisor $B$.
Let $\vec{m}=(m_1,\dots,m_b)\in \zz_{>0}^b$.
We write $\Delta_{\vec{m}}$ for the boundary divisor with coefficient $1-\frac{1}{m_i}$ at $P_i$.
Hence, for every such $\vec{m}$, we
have that $\Delta_{\vec{m}}$ has standard coefficients, $\lfloor \Delta_{\vec{m}}\rfloor=0$,
and $\Delta_{\vec{m}}\leq B$.

By the proof of~\cite[Theorem 3.3]{MYY24},
we know that there is an isomorphism 
\begin{equation}\label{eq:isom-fun} 
\pi_1^{\rm alg}(X\setminus B) \simeq 
\varprojlim_{\vec{m}} \pi_1^{\rm reg}(X,\Delta_{\vec{m}}).
\end{equation}
Let $\Delta_{\vec{m}}\leq B$ be a boundary divisor with standard coefficients and 
$\lfloor \Delta_{\vec{m}}\rfloor =0$.
We argue that there is a short exact sequence 
\begin{equation}\label{eq:ses-finite} 
1 \rightarrow A_{\vec{m}} 
\rightarrow \pi_1^{\rm reg}(X,\Delta_{\vec{m}}) 
\rightarrow N_{\vec{m}} \rightarrow 1,
\end{equation} 
where $A_{\vec{m}}$ is a finite abelian group of rank at most $n-k_{\rm alt}(X,B)$
and $|N_{\vec{m}}| \leq c(n)$.
Here, $c(n)$ is a constant that only depends on the dimension $n$.
Write $G_{\vec{m}}:=\pi_1^{\rm reg}(X,\Delta_{\vec{m}})$.
Let $Y_{\vec{m}}\rightarrow X$ be the cover associated to $G_{\vec{m}}$.
Let $(Y_{\vec{m}},D_{\vec{m}})$
be the log pull-back of $(X,\Delta_{\vec{m}})$.
By~\cite[Theorem 3]{BFMS20},
there is a normal finite abelian subgroup 
$A_{\vec{m}}\leqslant G_{\vec{m}}$ of index at most $c_0(n)$.
By Lemma~\ref{lem:FT-under-morphisms}, the variety $Y_{\vec{m}}$ is of Fano type.
We know that $A_{\vec{m}}$ has rank at most $n$.
Write $A_{\vec{m}} := \bigoplus_{i=1}^n \zz_{b_{i,\vec{m}}}$ with $b_{1,\vec{m}}\mid \dots \mid b_{n,\vec{m}}$.
Let $s\in \{1,\dots,n\}$ be the smallest positive for which $b_{s,\vec{m}}\geq m_0(n,\{0,1\})$, 
where $m_0(n,\{0,1\})$ is as in Theorem~\ref{thm:fin-act-conic}.
We replace 
$A_{\vec{m}}:= \bigoplus_{i=1}^n \zz_{b_{i,\vec{m}}}$ with $\bigoplus_{i=s}^n \zz_{b_{i,\vec{m}}}$. 
Note that the former group has index bounded above by $m_0(n,\{0,1\})^{s-1}$ in $A_{\vec{m}}$.
After such replacement, by Theorem~\ref{thm:fin-act-conic}, we have an
$A_{\vec{m}}$-equivariant crepant birational model
$(Y'_{\vec{m}},\Delta'_{\vec{m}})$ of $(Y_{\vec{m}},\Delta_{\vec{m}})$
and a sequence of equivariant Mori fiber spaces:
\[
\xymatrix{ 
(Y'_{\vec{m}}, \Delta'_{\vec{m}})
\ar[r]^-{\psi_1}
& 
(Z_1,\Delta_1,\mathbf{M}_1)
\ar[r]^-{\psi_2}
& 
\dots 
\ar[r]^-{\psi_r}
& 
(Z_r,\Delta_r,\mathbf{M}_r)
\ar[r]^-{\psi_{r+1}}
&
{\rm Spec}(\kk) 
}
\]
We write $\psi_{1,j}:=\psi_j \circ \dots \circ \psi_1$.
By construction, we know that there is a subset $I\subseteq \{1,\dots,r+1\}$
with $|I|=n-s$ 
such that each $\psi_i$ with $i\in I$ is a conic fibration.
Let $A_{\vec{m},i}$ be the homomorphic image of $A_{\vec{m}}$ acting on $Z_i$.
Furthermore, for each $i \in I$, the boundary divisor $\lfloor \Delta_i\rfloor$ has two disjoint $A_{\vec{m},i}$-invariant prime components that are horizontal over $Z_{i+1}$.
This implies that $k(Y'_{\vec{m}},\Delta'_{\vec{m}}) \leq s.$
Hence, the rank of $A_{\vec{m}}$ is at most $n-k(Y'_{\vec{m}},\Delta'_{\vec{m}})$.
So the rank of $A_{\vec{m}}$ is at most
$n-k_{\rm alt}(X,B)$. 
Therefore, we have a short exact sequence as in~\eqref{eq:ses-finite}.

If $\vec{m}\mid \vec{n}$, meaning that $m_i\mid n_i$ for every $i\in \{1,\dots,b\}$,
then we have a surjective homomorphism 
\[
\pi_1^{\rm reg}(X,\Delta_{\vec{n}})
\rightarrow
\pi_1^{\rm reg}(X,\Delta_{\vec{m}}).
\]
Note that $|N_{\vec{m}}|\leq c(n)$.
Thus, the finite group $N_{\vec{m}}$ stabilizes for $\vec{m}$ divisible enough.
So, taking the inverse limit preserves surjectivity in this case.
Taking the inverse limit, we get an exact sequence 
\[
1\rightarrow 
A \rightarrow
\varprojlim_{\vec{m}} \pi_1^{\rm reg}(X,\Delta_{\vec{m}}) \rightarrow N \rightarrow 1,
\]
where $|N|\leq c(n)$.
By construction, $A$ is the inverse limit of finite abelian groups of rank at most $n-k_{\rm alt}(X,B)$.
Then, the isomorphism~\eqref{eq:isom-fun}
implies the statement.
\end{proof}

As a consequence, 
if $(X,B)$ has maximal
alteration conic complexity,
then the algebraic fundamental group
of $U$ is finite
and its order is
bounded above in terms of $n$.

\begin{corollary}\label{corollary:fun-group-max-alt-comp}
Let $n$ be a positive integer. 
There exists a constant $a(n)$, only depending on $n$, satisfying the following.
Let $X$ be a $n$-dimensional Fano type variety.
Assume that $X$ admits a $1$-complement $B$ for which $k_{\rm alt}(X,B)=n$.
If 
$U=X^{\rm reg} \setminus B$,
then $\pi_1^{\rm alg}(U)$ is finite and its order is bounded above by $a(n)$.
\end{corollary} 

\begin{proof}[Proof of Corollary~\ref{corollary:fun-group-max-alt-comp}]
Follows from Theorem~\ref{theorem:from-pi-to-fun}.
\end{proof}

The previous corollary gives simple criteria 
to check whether 
a log Calabi--Yau pair $(X,B)$ 
of index one has maximal alteration complexity:
if $\pi_1(U)$ is infinite for $U=X\setminus B$,
then $k_{\rm alt}(X,B)$ is strictly less than the dimension of $X$.

\subsection{Log Calabi--Yau pairs of coregularity 0}

In this subsection, we prove theorems regarding log Calabi--Yau pairs of coregularity zero. 

Log Calabi--Yau pairs of coregularity zero
are of special interest~\cite{FMM22}.
Indeed, every $n$-dimensional Fano variety of coregularity zero $X$
admits a $2$-complement $(X,B)$ of coregularity zero (see, e.g.,~\cite[Theorem 4]{FFMP22}).
Thus, a $n$-dimensional Fano variety $X$ of coregularity zero admits a $1$-complement of coregularity zero $(X,B)$
after possibly replacing $X$ with a $2$-to-$1$ cover.
By~\cite[Theorem 1.4]{MM24}, we know that $c_{\rm bir}(X,B)\leq n$. 
If $c_{\rm bir}(X,B)<n$, then Theorem~\ref{theorem:cbir-conic} implies that $(X,B)$ admits a non-trivial birational Mori fiber space.
Thus, the building blocks of log Calabi--Yau pairs of coregularity zero and index one are:
\begin{enumerate}
\item[(i)] the pair $(\pp^1,\{0\}+\{\infty\})$, and 
\item[(ii)] log Calabi--Yau pairs of coregularity zero, index one, and maximal birational complexity.
\end{enumerate}
If we allow to replace the pair $(X,B)$
by both: birational transformations and finite covers, then the building blocks are 
pairs of maximal alteration complexity.
The following theorem makes this assertion  precise.

\begin{theorem}\label{theorem:coreg-0-decomp}
Let $(X,B)$ be a log Calabi--Yau pair
with coregularity zero
and standard coefficients.
Then, there exists a log Calabi--Yau pair 
$(X',B')$ obtained from $(X,B)$ via crepant finite Galois covers and crepant birational transformations, satisfying the following.
There exists a Fano type fibration $\phi\colon X'\rightarrow W$ and either 
\begin{enumerate}
    \item the fibration $\phi$ is a conic fibration and $\lfloor B'\rfloor$ has two disjoint sections over $W$, or 
    \item we have that $\bar{c}_{\rm alt}(X'/W,B')=\dim X' - \dim W$.
\end{enumerate}
\end{theorem} 

\begin{proof}[Proof of Theorem~\ref{theorem:coreg-0-decomp}]
Note that $2(K_X+B)\sim 0$ by~\cite[Corollary 3]{FMM22}.
We replace $(X,B)$ with its index one cover
and assume that $K_X+B\sim 0$.
By~\cite[Proposition 7.4]{ELM23}, we know that the coregularity stays zero when we pass to the index one cover.

Let $(X_0,B_0)$ be a log Calabi--Yau pair, 
obtained from $(X,B)$ via crepant finite Galois morphisms and crepant birational transformations, such that $k(X_0,B_0)$ is minimized. 
Consider the sequence of fibrations for $X_0$ that maximizes the number of strict conic fibrations:
\[
\xymatrix{
X_0 \ar[r]^-{\phi_1}  & X_1 \ar[r]^-{\phi_2} & \dots \ar[r]^-{\phi_r} & X_r \ar[r]^-{\phi_{r+1}} & {\rm Spec}(\kk). 
}
\]
We may assume that $\phi_1$ is a Mori fiber space. If $\phi_1\colon X_0\rightarrow X_1$ is a conic fibration with two disjoint sections contained in $\lfloor B_0\rfloor$, then condition (1) holds.
From now on, we assume that $\phi_1$ is not a conic fibration.

Note that $K_{X_0}+B_0\sim 0$ and ${\rm coreg}(X_0,B_0)=0$.
Indeed, both properties are preserved
under crepant finite Galois morphism
and crepant birational transformations 
(see~\cite[Proposition 7.4]{ELM23}
and~\cite[Proposition 3.11]{FMM22}).
By Theorem~\ref{theorem:alt-comp-noneg} and~\cite[Theorem 1.4]{MM24}, 
we have 
\[
c_{\rm alt}(X_0/X_1,B_0) 
\leq 
c_{\rm bir}(X_0/X_1,B_0) 
\leq 
\dim X_0 -\dim X_1 +{\rm coreg}(X_0,B_0)=
\dim X_0 - \dim X_1.
\]
We argue that $c_{\rm alt}(X_0/X_1,B_0)=\dim X_0-\dim X_1$.
Assume otherwise that 
$c_{\rm alt}(X_0/X_1,B_0)< \dim X_0-\dim X_1$.
Then, by Theorem~\ref{theorem:cbir-conic}, there is a commutative diagram:
\[
\xymatrix{
(Y_0,\Delta_0)\ar[d]^-{f_1} \ar[dr] \ar@{-->}[r]^-{\pi_1} &
(Y_1,\Delta_1)\ar[d] \ar[r]^-{\psi_2} &
Y_2 \ar[r]^-{\psi_3} \ar[ld] & 
\dots \ar[r]^-{\psi_s} & 
Y_s \ar[llld] \\ 
(X_0,B_0) \ar[r]^-{\phi_1} & X_1
}
\]
In the previous diagram, all the morphisms onto $X_1$ are fibrations. Further, the following conditions are satisfied:
\begin{itemize}
\item the morphism $f_1$ is a crepant finite Galois from a sub-log Calabi--Yau pair $(Y_0,\Delta_0)$, 
\item the birational map $\pi_1$ is a crepant birational transformation to a log Calabi--Yau pair $(Y_1,\Delta_1)$, and 
\item at least one of the $\psi_i$'s is a strict conic fibration.
\end{itemize}
The last two conditions imply that $k(Y_1,\Delta_1)<k(X_0,B_0)$. 
By Lemma~\ref{lem:finite-act-regularization}, the log Calabi--Yau pair $(Y_1,\Delta_1)$ can be obtained from $(X_0,B_0)$ via crepant birational transformations
and a single crepant finite Galois morphism.
In particular, the pair $(Y_1,\Delta_1)$ can be obtained from $(X,B)$ via finite Galois crepant morphisms
and crepant birational transformations.
This contradicts the minimality of $k(X_0,B_0)$. 
Thus, we conclude that
$c_{\rm alt}(X_0/X_1,B_0)=\dim X_0-\dim X_1$. Then, condition (2) holds.
\end{proof}

In low dimensions, we can prove the following theorem:

\begin{theorem}\label{theorem:bound-coreg-0-3-fold}
Let $(X,B)$ be a log Calabi--Yau pair
with index one and ${\rm coreg}(X,B)=0$.
\begin{enumerate}
\item If $\dim X=2$, then $c_{\rm bir}(X,B)=c_{\rm alt}(X,B)=0$.
\item If $\dim X=3$ and $c_{\rm bir}(X,B)<3$, then either 
\begin{enumerate} 
\item[(i)] we have $c_{\rm alt}(X,B)=0$, or
\item[(ii)] the variety $X$ admits a birational
del Pezzo fibration.
\end{enumerate} 
\end{enumerate} 
\end{theorem}

\begin{proof}[Proof of Theorem~\ref{theorem:bound-coreg-0-3-fold}]
First assume that $(X,B)$ has dimension $2$.
If $K_X+B\sim 0$ and ${\rm coreg}(X,B)=0$, then $\mathcal{D}(X,B)\simeq S^1$ (see, e.g.,~\cite[Lemma 2.12]{GLM23}).
Hence, by~\cite[Proposition 1.3]{GHS16}, we conclude that $(X,B)$ is crepant birational equivalent to 
$(\pp^2,H_0+H_1+H_2)$. 
Furthermore, the birational map $\pp^2\dashrightarrow X$ only extracts log canonical places of $(X,B)$.
This implies that $c_{\rm bir}(X,B)=0$ and hence $c_{\rm alt}(X,B)=0$.

Let $(X,B)$ be a log Calabi--Yau $3$-fold
of index one and coregularity zero.
Assume that $c_{\rm bir}(X,B)<3$.
By Theorem~\ref{theorem:cbir-conic}, we conclude that $k_{\rm bir}(X,B)<3$.
Hence, we can find a crepant birational model $(X',B')$ of $(X,B)$ with $k(X',B')<3$. Therefore, there is a sequence of morphisms:
\[
\xymatrix{
X' \ar[r]^-{\phi_1} & X_1 \ar[r]^-{\phi_2} & X_2 \ar[r]^-{\phi_3} & {\rm Spec}(\kk), 
}
\]
such that each $\phi_i$ is either the identity or a Mori fiber space.
Furthermore, as $k(X',B')<3$, at least one of the $\phi_i$'s is a conic fibration with two disjoint sections for the corresponding log Calabi--Yau pair.

First, assume that $\phi_1$ is a conic fibration and $\lfloor B'\rfloor$ has two disjoint sections $S_0$ and $S_\infty$ over $X_1$.
Hence, the log Calabi--Yau pair $(X_1,B_1)$, obtained by the canonical bundle formula,
satisfies that $K_{X_1}+B_1\sim 0$ and ${\rm coreg}(X_1,B_1)=0$.
Let $(S_0,B_{S_0})$ be the log Calabi--Yau pair 
induced by adjunction of $(X_1,B_1)$ to $S_0$.
Then, the log Calabi--Yau pair $(S_0,B_{S_0})$ is isomorphic to $(X_1,B_1)$.
By the first paragraph,
there exists a crepant birational map
$\phi\colon (\pp^2,H_0+H_1+H_2)\dashrightarrow (X_1,B_1)$.
Let $E\subset \pp^2$ be a prime divisor that is exceptional over $X_1$. Then, we have 
$a_E(X_1,B_1) \in \{0,1\}$.
By inversion of adjunction there exists a divisor $F$ over $X'$ that maps onto $E$ and for which $a_F(X',B')\in \{0,1\}$.
Let $(X'',B'')$ be a log resolution of $(X',B')$ that dominates $\pp^2$
with a fibration $\varphi\colon X''\rightarrow \pp^2$.
For every prime divisor $P$ of $\pp^2$ 
there exists a prime divisor of $X''$ with non-negative coefficient in $B''$ contained in $\varphi^{-1}(P)$.
Let $\Gamma$ be the divisor obtained from $B''$ by increasing all the negative coefficients to zero. 
We run a $(K_{X''}+\Gamma)$-MMP over $\pp^2$ that terminates with a Mori fiber space 
$X'''\rightarrow \pp^2$. 
The push-forward $B'''$ of $B''$ to $X'''$ gives a log Calabi--Yau pair $(X''',B''')$.
By Proposition~\ref{prop:2-ray}, there 
exists a log Calabi--Yau pair 
$(Y,\Delta)$, that is crepant birational equivalent to $(X''',B''')$, with a Mori fiber space $f\colon Y\rightarrow \pp^2$ satisfying
$f^*H_j\subset \supp(\Delta)$ for each $j$.
By construction, the divisor $\Delta$ has $5$ prime components. On the other hand $\rho(Y)=2$.
Then, we have $c(Y,\Delta)=0$ and $c_{\rm bir}(X,B)=0$. In particular, $c_{\rm alt}(X,B)=0$. 
Then, condition (a) holds.

Assume that $\phi_1$ is a conic fibration over $X_1$.
If $B'$ has two sections over $X_1$, then we may apply Lemma~\ref{lem:making-sections-disjoints}.
After applying the lemma, we have two disjoint sections in $B'$ over $X_1$.
Thus, we are in the situation of the previous paragraph and we conclude that $c_{\rm bir}(X,B)=0$.
In particular, $c_{\rm alt}(X,B)=0$.
So, condition (a) holds.

Assume that $\phi_1$ is a conic fibration over $X_1$ and $B'$ has a prime component that maps $2$-to-$1$ onto $X_1$. Up to a $2$-to-$1$ cover of $X'$, possibly ramified along $B'$, we may assume that $B'$ has two prime components that are sections over $X_1$. Hence, we are in the situation of the previous paragraph.
In this case, we conclude that $c_{\rm alt}(X,B)=0$, so condition (a) holds.

From now on, we may assume that $\phi$ does not have relative dimension one. 
Hence, we may assume that $\phi_1$ has relative dimension $2$, 
$\phi_2$ is the identity, and $X_2\simeq \pp^1$.
We take a terminalization of $X'$ and run a minimal model program over $\pp^1$.
If this minimal model program
terminates with a Mori fiber space onto a surface, then the statement follows from the previous paragraph.
We may assume that the MMP terminates with a Mori fiber space onto $\pp^1$.
Thus, $X$ is birational to a del Pezzo fibration and (b) holds.
\end{proof}

As mentioned above, 
we have $c_{\rm alt}(X,B)\in \{0,1,\dots,n\}$
for a $n$-dimensional log Calabi--Yau pair
of index one and coregularity zero.
Theorem~\ref{theorem:bound-coreg-0-3-fold} 
implies that the values $1$ and $2$ do not occur for $c_{\rm alt}(X,B)$ in dimension $2$.
For $3$-folds, only the values $0$ and $1$ are known to occur (see Question~\ref{quest:what-value}).

\subsection{Weil index of Fano varieties} 

In this subsection, we prove the theorems in the introduction regarding Weil indices of Fano varieties.

The {\em Weil index} $i(X)$ of a Fano variety $X$ is the largest positive integer
for which we can write $-K_X \sim iW$ with $W$ a Weil divisor.
If $X$ is smooth, then $i(X)\leq n+1$ and 
$X\simeq \pp^n$ whenever $i(X)>n$ (see, e.g.,~\cite{KO73}).
However, if $X$ is singular, the Weil index can be arbitrarily large.
For instance $\pp(1,1,n)$ has Weil index $n$.
Thus, a priori, the Weil index can be large and not reflect in the geometry of the Fano variety.
The following theorem shows that $n$-dimensional Fano varieties with large Weil index (compared to $n$) have non-maximal alteration conic complexity.

\begin{theorem}\label{theorem:weil-index}
Let $n$ be a positive integer.
There exists a constant $i(n)$, only depending on $n$,
satisfying the following.
Let $X$ be a $n$-dimensional Fano variety for which: 
\begin{enumerate}
\item the variety $X$ admits a $1$-complement $B$, and 
\item we have $i(X)\geq i(n)$. 
\end{enumerate} 
Then, we have $k_{\rm alt}(X,B)<n$.
\end{theorem}

\begin{proof}[Proof of Theorem~\ref{theorem:weil-index}]
Set $k(n):=m_0(n,\{0,1\})$, where $m_0$ is as in Theorem~\ref{thm:fin-act-conic}.
Write $mW \sim -K_X$ where $m$ is a positive integer and $m\geq k(n)$.
Note that $K_X+B \sim 0$ 
so $B \sim mW$.
Consider the torsion $\qq$-divisor
\[
D_m:=K_X+\left(1-\frac{1}{m}\right)B + W \sim_\qq 0.
\]
We take the index one cover $Y\rightarrow X$ associated to $D_m$.
Note that $D_m$ is Weil in the complement of $\Gamma$ the divisorial branch locus of $Y\rightarrow X$
is contained in the support of $\Gamma$.
Let $(Y,B_Y)$ be the log pull-back of $(X,B)$ to $Y$.
By Lemma~\ref{lem:FT-under-morphisms}, the variety $Y$ is a Fano type variety.
The pair 
$(Y,B_Y)$ is log Calabi--Yau,
and $K_Y+B_Y\sim 0$.
Furthermore, $\zz_m$ acts on $(Y,B_Y)$.
Whenever $m\geq k(n)\geq m_0(n,\{0,1\})$, Theorem~\ref{thm:cbir-conic} implies that $k_{\rm bir}(Y,B_Y)<n$ 
so $k_{\rm alt}(X,B)<n$.
\end{proof}

If $X$ is a Fano variety of Picard rank one
and $B$ is a $1$-complement,
then the Weil index $i(X)$ measures
how far are the prime components of $B$ from forming a basis of the Class group ${\rm Cl}(X)$.
In the case that $(X,B)$ is a log Calabi--Yau pair
of index one
and maximal alteration conic complexity,
the components of $B$ are close to being an integral basis
of a sub-lattice of ${\rm Cl}(X)$.
The following theorem makes this statement precise.

\begin{theorem}\label{theorem:decomp-weil}
Let $n$ be a positive integer.
Then, there exists a constant $d(n)$, only depending on $n$, satisfying the following.
Let $X$ be a Fano variety of dimension $n$ 
that admits a $1$-complement $B$.
Assume that $k_{\rm alt}(X,B)=\dim X$. 
Then, there exists a crepant finite Galois cover $(Y,B_Y)\rightarrow (X,B)$ of degree at most $d(n)$
such that the prime components of $B_Y$ 
form a basis of the lattice
${\rm Cl}(Y) \cap \langle B_Y \rangle$.
\end{theorem} 

In the previous theorem $\langle B_Y\rangle$ denotes the span over $\qq$ of the prime components of the divisor $B_Y$
in ${\rm Cl}(Y)_\qq$.

\begin{proof}[Proof of Theorem~\ref{theorem:decomp-weil}]
Let $X$ be a $n$-dimensional Fano variety
and $B$ be a $1$-complement. 
Assume that $k_{\rm alt}(X,B) = \dim X$.
By Theorem~\ref{theorem:cbir-conic}, since $k_{\rm bir}(X,B)= \dim X$, the components of $B$ form a basis 
of the $\qq$-vector space 
$\langle B \rangle \subseteq {\rm Cl}(X)_\qq$.
We define a sequence of finite covers of $(X,B)$. 
If the components of $B$ form a basis
of the lattice $\langle B \rangle \cap {\rm Cl}(X)$, then we stop the sequence of finite covers.

Write $B=\sum_{j\in J} B_j$.
Assume that the components of $B$ do not form a basis of the lattice
$\langle B \rangle \cap {\rm Cl}(X)$.
Hence, there is a Weil divisor $W$ in  $\langle B \rangle \cap {\rm Cl}(X)$ that, up to linear equivalence, is not an integral combination 
of the components of $B$.
Let $m\geq 2$ be the smallest positive integer for which $mW$ is linearly equivalent to an integral combination of the components of $B$.
Thus, we can write 
\begin{equation}\label{eq:lin-equiv} 
mW \sim \sum_{i\in I} \alpha_i B_i,
\end{equation} 
for certain integers $\alpha_i$
and a subset $I\subseteq J$.
Furthermore, $m$ is the smallest integer for which there exists a linear equivalence as in~\eqref{eq:lin-equiv}.
Consider the torsion $\qq$-divisor
\begin{equation}\label{eq:tor-div}
K_X+\sum_{i\in J\setminus I} B_i
+ \sum_{i\in I}\left(1-\frac{\alpha_i}{m}\right)B_i + W \sim_\qq 0.
\end{equation}
By construction, the torsion index 
of the divisor~\eqref{eq:tor-div} is precisely $m$.
Otherwise, we contradict the minimality of $m$ in the linear equivalence~\eqref{eq:lin-equiv}.
Let $X_1\rightarrow X$ be the index one cover associated to the divisor~\eqref{eq:tor-div}.
Let $(X_1,B_1)$ be the log pull-back of $(X,B)$ to $X_1$.
Hence, the variety $X_1$ is of Fano type and $B_1$ is a $1$-complement (see Lemma~\ref{lem:FT-under-morphisms}).
Note that $k_{\rm alt}(X_1,B_1)= \dim X$ as $k_{\rm alt}(X,B)=\dim X$. In particular, we have 
$k_{\rm bir}(X_1,B_1) = \dim X_1$, 
so the components of $B_1$ form a basis of the $\qq$-vector space $\langle B_1 \rangle \subseteq {\rm Cl}(X_1)_\qq$.
If the components of $B_1$ form a basis of the lattice
$\langle B_1 \rangle \cap {\rm Cl}(X_1)$, then we stop the sequence of finite covers. Otherwise, we define $X_2\rightarrow X_1$ by repeating the construction explained above.

Hence, we get a sequence of finite 
crepant morphisms 
\[
\xymatrix{ 
(X,B) & (X_1,B_1) \ar[l]_-{f_1} & \dots \ar[l]_-{f_2} & (X_r,B_r) \ar[l]_-{f_r} & 
\dots \ar[l]_-{f_{r+1}} 
}
\]
By construction, the degree of each $f_i$ is positive.
We write $g_i:=f_1 \circ \dots \circ f_i$.
Thus, we have $\deg(g_i)>\deg(g_{i-1})$ for each $i$.
Let $h_i\colon Y_i\rightarrow X$ be the Galois closure of $g_i$.
Then, we have $\deg(h_i)\geq \deg(g_i)$.
Furthermore, if $(Y_i,\Delta_i)$ is the log pull-back of $(X,B)$ to $Y_i$,
then $h_i\colon (Y_i,\Delta_i)\rightarrow (X,B)$ is a crepant finite Galois morphism.
By Lemma~\ref{lem:FT-under-morphisms}, $Y_i$ is of Fano type
and $\Delta_i$ is a $1$-complement.
There is a finite group $G_i$ with $|G_i|\geq 2^i$ acting on $(Y_i,\Delta_i)$.
By~\cite[Theorem 3]{BFMS20}, there is a normal abelian subgroup
$A_i \leqslant G_i$ of rank at most $n$ and index at most $k_0(n)$. 
Let $m_0:=m_0(n,\{0,1\})$ be as in Theorem~\ref{thm:cbir-conic}.
If $\deg(g_i)>k_0m_0^n$, then
$(Y_i,\Delta_i)$ admits the action
of $\zz_l$ with $l>m_0$.
Thus, by Theorem~\ref{thm:cbir-conic}, we conclude that 
$k_{\rm bir}(Y_i,\Delta_i)<n$ so 
$k_{\rm alt}(X,B)<n$. 
We conclude that the sequence of finite covers stops at $(X_s,B_s)$ for some $s$. Furthermore, the degree of the finite morphism $X_s\rightarrow X$, which we may assume Galois, is bounded above by $k_0m_0^n$.
By construction, the components of $B_s$ generate the lattice $\langle B_s \rangle \cap {\rm Cl}(X_s)$.
Thus, it suffices to set $(Y,B_Y):=(X_s,B_s)$.
\end{proof}

\subsection{Complexity of klt singularities} In this subsection, we prove an application to klt singularities.

\begin{theorem}\label{theorem:klt-sing}
Let $(X;x)$ be a $n$-dimensional $\qq$-factorial klt singularity that admits a $1$-complement $B$.
Let $b$ be the number of components of $B$ through $x\in X$.
Then, for every log canonical place $\nu$ of $(X,B)$ with center $c_X(\nu)=x$, we have  
$k_{\rm bir}(\nu) \leq n-b$.
\end{theorem}

In the previous statement, the value $k_{\rm bir}(\nu)$ is defined to be the minimum among $k_{\rm bir}(E_\nu)$ where $E_\nu$ is a divisorial center of $\nu$ on some higher birational model of $X$.

\begin{proof}[Proof of Theorem~\ref{theorem:klt-sing}]
Let $(X;x)$ be a $n$-dimensional $\qq$-factorial klt singularity.
Let $B$ be a $1$-complement of the singularity. 
Let $b$ be the number of components of $B$ through $x\in X$.
Let $\nu$ be a log canonical place with $c_X(\nu)=x$.
By~\cite[Theorem 1]{Mor20}, there exists a projective birational morphism $\pi\colon Y\rightarrow X$ from a normal $\qq$-factorial variety which is an isomorphism over $X\setminus \{x\}$ and $\phi^{-1}(x)=E_\nu$ is a unique prime divisor corresponding to $\nu$.
Let $\pi^*(K_X+B)=K_Y+B_Y+E_\nu$.
The divisor $E_\nu$ may not be normal.
However, we know that $E_\nu$ is an slc variety.
Hence, there exists a projective birational morphism $\phi\colon Z \rightarrow X$ satisfying the following conditions:
\begin{itemize}
\item the variety $Z$ is normal and $\qq$-factorial,
\item the center of $\nu$ on $Z$ is a prime divisor $E_Z$, and
\item if $K_Z+B_Z=\pi^*(K_X+B)$, then every component of $B_Z$ intersects $E_Z$.
\end{itemize} 
The variety $Z$ can be obtained by blowing up the non-normal locus of $E_\nu$ in $Y$.
Let $\rho(Z/X)=\rho$.
Then, $B_Z$ has $\rho+b$ components.
Write $B_Z:=\sum_{i=1}^{\rho+b} B_{i,Z}$ with $B_{0,Z}=E_Z$.
In particular, $B_Z-E_Z$ has $\rho+b-1$ components.
Note that $\langle B\rangle =0$ as $X$ is $\qq$-factorial.
The span $\langle B_Z-E_Z\rangle$ has dimension at most $\rho$ in ${\rm Cl}(Z/X)_\qq$.
Let $(E,B_E)$ be the pair obtained by adjunction of $(Z,B_Z)$ to $E$.
Hence, the log pair $(E,B_E)$ is a log Calabi--Yau pair.
Let $\Sigma_E$ be the decomposition of $B_E$ given by $\Sigma_E :=\sum_{i=1}^{b+\rho} B_{i,Z}|_E$.
Observe that
\[
\dim_\qq \langle \Sigma_E \rangle 
\leq \rho 
\text{ and }
|\Sigma_E| \geq \rho+b-1.
\]
Therefore, we conclude
\[
\bar{c}(E,B_E) \leq 
\dim E + \dim_\qq \langle B_E \rangle - |B_E| 
\leq n-1 + \rho -(\rho+b-1) =
n-b.
\]
Hence, we have $\bar{c}(E,B_E)\leq n-b$.
Theorem~\ref{thm:cbir-conic} implies that 
$k_{\rm bir}(E,B_E)\leq n-b$. 
Thus, we have $k_{\rm bir}(\nu)\leq k_{\rm bir}(E)\leq n-b$.
\end{proof}

If $b=n$ in the setting of Theorem~\ref{theorem:klt-sing}, then it is known that $(X;x)$ is formally toric (see, e.g.,~\cite{MS21}).
Hence, every log canonical place of $(X;x)$ with center $x$ is a toric valuation.
In particular, such a valuation is extracted by a projective toric birational morphism.
Thus, 
Theorem~\ref{theorem:klt-sing} can be regarded as a generalization of the fact that
lc centers of a toric germ 
are toric valuations.

\section{Examples and Questions}
\label{sec:ex-and-quest}

In this section, we collect some examples
and propose some questions to motivate further research. In the first example, we study equivariant generalized log Calabi--Yau pairs on $\pp^1$.

\begin{example}\label{ex:p1}
{\em 
Let $(\pp^1,B,\mathbf{M})$ be a generalized log Calabi--Yau pair.
Assume $\zz_m \leqslant {\rm Aut}(\pp^1,B,\mathbf{M})$.
Let $\Omega\subset (0,1]$ be a set satisfying the DCC. 
Assume that the coefficients of $B$ belong to $\Omega$
and the coefficients of $\mathbf{M}$ belong to $\Omega$ in the quotient by $\zz_m$ where it descends.
Let $\lambda_0>0$ such that $\Omega \subset [\lambda_0,1]$.
We argue that if $m\geq \min \{7,2\lambda_0^{-1}\}$, then 
\begin{equation} 
\label{eq:isom-p1}
(\pp^1,B,\mathbf{M})
\simeq 
(\pp^1,\{0\}+\{\infty\},\mathbf{M}),
\end{equation}
where $\mathbf{M}$ is a torsion b-nef divisor.

If $m\geq 7$, then $\zz_m$ is acting on $\pp^1$ by multiplication of $m$-roots of unity. 
We show that the b-nef divisor $\mathbf{M}$ is trivial provided that $m\geq \min\{7,2\lambda_0^{-1}\}$.
Indeed, let 
$p_m \colon \pp^1 \rightarrow \pp^1$ be the quotient by the $\zz_m$-action.
Let $\mathbf{N}$ be the b-nef nef divisor on $\pp^1$ such that $\mathbf{M}=p_m^*\mathbf{N}$.
By assumption, we can write 
\[
\mathbf{N}_{\pp^1}= \sum_{i\in I} \lambda_i N_i,
\]
where $\lambda_i \in \Omega$ 
and each $N_i$ is a nef Cartier divisor. Write $J\subseteq I$ for the subset consisting on nef Cartier divisors $N_i$ that are $\qq$-trivial.
We have 
\[
\mathbf{M}_{\pp^1}=p_m^*\left(\mathbf{N}_{\pp^1}\right) =\sum_{i\in I} \lambda_i p_m^*(N_i), 
\]
so 
\begin{equation}\label{eq:deg-moduli}
\deg\left( 
\mathbf{M_{\pp^1}} 
\right) = 
\deg\left( 
\sum_{i\in I} \lambda_i p_m^*(N_i)
\right) \geq m|I\setminus J|\lambda_0 
\end{equation}
Therefore, $\mathbf{M}$ is a torsion b-nef divisor provided $m>2\lambda_0^{-1}$.
Otherwise from~\eqref{eq:deg-moduli}, we conclude $\deg\left( 
\mathbf{M_{\pp^1}} 
\right)>2$ leading to a contradiction.
On the other hand, if $m>2\lambda_0^{-1}$, then $B$ is supported on $\{0\}$ and $\{\infty\}$,
as the coefficients of $B$ are in $\Omega$.
Thus, if $m\geq \min\{7,2\lambda_0^{-1}\}$, then $\mathbf{M}$ is a torsion b-nef divisor
and $B$ is supported on $\{0\}+\{\infty\}$ which implies the isomorphism~\eqref{eq:isom-p1}.
}
\end{example} 

\begin{example}\label{ex:not-conic-bundle}
{\em
Let $X_d\subset \pp^n$ be a smooth Fano hypersurface of degree $d<n+1$.
Then, $\rho(X_d)=1$ by the Lefschetz hyperplane theorem.
If $H_1,\dots,H_{n+1-d}$ are general hyperplanes in $\pp^n$, then by adjunction 
the pair $(X_d,H_1+\dots+H_{n+1-d})$ is log Calabi--Yau.
Let $F_i$ be the restriction of $H_i$ to $X_d$.
Then, the pair
\[
(X_d,F_1+\dots+F_{n+1-d})
\]
is log Calabi--Yau
and its complexity is
$d-1$. 
So, the birational complexity of $X_d$ is at most $d-1$.
By Theorem~\ref{theorem:cbir-conic}, the variety $X_d$ has a birational model $X'_d$ that admits a sequence of Mori fiber spaces 
\[
\xymatrix{ 
X'_d \ar[r]^-{\phi_1} &
X'_{d,1} \ar[r]^-{\phi_2} & 
\dots \ar[r]^-{\phi_k} & 
X'_{d,k}
}
\]
of which at least $n-d+1$ are conic fibrations.
On the other hand, Koll\'ar proved that whenever
$X_d$ is very general 
and $d\geq 3\lceil \frac{n+3}{4}\rceil$
the variety $X_d$ does not admit a birational conic fibration (see~\cite{Kol95} and~\cite[Theorem V.5.14]{Kol96}).
Thus, the first Mori fiber space $\phi_1$ is not a conic fibration.
In the case of cubic $3$-folds $X_3\subset \pp^4$ our result simply recovers the birational pencil structure described by Pukhlikov~\cite{Puk10}.
In this case, the {\em conic fibration} is the structure morphism of the base of the pencil.
}
\end{example}

\begin{example}\label{ex:not-index-1}
{\em 
In this example, we show that the statement of Theorem~\ref{theorem:bir-comp-zero} does not hold
if we drop the condition on the index of the log Calabi--Yau pair.
Indeed, consider the log Calabi--Yau pair
\[
(T,B_T):=
\left(\pp^1\times \pp^1,
\left(\sum_{i=1}^6\frac{p_i}{3} \right)\times \pp^1+
\pp^1 \times \left(\sum_{i=1}^6\frac{p_i}{3}\right)
\right).
\]
The variety $T$ is toric and $B_T$ is a boundary divisor that is not torus invariant. 
The log Calabi--Yau pair $(T,B_T)$ is terminal and it has no non-isomorphic
crepant birational models.
Thus $(T,B_T)$ is not crepant birational equivalent to a log Calabi--Yau pair
of the form $(\pp^2,\Delta)$.
}
\end{example}

\begin{example}\label{ex:alt-comp-two-ways}
{\em
In Theorem~\ref{theorem:alt-comp-computation}, we show that the alteration complexity can be computed after 
performing crepant birational transformations
and a crepant finite Galois morphisms.
In this example, we show that in general, it is necessary to perform both kinds of geometric modifications.
Let $T$ be $(\pp^1)^3$ 
and $\Delta$ be the torus invariant divisor.
Consider the involution
$i\colon T\rightarrow T$ given by
\[
i([x_1:y_1],[x_2:y_2],[x_3:y_3])=
([y_1:x_1],[y_2:x_2],[y_3:x_3]).
\]
Let $X:=T/i$
and $B$ be the quotient of $\Delta$ on $X$.
Let $Z_1$ (resp. $Z_2$, and $Z_3$) be the curves 
defined by the vanishing of $x_3$ and $x_2y_1-x_1y_2$ (resp. $x_2$ and $x_1y_3-x_3y_1$, and $x_1$ and $x_3y_2-x_3y_2$).
Let $C_1,C_2$ and $C_3$ be the images of $Z_1,Z_2$ and $Z_3$ in $X$.
Let $W\rightarrow X$ be the blow-up of $C_1,C_2$ and $C_3$.
Let $(W,\Gamma)$ be the log pull-back of 
$(X,\Delta)$ to $W$.
Then, the following conditions are satisfied:
\begin{enumerate}
    \item[(i)] we have $\mathcal{D}(W,\Gamma)\simeq \pp^2_\rr$, 
    \item[(ii)] we have $\pi_1^{\rm reg}(W,\Gamma)\simeq \zz_2$, and 
    \item[(iii)] we have $c_{\rm alt}(W,\Gamma)=0$.
\end{enumerate}
By the first condition, we conclude that 
$c_{\rm bir}(W,\Gamma)=3$.
Thus, the alteration complexity of $(W,\Gamma)$ cannot be obtained by only performing birational transformations.
On the other hand, since $\pi_1^{\rm reg}(W,\Gamma)\simeq \zz_2$, there is just one possible cover $W_0\rightarrow W$ for the log pair.
Let $(W_0,\Gamma_0)$ be the log pull-back of $(W,\Gamma)$ to $W_0$.
By construction, the log pair $(W_0,\Gamma_0)$ is obtained form $(T,B)$ by blowing up the $i$-orbit of $Z_1,Z_2$ and $Z_3$.
Thus $(W_0,\Gamma_0)$ is not a log Calabi--Yau toric pair.
The previous implies that 
the alteration complexity
cannot be computed
by only taking crepant finite covers.
We have a commutative diagram
\[
\xymatrix{
(W_0,\Gamma_0)\ar[d]_-{/\zz_2} \ar@{-->}[r] & (T,\Delta)\ar[d]^-{/\zz_2} \\ 
(W,\Gamma)\ar@{-->}[r] & (X,B).
}
\]
The diagram shows that the alteration complexity can be computed by either performing 
a crepant birational modification
and a crepant finite Galois cover or vice-versa.
}
\end{example}

\begin{example}\label{ex:comparison-invariants}
{\em 
In this example, we show that there is a log Calabi--Yau pair $(X,B)$ for which the strict inequality $c(X,B)>c_{\rm bir}(X,B)$ holds. 

Consider the log Calabi--Yau pair $(\pp^2,H_0+H_1+H_2)$ where the $H_i$'s are the coordinate lines.
Let $(Y,B_0+B_1+B_2)$ be the log Calabi--Yau pair obtained from $(\pp^2,H_0+H_1+H_2)$ by blowing up two points in $H_0\setminus H_1\cup H_2$, 
three points in $H_1\setminus H_0\cup H_2$
and four points in $H_2\setminus H_0\cup H_1$.
Then, by construction, we have $B_0^2=-1,B_1^2=-2$, and $B_2^2=-3$.
Let $(W,C_2)$ be the log Calabi--Yau surface
obtained by contracting $B_0$ and then the strict transform of $B_1$.
Then, $W$ is a smooth projective surface of Picard rank $7$ and $C_2$ is a nodal rational curve with $C_2^2=-1$.
Hence, we may contract $C_2$ and obtain $W\rightarrow X$.
Note that $X$ has a single lc singularity and $K_X\sim_\qq 0$. We set $B=0$.
By construction, we have $c_{\rm bir}(X,B)=0$
and $c(X,B)=3$.
}
\end{example}

\begin{example}
{\em  
In this example, we show that there exists a log Calabi--Yau pair $(X,B)$ for which the inequality $c_{\rm bir}(X,B)>\bar{c}_{\rm bir}(X,B)$ holds.

Let $X\simeq (\pp^1)^n$ and $H\in |-mK_X|$ be a general hypersurface. Then, the pair $(X,H/m)$ is a terminal log Calabi--Yau pair.
Thus, we have $c_{\rm bir}(X,H/m)=2n-\frac{1}{m}$.
On the other hand, we have $\bar{c}_{\rm bir}(X,H/m)=n$ by taking the trivial decomposition of the divisor $H/m$.
} 
\end{example} 

\begin{example} 
{\em 
In this example, we show that there is a log Calabi--Yau pair $(X,B)$ for which 
the strict inequality $\bar{c}_{\rm bir}(X,B)>\bar{c}_{\rm alt}(X,B)$ holds.

Let $(X,B)$ be the log Calabi--Yau constructed in Example~\ref{theorem:alt-comp-computation}.
We know that $\bar{c}_{\rm alt}(X,B)=0$.
On the other hand,
$\bar{c}_{\rm bir}(X,B)=3$, otherwise 
the dual complex $\mathcal{D}(X,B)$
is PL-isomorphic to a $2$-dimensional sphere.
}
\end{example}

We finish this section by proposing some questions that arise after the results of this article.

\begin{question}\label{quest:what-value}
{\em 
Let $(X,B)$ be a log Calabi--Yau pair of dimension $n$, 
coregularity zero,
and index one.
What are the possible values of $c_{\rm alt}(X,B)$?
}
\end{question}

In the setting of the question, we know that $c_{\rm alt}(X,B)\in \{0,1,\dots,n\}$.
However, in the case of surfaces, the values $1$ and $2$ do not happen. In the case of threefolds, the value $2$ does not happen.
It is not clear that the alteration complexity behaves well for the product of log Calabi--Yau pairs. The following is the fundamental question.

\begin{question}\label{quest:max-alt-comp}
{\em 
Are there $n$-dimensional log Calabi--Yau pairs of coregularity zero, index one, and alteration complexity equal to $n$? 
}
\end{question}

In~\cite[Theorem 1.18]{MM24}, for each odd dimension $n$, the authors construct a log Calabi--Yau pair $(X_n,B_n)$ of dimension $n$, coregularity zero, index one, and 
birational complexity equal to $n$.
All these examples have alteration complexity zero.
The case of the alteration complexity is more complicated. 
In order to understand the alteration complexity of $(X,B)$, we also need to understand the fundamental group $\pi_1(X',B')$ for every crepant birational model $(X',B')$ of $(X,B)$.
This task is hard in practice
unless $X$ is birationally superrigid.
Thus, it is natural to expect that a positive answer to Question~\ref{quest:max-alt-comp} could be found by studying $1$-complements on smooth hypersurfaces of degree $n$ in $\pp^n$. However, it is not yet clear which hypersurfaces of degree $n$ in $\pp^n$ admit $1$-complements of coregularity zero.

Finally, Theorem~\ref{theorem:cbir-conic} states that a log Calabi--Yau pair $(X,B)$ of dimension $n$ and birational complexity $c$ 
has a birational model that admits a tower of Mori fiber spaces of which at least $n-c$ are conic fibrations. In the case of smooth Fano hypersurfaces, we can make sure that the conic fibrations are precisely the last $n-c$ Mori fiber spaces. This seems to be the optimal situation.
Indeed, in this case, the base of the fibrations is a Bott $\qq$-tower (see Theorem~\ref{theorem:conic-complexity-zero}).
It is not clear whether this is the case for other log Calabi--Yau pairs.

\begin{question}
{\em 
Let $(X,B)$ be a log Calabi--Yau pair of dimension $n$ and birational complexity $c$.
What are the possible positions of the $n-c$ conic fibrations in birational towers of Mori fiber spaces for $(X,B)$?
}
\end{question}

\bibliographystyle{habbvr}
\bibliography{references}

\end{document}